\documentclass[runningheads]{llncs}
\usepackage{graphicx}
\usepackage{amsthm}

\newtheorem{thm}{Theorem}
\newtheorem{defn}[thm]{Definition}
\newtheorem{lem}[thm]{Lemma}

\newtheorem{prop}[thm]{Proposition}

\usepackage{lscape}
\usepackage{pdflscape}

\usepackage{stmaryrd}

\usepackage{cmll}

\usepackage{comment}

\usepackage{amscd}
\usepackage{amssymb,amsmath}
\usepackage{mathptmx}
\usepackage{mathrsfs}
\usepackage{color}
\usepackage{xspace}
\usepackage{bussproofs}
\EnableBpAbbreviations

\usepackage{tikz}

  \usepackage{txfonts}

\usepackage{bpextra}

\usepackage{enumerate}

\usepackage{centernot}

\usepackage{mathtools}

\usepackage{hyperref}

\usepackage{cleveref}

\usepackage{relsize}

\usepackage{caption}

\usepackage{multirow}

\usepackage{multicol}

\usepackage{array}
\newcolumntype{C}[1]{>{\centering\arraybackslash}p{#1}}
\newcolumntype{L}[1]{>{\arraybackslash}p{#1}}

\def\fCenter{{\mbox{$\ \vdash\ $}}}

\newcommand{\scs}{\scriptsize}
\newcommand{\fns}{\footnotesize}

\def\mc{\multicolumn}

\newcommand{\mb}{\mathbb}
\newcommand{\mcl}{\mathcal}

\newcommand{\xneg}{\ensuremath{\neg}\xspace}
\newcommand{\xtop}{\ensuremath{\top}\xspace}
\newcommand{\xbot}{\ensuremath{\bot}\xspace}
\newcommand{\xand}{\ensuremath{\wedge}\xspace}
\newcommand{\xor}{\ensuremath{\vee}\xspace}
\newcommand{\xrarr}{\ensuremath{\rightarrow}\xspace}

\newcommand{\XNEG}{\ensuremath{\:\tilde{\neg}}\xspace}
\newcommand{\XTOP}{\hat{\top}}
\newcommand{\XBOT}{\ensuremath{\check{\bot}}\xspace}
\newcommand{\XAND}{\ensuremath{\:\hat{\wedge}\:}\xspace}
\newcommand{\XOR}{\ensuremath{\:\check{\vee}\:}\xspace}

\newcommand{\dneg}{\ensuremath{{\sim}}\xspace}
\newcommand{\rdneg}{\ensuremath{\neg}\xspace}
\newcommand{\dtop}{\ensuremath{1}\xspace}
\newcommand{\dbot}{\ensuremath{0}\xspace}
\newcommand{\dand}{\ensuremath{\cap}\xspace}
\newcommand{\dor}{\ensuremath{\cup}\xspace}

\newcommand{\DNEG}{\ensuremath{\:\tilde{\sim}}\xspace}

\newcommand{\DTOP}{\ensuremath{\hat{1}}\xspace}
\newcommand{\DBOT}{\ensuremath{\check{0}}\xspace}
\newcommand{\DAND}{\ensuremath{\:\hat{\cap}\:}\xspace}
\newcommand{\DOR}{\ensuremath{\:\check{\cup}\:}\xspace}

\newcommand{\XDIANU}{\ensuremath{\langle\hat{\nu}\rangle}\xspace}
\newcommand{\xdianu}{\ensuremath{\langle\nu\rangle}\xspace}
\newcommand{\DBOXUN}{\ensuremath{[\check{\rotatebox[origin=c]{180}{$\nu$}}]}\xspace}
\newcommand{\dboxun}{\ensuremath{[\rotatebox[origin=c]{180}{$\nu$}]}\xspace}

\newcommand{\DDIANI}{\ensuremath{\langle\hat{\ni}\rangle}\xspace}
\newcommand{\ddiani}{\ensuremath{\langle\ni\rangle}\xspace}

\newcommand{\XBOXNIN}{\ensuremath{[\check{\not\in}]}\xspace}
\newcommand{\xboxnin}{\ensuremath{[\not\in]}\xspace}

\newcommand{\XDIAIN}{\ensuremath{\langle\hat{\in}\rangle}\xspace}
\newcommand{\xdiain}{\ensuremath{\langle\in\rangle}\xspace}
\newcommand{\DBOXNI}{\ensuremath{[\check\ni]}\xspace}
\newcommand{\dboxni}{\ensuremath{[\ni]}\xspace}
\newcommand{\DDIANNI}{\ensuremath{\langle\hat{\not\ni}\rangle}\xspace}
\newcommand{\ddianni}{\ensuremath{\langle\not\ni\rangle}\xspace}

\newcommand{\DDIAUNC}{\ensuremath{\langle\hat{\rotatebox[origin=c]{180}{$\nu$}}^c\rangle}\xspace}
\newcommand{\ddiaunc}{\ensuremath{\langle\rotatebox[origin=c]{180}{$\nu$}^c\rangle}\xspace}

\newcommand{\XBOXNUC}{\ensuremath{[\check{\nu^c}]}\xspace}
\newcommand{\xboxnuc}{\ensuremath{[\nu^c]}\xspace}

\newcommand{\MTRA}{\ensuremath{\check{\,\vartriangleright\,}}\xspace}
\newcommand{\mtra}{\ensuremath{\vartriangleright}\xspace}

\newcommand{\MTBRA}{\ensuremath{\check{\,\blacktriangleright\,}}\xspace}
\newcommand{\mtbra}{\ensuremath{\blacktriangleright}\xspace}

\newcommand{\MTAND}{\ensuremath{\,\hat{\blacktriangle}\,}\xspace}

\newcommand{\mtAND}{\ensuremath{\blacktriangle}\xspace}

\newcommand{\DRHDNNI}{\ensuremath{[\check{\not\ni}\rangle}\xspace}
\newcommand{\drhdnni}{\ensuremath{[\not\ni\rangle}\xspace}

\newcommand{\XRHDNIN}{\ensuremath{[\check{\not\in}\rangle}\xspace}
\newcommand{\xrhdnin}{\ensuremath{[\not\in\rangle}\xspace}

\newcommand{\abla}{\nabla}
\newcommand{\mbf}{\mathbf}

\usetikzlibrary{arrows}
\usetikzlibrary{matrix}
\usetikzlibrary{patterns}
\usetikzlibrary{shapes}
\usepackage{authblk}

\makeindex

\begin{document}
\title{Non-normal logics: semantic analysis and proof theory
}
%
\author{Jinsheng Chen\inst{2} 
\and
Giuseppe Greco\inst{1}\thanks{The research of the second author is supported by a NWO grant under the scope of the project ``A composition calculus for vector-based semantic modelling with a localization for Dutch'' (360-89-070).}
\and
Alessandra Palmigiano \inst{2,3}\thanks{The research of the third and fourth author is supported by the NWO Vidi grant 016.138.314, the NWO Aspasia grant 015.008.054, and a Delft Technology Fellowship awarded to the third author}
\and
Apostolos Tzimoulis\inst{2}
}
\authorrunning{Chen,  Greco, Palmigiano, Tzimoulis}
%
\institute{
University of Utrecht, the Netherlands
\and 
Vrije Universiteit Amsterdam, the Netherlands
\and 
Department of Mathematics and Applied Mathematics, University of Johannesburg, South Africa
}
\maketitle              
\begin{abstract}
We introduce proper display calculi for  basic monotonic modal logic, the conditional logic CK and a number of their axiomatic extensions. These calculi are sound, complete, conservative and enjoy cut elimination and subformula property. Our proposal applies the multi-type methodology in the design of proper display calculi, starting from a semantic analysis which motivates syntactic translations from single-type non-normal modal logics to multi-type normal poly-modal logics.
\keywords{ Monotonic modal logic \and Conditional logic \and  Proper display calculi.}
\end{abstract}
%
%
\section{Introduction}
By {\em non-normal logics} we understand in this paper those propositional logics algebraically captured by varieties of {\em Boolean algebra expansions}, i.e.~algebras $\mathbb{A} = (\mathbb{B}, \mathcal{F}^\mathbb{A},$ $\mathcal{G}^\mathbb{A})$ such that $\mathbb{B}$ is a Boolean algebra, and $\mathcal{F}^\mathbb{A}$ and $\mathcal{G}^\mathbb{A}$ are finite, possibly empty families of operations on $\mathbb{B}$ in which the requirement is dropped that each operation in $\mathcal{F}^\mathbb{A}$ be finitely join-preserving or meet-reversing  in each coordinate and  each operation in $\mathcal{G}^\mathbb{A}$ be finitely meet-preserving or join-reversing  in each coordinate. Very well-known examples of non-normal logics are {\em monotonic modal logic} \cite{chellas1980modal} and {\em conditional logic} \cite{nute2012topics,chellas1975basic}, which have been intensely investigated, since they capture key aspects of agents' reasoning, such as the epistemic \cite{van2011dynamic}, strategic \cite{pauly2003game,pauly2002modal}, and hypothetical \cite{gabbay2000conditional,lewis2013counterfactuals}.

Non-normal logics have been extensively investigated both with model-theoretic tools \cite{hansen2003monotonic} and with proof-theoretic tools \cite{negri2017proof,Olivetti2007ASC,GilMaf15}. Specific to proof theory, the main challenge is to endow non-normal logics with analytic calculi which can be modularly expanded with additional rules so as to uniformly capture wide classes of axiomatic extensions of the basic frameworks, while preserving key properties such as cut elimination. In this paper, which builds and expands on \cite{chen2019non}, we propose a method to achieve this goal. We will illustrate this method for the two specific signatures of monotonic modal logic and conditional logic. 

Our starting point is the observation, very well-known e.g.~from \cite{hansen2003monotonic}, that, under the interpretation of the modal connective of monotonic modal logic in neighbourhood frames $\mathbb{F} = (W, \nu)$, the monotonic `box' operation can be understood as the composition of a {\em normal} (i.e.~finitely join-preserving) semantic diamond $\xdianu$ and a {\em normal} (i.e.~finitely meet-preserving) semantic box $\dboxni$. The binary relations $R_\nu$ and $R_\ni$ corresponding to these two normal operators are not defined on one and the same domain, but span over two domains, namely $R_\nu\subseteq W\times \mathcal{P}(W)$ is s.t.~$w R_\nu X$ iff $X\in \nu(w)$ and $R_\ni\subseteq \mathcal{P}(W)\times W$ is s.t.~$X R_\ni w$ iff $w\in X$ (cf.~\cite[Definition 5.7]{hansen2003monotonic}, see also \cite{kracht1999normal,gasquet1996classical}). 
We refine and expand these observations so as to: (a) introduce a semantic environment of two-sorted Kripke frames (cf.~Definition \ref{def:2sorted Kripke frame}) and their heterogeneous algebras  (cf.~Definition \ref{def:heterogeneous algebras}); (b) outline a network of discrete dualities and correspondences among these semantic structures and the algebras and frames for monotone modal logic and conditional logic (cf.~Propositions \ref{prop:dduality single type},  \ref{prop:dduality multi-type}, \ref{prop:alg characterization of single into multi}, \ref{prop:adjunction-frames}); (c) based on these semantic relationships, introduce multi-type {\em normal} logics into which the original non-normal logics can be embedded via suitable translations (cf.~Section \ref{sec:embedding}) following a methodology which was successful in several other cases \cite{8134377,frittella2016multi,frittella2016proof,Inquisitive,GLMP18,KriPalNac18,LogicsForRoughConceptAnalysis,bilattice,GrecoPalmigianoLatticeLogic,apostolos2018}; (d) retrieve well-known dual characterization results  for axiomatic extensions of monotone modal logic and conditional logics as instances of general algorithmic correspondence theory for normal  (multi-type)  LE-logics applied to the translated axioms (cf.~Section \ref{sec:ALBA runs}); (e) extract analytic structural rules from the computations of the first-order correspondents of the translated axioms, so that, again by general results on  proper display calculi \cite{greco2016unified} (which, as discussed in \cite{bilkova2018logic}, can be applied also to multi-type logical frameworks) the resulting calculi are sound, complete, conservative and enjoy cut elimination and subformula property.

\paragraph{Structure of the paper} In Section \ref{sec: preliminaries}, we collect well-known definitions and facts about monotone modal logic and conditional logic, their algebraic and state-based semantics, and the connection between the two. In Section \ref{sec: semantic analysis}, we introduce the multi-type environment (both in the form of heterogeneous algebras and of multi-type Kripke frames) which will provide the semantic justification for the two-sorted modal logics introduced in Section \ref{sec:embedding}, as well as for the syntactic translation of the original languages of monotone modal logic and conditional logic into suitable (multi-type) normal modal languages. In Section \ref{section:correspondence}, the theory of unified correspondence is applied to this two-sorted  environment to establish a Sahlqvist-type correspondence framework for  monotone modal logic and conditional logic which encompasses and extends the extant correspondence-theoretic results for these logics. In Section \ref{sec:calculi}, proper (multi-type) display calculi are introduced for the basic two sorted normal modal languages and for some of their best known extensions. The main properties of these calculi are discussed in Section  \ref{sec:properties}. Conclusions and further directions are discussed in Section \ref{sec: Conclusions}. 
\section{Preliminaries}
\label{sec: preliminaries}
\paragraph{Notation.}
\label{ssec:notation}
Throughout the paper, 
the superscript $(\cdot)^c$ denotes the relative complement of the subset of a given set. When the given set is a singleton $\{x\}$, we will write $x^c$ instead of $\{x\}^c$.  
For any binary relation $R\subseteq S\times T$, let $R^{-1}\subseteq T\times S$ be the \emph{converse relation} of $R$, i.e.~$tR^{-1}s$ iff $sRt$. 
 For any $S'\subseteq S$ and $T'\subseteq T$, we  let $R[S']: = \{t\in T\mid (s, t)\in R \mbox{ for some } s\in S'\}$ and $R^{-1}[T']: = \{s\in S\mid (s, t)\in R \mbox{ for some } t\in T'\}$. As usual, we write $R[s]$ and $R^{-1}[t]$ in place of $R[\{s\}]$ and $R^{-1}[\{t\}]$, respectively. 
For any ternary relation $R\subseteq S\times T\times U$ 
and subsets $S'\subseteq S$, $T'\subseteq T$, and $U'\subseteq U$, we also let
{\small{
\begin{itemize}
\item $R^{(0)}[T', U'] =\{s\in S\mid\  \exists t\exists u(R(s, t, u)\ \&\  t\in T' \ \&\ u\in U')\},$
\item $R^{(1)}[S', U'] =\{t\in T\mid\  \exists s\exists u(R(s, t, u)\ \&\  s\in S' \ \&\ u\in U')\},$
\item $R^{(2)}[S', T'] =\{u\in U\mid\  \exists s\exists t(R(s, t, u)\ \&\  s\in S' \ \&\ t\in T')\}.$
\end{itemize}
}}
Any binary relation $R\subseteq S\times T$ gives rise to the
  {\em modal operators} $\langle R\rangle, [R], [R\rangle, \langle R] :\mathcal{P}(T)\to \mathcal{P}(S)$ s.t.~for any $T'\subseteq T$ 
  {\small{
  \begin{itemize}
\item $\langle R\rangle T' : = R^{-1}[T'] = \{s\in S\mid \exists t( s R t \ \&\  t\in T')\}$;
\item  $ [R]T': = (R^{-1}[{T'}^c])^c = \{s\in S\mid \forall t( s R t \ \Rightarrow\  t\in T')\}$; 
\item  $ [R\rangle T': = (R^{-1}[T'])^c = \{s\in S\mid \forall t( s R t \ \Rightarrow\  t\notin T')\}$;
\item $\langle R] T' : = R^{-1}[{T'}^c] = \{s\in S\mid \exists t( s R t \ \&\  t\notin T')\}$.
\end{itemize}
}}
\noindent By construction, these modal operators are normal.  In particular, $\langle R\rangle$ is completely join-preserving, $[R]$ is completely meet-preserving, $[R\rangle$ is completely join-reversing and $\langle R]$ is completely meet-reversing. Hence, their adjoint maps exist and coincide with $[R^{-1}]\langle R^{-1}\rangle,  [R^{-1}\rangle, \langle R^{-1}]: \mathcal{P}(S)\to \mathcal{P}(T)$, respectively. That is, for any $T'\subseteq T$ and $S'\subseteq S$,
\begin{center}
\begin{tabular}{r c l}
$\langle R\rangle T'\subseteq S'\quad$& iff & $\quad T' \subseteq [R^{-1}] S', $\\
$  S' \subseteq [R] T'\quad$& iff &  $\quad\langle R^{-1}\rangle S' \subseteq T'$,\\ 
$S' \subseteq [R\rangle T'\quad$& iff & $\quad T' \subseteq [R^{-1}\rangle S'$\\
$\langle R] T' \subseteq S'\quad$& iff & $\quad \langle R^{-1}] S' \subseteq T'.$\\
\end{tabular}
\end{center}

Any ternary relation $R\subseteq S\times T\times U$ gives rise to binary  modal operators \[\mtra_R: \mathcal{P}(T)\times \mathcal{P}(U)\to \mathcal{P}(S) \quad \mtAND_R: \mathcal{P}(T)\times \mathcal{P}(S)\to \mathcal{P}(U)\quad \mtbra_R: \mathcal{P}(S)\times \mathcal{P}(U)\to \mathcal{P}(T)\]  s.t.~for any $S'\subseteq S$, $T'\subseteq T$, and $U'\subseteq U$, 
  {\small{
  \begin{itemize}
\item $T' \mtra_R U': =  (R^{(0)}[T', {U'}^c])^c =\{s\in S\mid\  \forall t\forall u(R(s,t,u)\ \&\  t\in T' \Rightarrow u\in U')\}$;
\item $T' \mtAND_R S': =  R^{(2)}[T', S'] =\{u\in U\mid\  \exists t\exists s(R(s,t,u)\ \&\  t\in T' \ \&\  s\in S')\}$;
\item $S' \mtbra_R U': =  (R^{(1)}[S', {U'}^c])^c =\{t\in T\mid\  \forall s\forall u(R(s,t,u)\ \&\  s\in S' \Rightarrow u\in U')\}$.
\end{itemize}
}}
The stipulations above guarantee that these modal operators are normal.  In particular,   $\mtra_R$ and $\mtbra_R$ are completely join-reversing in their first coordinate and completely meet-preserving in their second coordinate, and $\mtAND_R$ is completely join-preserving in both coordinates. These three maps are residual to each other, i.e.~for any $S'\subseteq S$, $T'\subseteq T$, and $U'\subseteq U$,
\[S'\subseteq T' \mtra_R U'\quad\text{ iff }\quad T' \mtAND_R S'\subseteq U'\quad\text{ iff }\quad T'\subseteq S' \mtbra_R U'.\]

\subsection{Basic monotonic modal logic and conditional logic}
\label{ssec:prelim}

\paragraph{Syntax.} For a countable set of propositional variables $\mathsf{Prop}$, the languages $\mathcal{L}_{\abla}$ and $\mathcal{L}_{>}$ of monotonic modal logic and conditional logic over $\mathsf{Prop}$ are defined as follows:
\[\mathcal{L}_{\abla}\ni \varphi ::= p\mid \neg \varphi \mid  \varphi \land \varphi \mid \abla \varphi\quad\quad\quad\quad\mathcal{L}_{>}\ni \varphi ::= p\mid \neg \varphi \mid  \varphi \land \varphi \mid  \varphi> \varphi.\]
The connectives $\top, \land,\lor, \to$ and $\leftrightarrow$ are defined as usual. 
The {\em basic monotone modal logic} $\mathbf{L}_{\abla}$ (resp.~{\em basic conditional logic} $\mathbf{L}_{>}$) is a set   of $\mathcal{L}_{\abla}$-formulas  (resp.~$\mathcal{L}_{>}$-formulas) containing the  axioms of classical propositional logic and closed under modus ponens, uniform substitution and the following rule(s) $M$ (resp. $RCEA$ and $RCK_n$ for all $n \ge 0$):
{\footnotesize{
\begin{center}
\AXC{$\varphi \rightarrow \psi$}
\LeftLabel{\tiny{M}}
\UIC{$\abla \varphi \to \abla \psi$}
\DP
\ \ \ 
\AXC{$\varphi \leftrightarrow \psi$}
\LeftLabel{\tiny{RCEA}}
\UIC{$(\varphi > \chi) \leftrightarrow (\psi > \chi)$}
\DP
\ \ \ 
\AXC{$\varphi_1 \wedge{\! \ldots \!}\wedge \varphi_n \rightarrow \psi$}
\LeftLabel{\tiny{RCK$_n$}}
\UIC{$(\chi > \varphi_1) \wedge{\! \ldots \!}\wedge (\chi > \varphi_n) \rightarrow (\chi > \psi)$}
\DP
\end{center}
}}

\paragraph{Algebraic semantics.} A {\em monotone Boolean algebra expansion}, abbreviated as {\em m-algebra} (resp.~{\em conditional algebra}, abbreviated as {\em c-algebra}) is a pair $\mathbb{A} = (\mathbb{B}, \abla^{\mathbb{A}})$ (resp.~$\mathbb{A} = (\mathbb{B}, >^{\mathbb{A}})$)  s.t.~$\mathbb{B}$ is a Boolean algebra and  $\abla^{\mathbb{A}}$ is a unary monotone operation on $\mathbb{B}$ (resp.~$>^{\mathbb{A}}$ is a binary operation on $\mathbb{B}$ which is finitely meet-preserving in its second coordinate). 
Such an  m-algebra (resp.~c-algebra)  is \emph{perfect} if $\mathbb{B}$ is a complete and atomic Boolean algebra  (and, in the c-algebra case,  $>^\mathbb{A}$  is completely meet-preserving in its second coordinate). Hence, the underlying Boolean algebra of any perfect m-algebra (resp.~c-algebra)  can be identified with the powerset algebra $\mathcal{P}(W)$ for some set $W$.

Interpretation of formulas in algebras under assignments $h:\mathcal{L}_{\abla}\to \mathbb{A}$ (resp.~$h:\mathcal{L}_{>}\to \mathbb{A}$)  and validity of formulas in algebras (in symbols: $\mathbb{A}\models\varphi$) are defined as usual. By a routine Lindenbaum-Tarski construction one can show that $\mathbf{L}_{\abla}$ (resp.~$\mathbf{L}_{>}$) is sound and complete w.r.t.~the class of m-algebras $V_m$ (resp.~c-algebras $V_c$). 

\paragraph{Canonical extensions.} 

 The {\em canonical extension} of an m-algebra (resp.~c-algebra)  $\mathbb{A}$ is  $\mathbb{A}^\delta: = (\mathbb{B}^\delta, \abla^{\sigma})$ (resp.~$\mathbb{A}^\delta: = (\mathbb{B}^\delta, >^{\pi})$), where $\mathbb{B}^\delta\cong \mathcal{P}(Ult(\mathbb{B}))$, with $Ult(\mathbb{B})$ denoting the set of the ultrafilters of $\mathbb{B}$, is the canonical extension of $\mathbb{B}$ \cite{jonsson1951boolean}, and $\abla^{\sigma}$ (resp.~$>^{\pi}$) is the $\sigma$-extension of $\abla^{\mathbb{A}}$ (resp.~the $\pi$-extension of $>^{\mathbb{A}}$). Let us recall that for all $u, u_1, u_2\in \mathbb{B}^\delta$,
 \[\abla^{\sigma} u: = \bigvee\{\bigwedge \{\abla a\mid a\in \mathbb{B} \text{ and } k\leq a\} \mid k\in K(\mathbb{B}^\delta) \text{ and } k\leq u\},\]
  \[u_1>^{\pi} u_2: = \bigwedge\{\bigvee \{ a_1> a_2\mid a_i\in \mathbb{B} \text{ and } o_i\leq a_i\leq k_i\} \mid k_i\in K(\mathbb{B}^\delta),  o_i\in O(\mathbb{B}^\delta) \text{ and } k_i\leq u_i\leq o_i\},\]
where $K(\mathbb{B}^\delta)$ and $ O(\mathbb{B}^\delta) $ respectively denote the join-closure and the meet-closure of $\mathbb{B}$ in $\mathbb{B}^\delta$ under the canonical embedding, mapping each $a\in \mathbb{B}$ to $\{U\in Ult(\mathbb{B})\mid a\in U\}$.
 
By definition and general results on canonical extensions of maps (cf.~\cite{gehrke2004bounded}), the canonical extension of an m-algebra (resp.~c-algebra) as above is a  perfect m-algebra (resp.~c-algebra).

\paragraph{Frames and models.}
A {\em neighbourhood frame}, abbreviated as {\em n-frame}  (resp.~{\em conditional frame}, abbreviated as {\em c-frame}) is a pair $\mathbb{F}=(W,\nu)$ (resp.~$\mathbb{F}=(W,f)$) s.t.~$W$ is a non-empty set and $\nu:W\to \mathcal{P}(\mathcal{P}(W))$ is a {\em neighbourhood function} ($f: W\times\mathcal{P}(W)\to \mathcal{P}(W)$ is a {\em selection function}). 
In the remainder of the paper, even if it is not explicitly indicated, we will assume that n-frames  are {\em monotone}, i.e.~s.t.~for every $w\in W$, if $X\in \nu(w)$ and $X\subseteq Y$, then $Y\in \nu(w)$.   For any n-frame  (resp.~c-frame) $\mathbb{F}$, the {\em complex algebra} of $\mathbb{F}$ is $\mathbb{F}^\ast: = (\mathcal{P}(W), \abla^{\mathbb{F}^\ast})$ (resp.~$\mathbb{F}^\ast: = (\mathcal{P}(W),  >^{\mathbb{F}^\ast})$) s.t.~for all $X, Y\in \mathcal{P}(W)$,
\begin{center}
$\abla^{\mathbb{F}^\ast} X: = \{w\mid X\in \nu(w)  \} \quad\quad\quad \quad X >^{\mathbb{F}^\ast} Y: =  \{w\mid  f(w, X)\subseteq Y\}. $\end{center}
\begin{prop}
If $\mathbb{F}$ is an n-frame (resp.~a c-frame), then $\mathbb{F}^\ast$ is a perfect m-algebra (resp.~c-algebra).
\end{prop}

\begin{proof}
Let $\mathbb{F}=(W,\nu)$ be an n-frame. Recall that, by definition, $\nu(w)$ is an upward-closed collection of subsets of $W$. To show that $\mathbb{F}^\ast$ is a perfect m-algebra, it is enough to show that $\abla^{\mathbb{F}^\ast}$ is monotone. Let  $w \in W$ and $X \subseteq Y\subseteq W$.  Since $\nu(w)$ is upward-closed,  $X \in \nu(w)$ implies that $Y \in \nu(w)$. Hence, $\abla^{\mathbb{F}^\ast} X= \{w\mid X\in \nu(w)  \}\subseteq \{w \mid Y \in \nu(w)\} = \abla^{\mathbb{F}^\ast} Y$.

Let $\mathbb{F}=(W,f)$ be a c-frame.  To show that $\mathbb{F}^\ast$ is a perfect c-algebra, it is enough to show that $>^{\mathbb{F}^\ast}$ is completely meet-preserving in its second coordinate. For any $X\subseteq W$,
\[X >^{\mathbb{F}^\ast} \top^{\mathbb{F}^\ast} = X >^{\mathbb{F}^\ast} W = \{w\mid  f(w, X)\subseteq W\} = W = \top^{\mathbb{F}^\ast},\]
and for any $\mathcal{X}\subseteq \mathcal{P}(W)$,
\begin{align*}
X >^{\mathbb{F}^\ast} \bigcap \mathcal{X}& =  \{w\in W\mid  f(w, X)\subseteq \bigcap\mathcal{X}\}\\
 & =\{w\mid  f(w, X)\subseteq Y\}\cap\{w\in W\mid  f(w, X)\subseteq Y \text{ for any } Y\in \mathcal{X}\}\\
 &= \bigcap\{(X >^{\mathbb{F}^\ast} Y)\mid Y\in \mathcal{X}\}.
 \end{align*}
\end{proof}

{\em Models} are pairs $\mathbb{M} = (\mathbb{F}, V)$ such that $\mathbb{F}$ is a frame and  $V:\mathcal{L} \to \mathbb{F}^\ast$ is a homomorphism of the appropriate type. Hence, the truth of formulas at states in models is defined as $\mb{M},w \Vdash \varphi$ iff $w\in V(\varphi)$, and unravelling this stipulation   for $\abla$- and $>$-formulas, we get:
\[
\mb{M},w \Vdash \abla \varphi \quad \text{iff}\quad  V(\varphi)\in \nu(w) \quad\quad\quad\quad \mb{M},w \Vdash \varphi> \psi \quad \text{iff}\quad  f(w, V(\varphi))\subseteq V(\psi).
\]


Local validity (notation: $\mb{F},w \Vdash \varphi$) is defined as local satisfaction for every valuation $V$. Global satisfaction (notation: $\mathbb{M}\Vdash\varphi$) and frame validity (notation: $\mathbb{F}\Vdash\varphi$) are defined in the usual way as local satisfaction/validity at every state. Thus, by definition, $\mathbb{F}\Vdash\varphi$ iff $\mathbb{F}^\ast\models \varphi$, from which the soundness of $\mathbf{L}_{\abla}$ (resp.~$\mathbf{L}_{>}$) w.r.t.~the corresponding class of frames  immediately follows from the algebraic soundness. Completeness  follows from algebraic completeness, by observing that (a) the canonical extension of any algebra refuting $\varphi$ will also refute $\varphi$;  (b) canonical extensions are perfect algebras; (c) perfect m-algebras (resp. c-algebras) can be associated with n-frames (resp. c-frames) as follows: for any $\mathbb{A} = (\mathcal{P}(W), \abla^{\mathbb{A}})$ (resp.~$\mathbb{A} = (\mathcal{P}(W), >^{\mathbb{A}})$) let $\mathbb{A}_\ast:=(W,\nu_{\abla^{\mathbb{A}}})$ (resp.~$\mathbb{A}_\ast:=(W,f_{>^{\mathbb{A}}})$) s.t.~for all $w\in W$ and $X\subseteq W$, 
\[\nu_{\abla^{\mathbb{A}}}(w): = \{X\subseteq W\mid w\in \abla^{\mathbb{A}} X\}\quad\quad\quad\quad f_{>^{\mathbb{A}}}(w, X): = \bigcap\{Y\subseteq W\mid w\in X>^{\mathbb{A}} Y\}.
\]
That $\mathbb{A}_\ast$ is a monotone n-frame can be proved as follows: if $X\in \nu_{\abla}(w)$ and $X\subseteq Y$, then the monotonicity of $\abla^{\mathbb{A}}$ implies that $\abla^{\mathbb{A}} X\subseteq \abla^{\mathbb{A}} Y$ and hence $Y\in \nu_{\abla^{\mathbb{A}}}(w)$, as required.

Let $\varphi\in \mathcal{L}_{\abla}$ (resp.~$\varphi\in \mathcal{L}_{>}$). It can be shown by a straightforward induction on $\varphi$ that $w \in V(\varphi)$ iff $(\mathbb{A}_\ast, V),w \Vdash \varphi$ for any perfect algebra $\mathbb{A}$ and assignment $V$. Then, $\mathbb{A}\models\varphi$ iff $\mathbb{A}_\ast\Vdash \varphi$. This completes the argument deriving the frame completeness of $\mathbf{L}_{\abla}$ (resp.~$\mathbf{L}_{>}$) from its algebraic completeness.

\begin{prop} 
\label{prop:dduality single type} If $\mathbb{A}$ is a perfect m-algebra (resp.~c-algebra)   and $\mathbb{F}$ is an  n-frame  (resp.~c-frame), then
 $(\mathbb{F}^\ast)_\ast\cong \mathbb{F}$ and  $(\mathbb{A}_\ast)^\ast\cong \mathbb{A}$.
\end{prop}

\begin{proof}
Let $\mathbb{F}=(W, \nu)$ be an n-frame. By definition, $(\mathbb{F}^\ast)_\ast=(W, \nu_{\nabla^{\mathbb{F}^\ast}})$, where, for every $w\in W$,
\begin{align*}
\nu_{\nabla^{\mathbb{F}^\ast}}(w) &= \{X\subseteq W\mid w \in \nabla^{\mathbb{F}^\ast} X\}\\
&= \{X\subseteq W\mid w \in \{u \mid X \in \nu(u)\}\}\\
&= \{X\subseteq W\mid X\in \nu(w)\}\\
&= \nu(w),
\end{align*}
which shows that  $(\mathbb{F}^\ast)_\ast= \mathbb{F}$, as required.
Let $\mathbb{F}=(W, f)$ be a c-frame. By definition, $(\mathbb{F}^\ast)_\ast=(w, f_{>^{\mathbb{F}^\ast}})$, where, for every $w\in W$ and $X\subseteq W$,
\begin{align*}
f_{>^{\mathbb{F}^\ast}}(w,X) &=\bigcap \{Y\subseteq W \mid w \in \ X>^{\mathbb{F}^\ast}Y \}\\
&=\bigcap \{Y\subseteq W \mid w \in \ \{u\in W \mid f(u,X)\subseteq Y \}\}\\
&=\bigcap \{Y\subseteq W \mid f(w,X)\subseteq Y \} \\
&=f(w,X),
\end{align*}
which shows that  $(\mathbb{F}^\ast)_\ast= \mathbb{F}$, as required.
Let $\mathbb{A}  = (\mathcal{P}(W), \abla^{\mathbb{A}})$ be a perfect m-algebra (up to isomorphism). Then $(\mathbb{A}_\ast)^\ast= (\mathcal{P}(W),\abla^{(\mathbb{A}_\ast)^\ast})$, where for every $X\subseteq W$,
\begin{align*}
\abla^{(\mathbb{A}_\ast)^\ast}X&=\{w \mid X\in \nu_{\abla^{\mathbb{A}}}(w)\}\\
& =\{w \mid  X\in \{Y\subseteq W\mid w \in \abla^\mathbb{A} Y\}  \}\\
& =\{w \mid  w \in \abla^\mathbb{A} X \}\\
&= \abla^\mathbb{A} X,
\end{align*}
which shows that $(\mathbb{A}_\ast)^\ast\cong \mathbb{A}$, as required. 
Let $\mathbb{A} = (\mathcal{P}(W), >^{\mathbb{A}})$ be a perfect c-algebra (up to isomorphism). Then $(\mathbb{A}_\ast)^\ast= (\mathcal{P}(W),>^{(\mathbb{A}_\ast)^\ast})$, where for all $X, Y\subseteq W$,
\begin{align*}
X>^{(\mathbb{A}_\ast)^\ast}Y&=\{w \mid f_{>^\mathbb{A}}(w,X) \subseteq Y\}\\
&=\{w \mid \bigcap\{Z\subseteq W \mid w \in \ X >^\mathbb{A}Z \}\subseteq Y\}\\
&= X>^\mathbb{A}Y.
\end{align*}
Let us show the last equality. If $w \in \ X>^\mathbb{A} Y$, then $Y \in \{Z\subseteq W \mid w \in \ X >^\mathbb{A}Z \}$, and hence $\bigcap \{Z\subseteq W \mid w \in \ X >^\mathbb{A}Z \} \subseteq Y$. Conversely, let $w\in W$ be s.t.~$ \bigcap\{Z\subseteq W \mid w \in \ X >^\mathbb{A}Z \}\subseteq Y$. Since $>^\mathbb{A}$ is completely meet-preserving in the second coordinate, this implies that 
\[w\in \bigcap \{X >^\mathbb{A} Z \mid Z\subseteq W \text{ and } w \in \ X >^\mathbb{A}Z \} = X >^\mathbb{A} \bigcap\{Z\subseteq W \mid w \in  X >^\mathbb{A}Z \} \subseteq X >^\mathbb{A} Y,\] 
as required. This completes the proof that $(\mathbb{A}_\ast)^\ast\cong \mathbb{A}$. \end{proof}

\paragraph{Axiomatic extensions.} A {\em monotone modal logic} (resp.~a {\em conditional logic}) is any  extension of $\mathbf{L}_{\abla}$ (resp.~$\mathbf{L}_{>}$)   with $\mathcal{L}_{\abla}$-axioms (resp.~$\mathcal{L}_{>}$-axioms). Below we collect correspondence results  for axioms that have cropped up in the literature \cite[Theorem 5.1]{hansen2003monotonic} \cite{Olivetti2007ASC}. 

\begin{thm}
\label{theor:correspondence-noAlba}
For every n-frame (resp.~c-frame) $\mathbb{F}$,

{\hspace{-0.7cm}
{\small{
\begin{tabular}{@{}rl@{}c@{}l}
N\; & $\mathbb{F}\Vdash\abla \top\quad$ &\!iff\!&   $\quad \mathbb{F}\models\forall w[W \in \nu (w)]$\\
P\; & $\mathbb{F}\Vdash\neg \abla \bot\quad$ & iff &  $\quad \mathbb{F}\models\forall w[\varnothing \not \in \nu(w)]$\\
C\; & $\mathbb{F}\Vdash \abla p \land \abla q \to \abla(p \land q)\quad$ & iff &  $\quad \mathbb{F}\models\forall w \forall X \forall Y  [(X \in \nu(w)\ \&\ Y \in \nu (w) )\Rightarrow X \cap Y \in \nu(w)]$\\
T\; & $\mathbb{F}\Vdash\abla p \to p \quad$ & iff &   $\quad \mathbb{F}\models\forall w  \forall X[X \in \nu(w) \Rightarrow w \in X]$\\
4\; & $\mathbb{F}\Vdash\abla \abla  p \to \abla p\quad$ & iff &   $\quad \mathbb{F}\models \forall w   \forall Y X[(X \in \nu (w)\ \&\ \forall x( x\in X\Rightarrow Y \in \nu(x))) \Rightarrow Y \in \nu(w)]$\\
4'\; & $\mathbb{F}\Vdash\abla p \to \abla \abla p\quad$ & iff &  $\quad \mathbb{F}\models\forall w \forall X[X \in \nu(w) \Rightarrow \{y \mid X \in \nu(y)\} \in \nu(w)]$\\
5\; & $\mathbb{F}\Vdash\neg \abla \neg p \to \abla \neg \abla \neg p \quad$ & iff &  $\quad \mathbb{F}\models \forall w \forall X [X \notin \nu(w) \Rightarrow  \{y \mid X \in \nu(y)\}^c \in \nu(w)]$\\
B\; & $\mathbb{F}\Vdash p \to \abla \neg \abla \neg p \quad$ & iff &   $\quad \mathbb{F}\models \forall w\forall X[w \in X \Rightarrow \{y \mid  X^c \in \nu (y)\}^c \in \nu(w)]$\\
D\; & $\mathbb{F}\Vdash \abla p \to \neg \abla \neg p\quad$ & iff &   $\quad \mathbb{F}\models \forall w \forall X[X\in \nu(w)\Rightarrow  X^c \not \in \nu(w)]$\\
CS\; &  $\mathbb{F}\Vdash (p\wedge q) \to (p > q)\quad$ & iff &  $\quad \mathbb{F}\models\forall x\forall Z[f(x,Z)\subseteq \{x\}]$\\
CEM\;&  $\mathbb{F}\Vdash (p > q) \vee (p> \neg q)\quad$ & iff &  $\quad \mathbb{F}\models\forall X \forall y[|f(y,X)|\leq 1]$\\

ID\; & $\mathbb{F}\Vdash p >  p \quad$ & iff &  $\quad\mathbb{F}\models\forall x\forall Z[f(x,Z)\subseteq Z].$\\

CN \; & $\mathbb{F}\Vdash (p > q) \lor (q > p ) \quad$ & iff &  $\quad\mathbb{F}\models\forall z  \forall x \forall y \forall X \forall Y  [  x \not \in f(z,X) \text{~or~} y \not \in f(z,Y)].$\\

T \; & $\mathbb{F}\Vdash(\bot >\neg p)\to p  \quad$ & iff &  $\quad\mathbb{F}\models\forall z \exists x[ x \in f(z,\varnothing) ].$\\

\end{tabular}
}}
}
\end{thm}
In the following section we will introduce a semantic environment which will make it possible to obtain all these correspondence results as instances of a suitable multi-type version of unified correspondence theory \cite{CoGhPa14,CoPa:non-dist}, 
and which will provide the motivation for the introduction of proper display calculi for the logics axiomatised by some of these axioms, namely, those the translation of which is analytic inductive (cf.~Section \ref{sec:embedding}). 

\section{Semantic analysis}
\label{sec: semantic analysis}
\subsection{Two-sorted Kripke frames and their discrete duality}
\label{ssec:2sorted Kripke frame}
Structures similar to those below are considered implicitly in \cite{hansen2003monotonic}, and explicitly in \cite{frittella2017dual}.
\begin{defn}\label{def:2sorted Kripke frame}
A {\em two-sorted n-frame} (resp.~{\em c-frame}) is a structure $\mathbb{K}: = (X, Y, R_\ni, R_{\not\ni},$ $R_\nu, R_{\nu^c})$ (resp.~$\mathbb{K}: = (X, Y, R_\ni, R_{\not\ni}, T_f)$) such that $X$ and $Y$ are nonempty sets, $R_\ni, R_{\not\ni}\subseteq Y\times X$ and $R_\nu, R_{\nu^c}\subseteq X\times Y$ and $T_f\subseteq X\times Y\times X$. 
Such an n-frame is {\em supported} if for every $D\subseteq X$,
\begin{equation}
\label{eq:supported n-frames}
R_{\nu}^{-1}[(R_{\ni}^{-1}[D^c])^c] = (R_{\nu^c}^{-1}[(R_{\not\ni}^{-1}[D])^c])^c.
\end{equation}
For any two-sorted n-frame  (resp.~c-frame) $\mathbb{K}$, the {\em complex algebra} of $\mathbb{K}$ is 

\begin{center}
$\mathbb{K}^+: = (\mathcal{P}(X), \mathcal{P}(Y), \dboxni^{\mathbb{K}^+}, \ddianni^{\mathbb{K}^+}, \xdianu^{\mathbb{K}^+}, \xboxnuc^{\mathbb{K}^+})$ (resp.~$\mathbb{K}^+: = (\mathcal{P}(X), \mathcal{P}(Y), \dboxni^{\mathbb{K}^+}, \drhdnni^{\mathbb{K}^+}, \mtra^{\mathbb{K}^+})$), s.t.
\end{center}

{\small{
\begin{center}
\begin{tabular}{r  cr c r}
$\xdianu^{\mathbb{K}^+}: \mathcal{P}(Y)\to \mathcal{P}(X)$ &$\quad$ & $\dboxni^{\mathbb{K}^+}: \mathcal{P}(X)\to \mathcal{P}(Y)$ &$\quad$ &$\ddianni^{\mathbb{K}^+}: \mathcal{P}(X)\to \mathcal{P}(Y)$\\ 
$U\mapsto R^{-1}_\nu[U]$ && $D\mapsto (R_\ni^{-1}[D^c])^c$ && $D\mapsto R_{\not\ni}^{-1}[D]$ \\
&&&\\
$\xboxnuc^{\mathbb{K}^+}: \mathcal{P}(Y)\to \mathcal{P}(X)$ &&$\drhdnni^{\mathbb{K}^+}: \mathcal{P}(X)\to \mathcal{P}(Y)$&$\quad$ & $\mtra^{\mathbb{K}^+}: \mathcal{P}(Y)\times \mathcal{P}(X)\to \mathcal{P}(X)$ \\
 $U\mapsto (R^{-1}_{\nu^c}[U^c])^c$ &&$D\mapsto (R_{\not\ni}^{-1}[D])^c$&& $(U, D)\mapsto (T_f^{(0)}[U, D^c])^c$\\
\end{tabular}
\end{center}
}}
\end{defn}
The adjoints and residuals of the maps above (cf.~Section \ref{ssec:notation}) are defined as follows:
{\small{
\begin{center}
\begin{tabular}{r  cr c rcr}
$\dboxun^{\mathbb{K}^+}: \mathcal{P}(X)\to \mathcal{P}(Y)$ &$\quad$ & $\xdiain^{\mathbb{K}^+}: \mathcal{P}(Y)\to \mathcal{P}(X)$ &$\quad$ &$\xboxnin^{\mathbb{K}^+}: \mathcal{P}(Y)\to \mathcal{P}(X)$ \\
$D\mapsto (R_\nu[D^c])^c$ && $U\mapsto R_\ni[U]$ &&$U\mapsto (R_{\not\ni}[U^c])^c$ \\
&&\\

$\ddiaunc^{\mathbb{K}^+}: \mathcal{P}(X)\to \mathcal{P}(Y)$ && $\xrhdnin^{\mathbb{K}^+}: \mathcal{P}(Y)\to \mathcal{P}(X)$ &$\quad$ & $\mtbra^{\mathbb{K}^+}: \mathcal{P}(X)\times \mathcal{P}(X)\to \mathcal{P}(Y)$ \\ 

 $D\mapsto R_{\nu^c}[D]$   && $U\mapsto (R_{\not\ni}[U])^c$&& $(C, D)\mapsto (T_f^{(1)}[C, D^c])^c$\\

&& $\mtAND^{\mathbb{K}^+}: \mathcal{P}(Y)\times \mathcal{P}(X)\to \mathcal{P}(X)$ \\ 
&& $(U, D)\mapsto T_f^{(2)}[U, D]$  \\
\end{tabular}
\end{center}
}}
Complex algebras of two-sorted frames can be recognized as perfect heterogeneous algebras (cf.~\cite{birkhoff1970heterogeneous}) of the following kind:
\begin{defn}
\label{def:heterogeneous algebras}
A {\em heterogeneous m-algebra} (resp.~{\em c-algebra}) is a structure \[\mathbb{H}: = (\mathbb{A}, \mathbb{B}, \dboxni^{\mathbb{H}},  \ddianni^{\mathbb{H}}, \xdianu^{\mathbb{H}}, \xboxnuc^{\mathbb{H}}) \quad\quad\text{(resp.~}\mathbb{H}: = (\mathbb{A}, \mathbb{B}, \dboxni^{\mathbb{H}}, \drhdnni^{\mathbb{H}}, \mtra^{\mathbb{H}})\text{)}\]
such that $\mathbb{A}$ and $\mathbb{B}$ are Boolean algebras, $\xdianu^{\mathbb{H}}, \xboxnuc: \mathbb{B}\to \mathbb{A}$  are finitely join-preserving and finitely meet-preserving respectively,   $\dboxni^{\mathbb{H}}, \drhdnni^{\mathbb{H}}, \ddianni^{\mathbb{H}}: \mathbb{A}\to \mathbb{B}$ are finitely meet-preserving, finitely join-reversing, and finitely join-preserving respectively, and $\mtra^{\mathbb{H}}:\mathbb{B}\times \mathbb{A}\to \mathbb{A}$ is finitely join-reversing in its first coordinate and finitely meet-preserving in its second coordinate. 
Such an $\mathbb{H}$ is {\em complete} if $\mathbb{A}$ and $\mathbb{B}$ are complete Boolean algebras and the operations above enjoy the complete versions of the finite preservation properties indicated above, and is {\em perfect} if it is complete and $\mathbb{A}$ and $\mathbb{B}$ are perfect.  
The {\em canonical extension} of a heterogeneous m-algebra (resp.~c-algebra)  $\mathbb{H}$ is \mbox{$\mathbb{H}^\delta: = (\mathbb{A}^\delta, \mathbb{B}^\delta, \dboxni^{\mathbb{H}^\delta}, \ddianni^{\mathbb{H}^\delta}, \xdianu^{\mathbb{H}^\delta}, \xboxnuc^{\mathbb{H}^\delta})$} (resp.~$\mathbb{H}^\delta: = (\mathbb{A}^\delta, \mathbb{B}^\delta, \dboxni^{\mathbb{H}^\delta}, \drhdnni^{\mathbb{H}^\delta}, \mtra^{\mathbb{H}^\delta})$), where $\mathbb{A}^\delta$ and $\mathbb{B}^\delta$ are the canonical extensions of $\mathbb{A}$ and $\mathbb{B}$ respectively \cite{jonsson1951boolean}, moreover $\dboxni^{\mathbb{H}^\delta}$, $ \drhdnni^{\mathbb{H}^\delta}$, $\xboxnuc^{\mathbb{H}^\delta}, \mtra^{\mathbb{H}^\delta}$  are the $\pi$-extensions of $\dboxni^{\mathbb{H}}, \drhdnni^{\mathbb{H}}, \xboxnuc^{\mathbb{H}}, \mtra^{\mathbb{H}}$  respectively, and $\xdianu^{\mathbb{H}^\delta},$ $\ddianni^{\mathbb{H}^\delta}$ are the $\sigma$-extensions of $\xdianu^{\mathbb{H}},\ddianni^{\mathbb{H}}$ respectively.
\end{defn}
\begin{defn}
A heterogeneous m-algebra $\mathbb{H}: = (\mathbb{A}, \mathbb{B}, \dboxni^{\mathbb{H}},  \ddianni^{\mathbb{H}}, \xdianu^{\mathbb{H}}, \xboxnuc^{\mathbb{H}})$ is {\em supported}
if $\xdianu^{\mathbb{H}} \dboxni^{\mathbb{H}}a = \xboxnuc^{\mathbb{H}}\ddianni^{\mathbb{H}} a$ for every $a\in \mathbb{A}$.
\end{defn}
It immediately follows from the definitions that 
\begin{lem}
The complex algebra of a supported two-sorted n-frame is a perfect heterogeneous  supported m-algebra. 
\end{lem}

\begin{proof}
Let $\mathbb{K} = (X, Y, R_\ni, R_{\not\ni},$ $R_\nu, R_{\nu^c})$ be a supported two-sorted n-frame. Then its complex algebra is  $\mathbb{K}^+ = (\mathcal{P}(X), \mathcal{P}(Y), \dboxni^{\mathbb{K}^+}, \ddianni^{\mathbb{K}^+}, \xdianu^{\mathbb{K}^+}, \xboxnuc^{\mathbb{K}^+})$, which is clearly perfect. Since $\mathbb{K}$ is also supported,  $R_{\nu}^{-1}[(R_{\ni}^{-1}[D^c])^c] = (R_{\nu^c}^{-1}[(R_{\not\ni}^{-1}[D])^c])^c$ for any $D\subseteq \mathbb{K}$.  Hence, 
\[ \xdianu^{\mathbb{K}^+} \dboxni^{\mathbb{K}^+}D =  R_{\nu}^{-1}[(R_{\ni}^{-1}[D^c])^c] 
= (R_{\nu^c}^{-1}[(R_{\not\ni}^{-1}[D])^c])^c
= \xboxnuc^{\mathbb{K}^+} \ddianni^{\mathbb{K}^+}  D.\]
\end{proof}

\begin{defn}
If  $\mathbb{H} = (\mathcal{P}(X), \mathcal{P}(Y), \dboxni^{\mathbb{H}},  \ddianni^{\mathbb{H}}, \xdianu^{\mathbb{H}}, \xboxnuc^{\mathbb{H}})$ is a  perfect heterogeneous m-algebra
 (resp.~$\mathbb{H} = (\mathcal{P}(X), \mathcal{P}(Y), \dboxni^{\mathbb{H}}, \drhdnni^{\mathbb{H}}, \mtra^{\mathbb{H}})$ is a perfect heterogeneous~c-algebra), its associated two-sorted n-frame (resp.~c-frame) is
 \[\mathbb{H}_+: = (X, Y, R_\ni, R_{\not\ni}, R_\nu, R_{\nu^c})\quad\quad \text{(resp.~}\mathbb{H}_+: = (X, Y, R_\ni, R_{\not\ni}, T_f) \text{), s.t.}\]
 {\small{
\begin{itemize}
\item  $R_{\ni}\subseteq Y\times X$ is defined by $yR_\ni x$ iff $y\notin \dboxni^{\mathbb{H}}x^c$,
\item  $R_{\not\ni}\subseteq Y\times X$ is defined by $xR_{\not\ni} y$ iff $y\in \ddianni^{\mathbb{H}}\{x\}$  (resp.~$y\notin \drhdnni^{\mathbb{H}}\{x\}$),
\item $R_\nu\subseteq X\times Y$ is defined by $xR_\nu y$ iff $x\in \xdianu^{\mathbb{H}}\{y\}$, 
\item $R_{\nu^c}\subseteq X\times Y$ is defined by $xR_{\nu^c} y$ iff $x\notin \xboxnuc^{\mathbb{H}}y^c$, 
\item $T_f\subseteq X\times Y\times X$ is defined by $(x', y, x)\in T_f$ iff $x'\notin  \{y\}\mtra^{\mathbb{H}} x^c$. 
\end{itemize} 
}}
\end{defn}
\begin{lem}
If  $\mathbb{H}$ is a  perfect supported  heterogeneous m-algebra, then $\mathbb{H}_+$ is a supported two-sorted n-frame.
\end{lem}

\begin{proof}
To show that $\mathbb{H}_+$ is supported, for every $D\subseteq X$,
\[R_{\nu}^{-1}[(R_{\ni}^{-1}[D^c])^c] = \xdianu^{\mathbb{H}} \dboxni^{\mathbb{H}}D
= \xboxnuc^{\mathbb{H}} \ddianni^{\mathbb{H}}  D
= (R_{\nu^c}^{-1}[(R_{\not\ni}^{-1}[D])^c])^c.\]
\end{proof}

The duality between perfect BAOs and Kripke frames can be readily extended to the present two-sorted case. 
The following proposition collects these well-known facts, the proofs of which are analogous to those of the single-sorted case, hence  are omitted.  
\begin{prop} \label{prop:dduality multi-type}
For every heterogeneous m-algebra (resp.~c-algebra)  $\mathbb{H}$ and every two-sorted n-frame  (resp.~c-frame) $\mathbb{K}$,
\begin{enumerate}
\item $\mathbb{K}^+$ is a perfect heterogeneous m-algebra (resp.~c-algebra);
\item $(\mathbb{K}^+)_+\cong \mathbb{K}$, and if $\mathbb{H}$ is  perfect, then $(\mathbb{H}_+)^+\cong \mathbb{H}$.
\end{enumerate}
\end{prop}

\subsection{Equivalent representation of  m-algebras  and c-algebras}
\label{Heterogeneous presentation}
Every supported heterogeneous m-algebra (resp.~c-algebra) can be associated with an m-algebra (resp.~a c-algebra) as follows:
\begin{defn}
For every  supported heterogeneous m-algebra  $\mathbb{H} = (\mathbb{A},\mathbb{B}, \dboxni^{\mathbb{H}}, \ddianni^{\mathbb{H}}, \xdianu^{\mathbb{H}},$ \\ $ \xboxnuc^{\mathbb{H}})$ 
 \ (resp.~c-algebra $\mathbb{H} = (\mathbb{A},\mathbb{B}, \dboxni^{\mathbb{H}}, \drhdnni^{\mathbb{H}}, \mtra^{\mathbb{H}})$),  let $\mathbb{H}_\bullet:  = (\mathbb{A}, \abla^{\mathbb{H}_\bullet})$ \ \ (resp.~$\mathbb{H}_\bullet:  = (\mathbb{A},$ \\ $>^{\mathbb{H}_\bullet})$), where for every $a\in\mathbb{A}$ (resp.~$a, b\in \mathbb{A}$),  
 \[\abla^{\mathbb{H}_\bullet} a = \xdianu^{\mathbb{H}}\dboxni^{\mathbb{H}} a = \xboxnuc^{\mathbb{H}}\ddianni^{\mathbb{H}} a\quad\quad \text{ (resp.~}a >^{\mathbb{H}_\bullet}b: = (\dboxni^{\mathbb{H}} a \wedge \drhdnni^{\mathbb{H}} a)\mtra^{\mathbb{H}} b\text{)}.\]
%
%
\end{defn}
It immediately follows from the stipulations above that $\abla^{\mathbb{H}_\bullet}$ is a monotone map (resp. $>^{\mathbb{H}_\bullet}$ is finitely  meet-preserving in its second coordinate), and hence $\mathbb{H}_\bullet$ is an m-algebra (resp.~a c-algebra).
Conversely, every complete m-algebra (resp.~c-algebra) can be associated with a complete supported heterogeneous m-algebra (resp.~a c-algebra) as follows:
\begin{defn}
For every  complete m-algebra  $\mathbb{C} = (\mathbb{A},\abla^{\mathbb{C}})$ (resp.~complete c-algebra  $\mathbb{C} = (\mathbb{A},>^{\mathbb{C}})$),  let $\mathbb{C}^\bullet:  = (\mathbb{A}, \mathcal{P}(\mathbb{A}), \dboxni^{\mathbb{C}^\bullet}, \ddianni^{\mathbb{C}^\bullet}, \xdianu^{\mathbb{C}^\bullet}, \xboxnuc^{\mathbb{C}^\bullet})$ (resp.~$\mathbb{C}^\bullet:  = (\mathbb{A}, \mathcal{P}(\mathbb{A}), \dboxni^{\mathbb{C}^\bullet},$ \mbox{$ \drhdnni^{\mathbb{C}^\bullet},$} $\mtra^{\mathbb{C}^\bullet})$), where for every $a\in\mathbb{A}$ and $B\in \mathcal{P}(\mathbb{A})$,
{\small{
\[\dboxni^{\mathbb{C}^\bullet} a: = \{b\in\mathbb{A}\mid b\leq a\}\quad \quad\xdianu^{\mathbb{C}^\bullet}B: = \bigvee \{\abla^{\mathbb{C}} b\mid b\in B\} \quad\quad  \drhdnni^{\mathbb{C}^\bullet} a: = \{b\in\mathbb{A}\mid a\leq b\}\]
\[\xboxnuc^{\mathbb{C}^\bullet} B: =\bigwedge \{\abla^{\mathbb{C}}b\mid b\notin B\} \quad B \mtra^{\mathbb{C}^\bullet} a: = \bigwedge\{b >^{\mathbb{C}} a\mid b\in B\}\quad  \ddianni^{\mathbb{C}^\bullet} a: = \{b\in\mathbb{A}\mid a\nleq b\}.\]
}}
\end{defn}
\begin{lem} 
\label{lemma:cbullet}
If $\mathbb{C}$  is a  complete m-algebra  (resp.~complete c-algebra), then
  $\mathbb{C}^\bullet$ is a complete supported heterogeneous m-algebra (resp.~c-algebra). 
\end{lem}

\begin{proof}
Let $\mathbb{C}=(\mathbb{A}, \nabla^\mathbb{C})$  be a  complete m-algebra. First we show that $\mathbb{C}^\bullet$ is a complete heterogeneous m-algebra. For $X \subseteq \mathbb{A}$ and $\Gamma \subseteq \mathcal{P}(\mathbb{A})$, 
\begin{align*} 
\dboxni^{\mathbb{C}^\bullet} \bigwedge X = \{ b \in \mathbb{A} \mid b \le \bigwedge X\}=\bigcap_{x \in X} \{ b \in \mathbb{A} \mid b \le  x\}= \bigcap_{x \in X} \dboxni^{\mathbb{C}^\bullet}x\\
\ddianni^{\mathbb{C}^\bullet} \bigvee X =\{b \in \mathbb{A} \mid  \bigvee  X\not < b\} = \bigcup_{x \in X}\{b \in \mathbb{A} \mid  x\not < b\}= \bigcup_{x \in X} \ddianni^{\mathbb{C}^\bullet}x \\
\xdianu^{\mathbb{C}^\bullet}\bigcup \Gamma = \bigvee\{\nabla^\mathbb{C} b \mid b \in \bigcup\Gamma\} = \bigvee_{Y \in \Gamma} \bigvee \{\nabla^\mathbb{C} b \mid b \in Y\}= \bigvee_{Y \in \Gamma} \xdianu^{\mathbb{C}^\bullet}Y\\
\xboxnuc^{\mathbb{C}^\bullet} \bigcap \Gamma = \bigwedge\{ \nabla^\mathbb{C} b \mid b \not \in \bigcap \Gamma\} =\bigcap_{Y \in \Gamma} \bigwedge \{ \nabla^\mathbb{C} b \mid b \not \in Y\}= \bigcap_{Y \in \Gamma}\xboxnuc^{\mathbb{C}^\bullet} Y.
\end{align*}
Let us show that $\mathbb{C}^\bullet$ is supported. For every $a\in \mathbb{A}$,
\begin{align*}
\xdianu^{\mathbb{C}^\bullet}\dboxni^{\mathbb{C}^\bullet} a =  \xdianu^{\mathbb{C}^\bullet} \{b \in \mathbb{A} \mid b\le a \}=\bigvee \{ \nabla^\mathbb{C} b \mid b \le a \}= \nabla^\mathbb{C} a, \\
\xboxnuc^{\mathbb{C}^\bullet}\ddianni^{\mathbb{C}^\bullet} a =\xboxnuc^{\mathbb{C}^\bullet} \{ b\in \mathbb{A}\mid a \not \le b\} = \bigwedge\{\nabla^\mathbb{C} b \mid a \le b\} = \nabla^\mathbb{C} a.
\end{align*}
Hence, $\xdianu^{\mathbb{C}^\bullet}\dboxni^{\mathbb{C}^\bullet} a=\xboxnuc^{\mathbb{C}^\bullet}\ddianni^{\mathbb{C}^\bullet} a $.

Let $\mathbb{C}=(\mathbb{A}, >^\mathbb{C})$ be a complete c-algebra. That $\dboxni^{\mathbb{C}^\bullet}$ is completely join preserving can be proved as shown above. As to the remaining connectives, for any $X \subseteq \mathbb{A}$ and $\Gamma \subseteq \mathcal{P}$, 
\begin{align*}
 \mbox \drhdnni^{\mathbb{C}^\bullet}\bigvee X = \{b \in \mathbb{A}\mid \bigvee X \le b\}=\bigcap_{x \in X} \{b \in \mathbb{A}\mid x \le b\}= \bigcap_{x \in X}  \mbox \drhdnni^{\mathbb{C}^\bullet}x\\
 \bigcup \Gamma \mtra^{\mathbb{C}^\bullet} a = \bigwedge \{ b >^\mathbb{C} a  \mid b \in \bigcup \Gamma\}=\bigwedge_{Y \in \Gamma} \bigwedge  \{ b >^\mathbb{C}a  \mid b \in  Y\}= \bigwedge_{Y \in \Gamma} (Y \mtra^{\mathbb{C}^\bullet} a)\\
B \mtra^{\mathbb{C}^\bullet} \bigwedge X =  \bigwedge \{ b >^\mathbb{C}  \bigwedge X  \mid b \in B \}= \bigwedge_{x\in X} \bigwedge \{ b >^\mathbb{C}  x\mid b \in B \}=  \bigwedge_{x\in X}(B \mtra^{\mathbb{C}^\bullet}x).
\end{align*}
\end{proof}

\begin{prop}
\label{prop:alg characterization of single into multi}
If   $\mathbb{C}$ is a complete  m-algebra (resp.~c-algebra), then   $\mathbb{C} \cong (\mathbb{C}^\bullet)_\bullet$. Moreover, if $\mathbb{H}$ is a complete supported heterogeneous m-algebra (resp.~c-algebra), then $\mathbb{H}\cong \mathbb{C}^\bullet$ for some complete  m-algebra (resp.~c-algebra) $\mathbb{C}$ iff $\mathbb{H} \cong (\mathbb{H}_\bullet)^\bullet$.
\end{prop}

\begin{proof}
For the first part of the statement, by definition, $\mathbb{C}$ and $(\mathbb{C}^\bullet)_\bullet$ have the same underlying Boolean algebra. Moreover, $\abla^{(\mathbb{C}^\bullet)_\bullet} a= \xdianu^{\mathbb{C}^\bullet}\dboxni^{\mathbb{C}^\bullet}a = \abla^{\mathbb{C}} a$ for every $a\in \mathbb{C}$, the first identity holding by definition, the second one being shown in the proof of Lemma \ref{lemma:cbullet}. 

As to the second part, for the left to right direction, assume that $\mathbb{H}\cong \mathbb{C}^\bullet$ for some complete  m-algebra (resp.~c-algebra) $\mathbb{C}$. From the first part of the proposition we know that $\mathbb{C} \cong (\mathbb{C}^\bullet)_\bullet$. Then $\mathbb{H}\cong \mathbb{C}^\bullet \cong ((\mathbb{C}^\bullet)_\bullet)^\bullet\cong  (\mathbb{H}_\bullet)^\bullet$. For the right to left direction, $\mathbb{H}_\bullet$ is the required complete m-algebra (resp. c-algebra).
\end{proof}

The proposition above characterizes up to isomorphism the supported heterogeneous  m-algebras (resp.~c-algebras) which arise from single-type m-algebras (resp.~c-algebras).

\subsection{Representing n-frames and c-frames as two-sorted Kripke frames}
Thanks to the discrete dualities discussed in Sections \ref{ssec:prelim} and \ref{ssec:2sorted Kripke frame}, we can transfer the algebraic characterization of Proposition \ref{prop:alg characterization of single into multi} to the side of frames, as detailed in this subsection.

\begin{defn}\label{def:two-sortedFrames}
For any n-frame (resp.~c-frame) $\mathbb{F}$, we let $\mathbb{F}^\star: = ((\mathbb{F}^\ast)^\bullet)_+$, and for every supported two-sorted n-frame (resp.~c-frame) $\mathbb{K}$, we let $\mathbb{K}_\star: = ((\mathbb{K}^+)_\bullet)_\ast$.
\end{defn}
Spelling out the definition above, if $\mathbb{F}=(W,\nu)$ (resp.~$\mathbb{F}=(W, f)$) then  $\mathbb{F}^\star = (W,\mathcal{P} (W), R_\ni,$ $R_{\not\ni}, R_{\nu}, R_{\nu^c})$  (resp.~$\mathbb{F}^\star = (W,\mathcal{P} (W),R_{\not\ni}, R_{\ni}, T_f)$) where:
{\small{
\begin{itemize}
\item $R_{\nu} \subseteq W\times \mathcal{P}(W)$ is defined as $x R_{\nu}  Z$ iff $Y\in \nu(x)$;
\item $R_{\nu^c} \subseteq W\times \mathcal{P}(W)$ is defined as $x R_{\nu^c}  Z$ iff $Z\notin \nu(x)$;
\item  $R_{\ni} \subseteq \mathcal{P}(W) \times W$ is defined as $Z R_{\ni}  x$ iff $x\in Z$;
\item  $R_{\not\ni} \subseteq \mathcal{P}(W) \times W$ is defined as $Z R_{\not\ni}  x$ iff $x\notin Z$;
\item $T_f\subseteq W\times \mathcal{P}(W) \times W$ is defined as $T_f(x, Z, x')$ iff $x'\in f(x, Z)$.
\end{itemize}
}}
Moreover, if $\mathbb{K} = (X, Y, R_\ni, R_{\not\ni}, R_{\nu}, R_{\nu^c})$ (resp.~$\mathbb{K} = (X, Y, R_\ni, R_{\not\ni}, T_f)$), then $\mathbb{K}_{\star} = (X, \nu_{\star})$
(resp.~$\mathbb{K}_\star = (X, f_\star)$) where:
{\small{
\begin{itemize}
\item $\nu_\star (x)  = \{D\subseteq X\mid x\in R_\nu^{-1}[(R_\ni^{-1}[D^c])^c] \} = \{D\subseteq X\mid x\in (R_{\nu^c}^{-1}[(R_{\not\ni}^{-1}[D])^c])^c\}$;
\item $f_\star(x, D)  
=  \bigcap \{C\subseteq X\mid x\in T^{(0)}_f[ \{C\}, D^c]\}$.
\end{itemize}
}}

\begin{lem}\label{lem:supported}
If $\mathbb{F} = (W,\nu)$ is an n-frame, then  $\mathbb{F}^\star$ is a supported two-sorted n-frame. 
\end{lem}
\begin{proof}
 By definition, $\mathbb{F}^\star$ is a two-sorted n-frame. Moreover, for any $D\subseteq W$,
{\fns{
\begin{center}
\begin{tabular}{clll}
$(R_{\nu^c}^{-1}[(R_{\not\ni}^{-1}[D])^c])^c$ &  = & $ \{w\mid \forall X(X\notin \nu(w)\Rightarrow \exists u(X\not\ni u\ \&\ u\in D))\}$\\
&  = & $ \{w\mid \forall X(X\notin \nu(w)\Rightarrow   D \not\subseteq  X)\}$\\
&  = & $ \{w\mid \forall X(D\subseteq  X \Rightarrow X\in \nu(w) )\}$\\
&  = & $ \{w\mid \exists X(X\in \nu(w) \ \&\  X\subseteq  D)\}$ & ($\ast$)\\
&  = & $R_{\nu}^{-1}[(R_{\ni}^{-1}[D^c])^c].$\\
\end{tabular}
\end{center}
}}
To show the identity marked with $(\ast)$, from top to bottom, take $X:= D$; conversely, if $D\subseteq Z$ then $X\subseteq Z$, and since  by assumption $X\in \nu(w)$ and $\nu(w)$ is upward closed, we conclude that $Z\in \nu(w)$, as required.
\end{proof}

The next proposition is the frame-theoretic counterpart of Proposition \ref{prop:alg characterization of single into multi}. 
\begin{prop}
\label{prop:adjunction-frames}
If   $\mathbb{F}$ is an n-frame (resp.~c-frame), then   $\mathbb{F} \cong (\mathbb{F}^\star)_\star$. Moreover, if $\mathbb{K}$ is a supported two-sorted n-frame (resp.~c-frame), then $\mathbb{K}\cong \mathbb{F}^\star$ for some n-frame (resp.~c-frame) $\mathbb{F}$ iff $\mathbb{K} \cong (\mathbb{K}_\star)^\star$.
\end{prop}

\begin{proof}
For the first part of the statement,
\begin{center}
\begin{tabular}{r c l l} 
$(\mathbb{F}^\star)_\star $ & $= $& $(((((\mathbb{F}^\ast)^\bullet)_+)^+)_\bullet)_\ast $ & definition of $(-)^\star$ and $(-)_\star$\\
 & $\cong $& $(((\mathbb{F}^\ast)^\bullet)_\bullet)_\ast$ & Proposition \ref{prop:dduality multi-type}.2,  $(\mathbb{F}^\ast)^\bullet$ perfect heterogeneous algebra\\ 
  & $= $& $(\mathbb{F}^\ast)_\ast$ & Proposition \ref{prop:alg characterization of single into multi}, since $\mathbb{F}^\ast$ is complete \\
    & $= $& $\mathbb{F}$. & Proposition \ref{prop:dduality single type}
  \end{tabular}
  \end{center}
As to  the second part, for the left to right direction, assume that $\mathbb{K}\cong \mathbb{F}^\star$ for some m-frame (resp.~c-frame) $\mathbb{F}$. From the first part of the statement we know that $\mathbb{F} \cong (\mathbb{F}^\star)_\star$. Then $\mathbb{K}\cong \mathbb{F}^\star \cong ((\mathbb{F}^\star)_\star)^\star\cong  (\mathbb{K}_\star)^\star$. For the right to left direction, $\mathbb{K}_\star$ is the required m-frame (resp. c-frame).

\end{proof}
\section{Embedding non-normal logics into two-sorted normal logics}
\label{sec:embedding}
The two-sorted frames and heterogeneous algebras discussed in the previous section serve as semantic environment for the multi-type  languages defined below.

\paragraph{Multi-type languages.} For a denumerable set $\mathsf{Prop}$ of atomic propositions, the languages $\mathcal{L}_{MT\abla}$ and $\mathcal{L}_{MT>}$ in  types $\mathsf{S}$ (sets) and $\mathsf{N}$ (neighbourhoods) over $\mathsf{Prop}$ are defined as follows:
{\small
\begin{center}
$\begin{array}{lll}
\mathsf{S} \ni A::= p  \mid \top \mid \bot \mid \neg A \mid A \land A \mid \xdianu \alpha\mid \xboxnuc\alpha &\quad\quad&\mathsf{S} \ni A::= p  \mid \top \mid \bot \mid \neg A \mid A \land A \mid  \alpha\mtra A
\\
\mathsf{N} \ni \alpha ::=  \dtop\mid \dbot \mid {\sim} \alpha \mid \alpha \dand \alpha \mid \dboxni A\mid \ddianni\alpha &\quad\quad&\mathsf{N} \ni \alpha ::=  \dtop\mid \dbot \mid {\sim} \alpha \mid \alpha \dand \alpha \mid \dboxni A\mid \drhdnni A.
\end{array}$
\end{center}
}
\paragraph{Algebraic semantics.} Interpretation of $\mathcal{L}_{MT\abla}$-formulas  (resp.~$\mathcal{L}_{MT>}$formulas) in heterogeneous m-algebras (resp.~c-algebras) under homomorphic assignments $h:\mathcal{L}_{MT\abla}\to \mathbb{H}$ (resp.~$h:\mathcal{L}_{MT>}\to \mathbb{H}$)  and validity of formulas in heterogeneous algebras ($\mathbb{H}\models\Theta$) are defined as usual. 

\paragraph{Frames and models.} $\mathcal{L}_{MT\abla}$-{\em models} (resp.~$\mathcal{L}_{MT>}$-{\em models}) are pairs $\mathbb{N} = (\mathbb{K}, V)$ s.t.~$\mathbb{K}= (X,Y,R_{\ni}, R_{\not\ni}, R_{\nu}, R_{\nu^c})$ is a supported two-sorted n-frame (resp.~$\mathbb{K}= (X,Y,R_{\ni},R_{\not\ni}, T_f)$ is a two-sorted c-frame) and $V:\mathcal{L}_{MT}\to\mathbb{K}^+$ is a heterogeneous algebra homomorphism of the appropriate signature. Hence, truth of formulas at states in models is defined as $\mb{N},z \Vdash \Theta$ iff $z\in V(\Theta)$ for every $z\in X\cup Y$ and $\Theta\in \mathsf{S}\cup\mathsf{N}$, and unravelling this stipulation   for formulas with a modal operator as main connective, we get:
{\small{
\begin{itemize}
\item $\mb{N},x \Vdash \xdianu \alpha \quad \text{iff}\quad  \mb{N},y \Vdash  \alpha \text{ for some } y \text{ s.t. } xR_\nu y$;
\item $\mb{N},x \Vdash \xboxnuc \alpha \quad \text{iff}\quad  \mb{N},y \Vdash  \alpha \text{ for all } y \text{ s.t. } xR_{\nu^c} y$;
\item $\mb{N},y \Vdash \dboxni A \quad \text{iff}\quad  \mb{N},x \Vdash  A \text{ for all } x \text{ s.t. } yR_\ni x$;
\item $\mb{N},y \Vdash \ddianni A \quad \text{iff}\quad  \mb{N},x \Vdash  A \text{ for some } x \text{ s.t. } yR_{\not\ni} x$;
\item $\mb{N},y \Vdash \drhdnni A \quad \text{iff}\quad  \mb{N},x \not\Vdash  A \text{ for all } x \text{ s.t. } yR_{\not\ni} x$;
\item $\mb{N},x \Vdash \alpha\mtra A \quad \text{iff}\quad  \text{ for all } y\text{ and all } x', \text{ if } T_f(x, y, x') \text{ and } \mb{N},y \Vdash \alpha \text{ then } \mb{N},x' \Vdash  A$.
\end{itemize}
}}
Global satisfaction (notation: $\mathbb{N}\Vdash\Theta$) is defined relative to the domain of the appropriate type, and frame validity (notation: $\mathbb{K}\Vdash\Theta$) is defined as usual. Thus, by definition, $\mathbb{K}\Vdash\Theta$ iff $\mathbb{K}^+\models \Theta$, 
and if $\mathbb{H}$ is a perfect heterogeneous algebra, then   $\mathbb{H}\models\Theta$ iff $\mathbb{H}_+\Vdash \Theta$. 
\paragraph{Correspondence theory for multi-type normal logics.} This semantic environment supports a straightforward extension of unified correspondence theory for multi-type normal logics, which includes the definition of inductive and analytic inductive formulas and inequalities in $\mathcal{L}_{MT\abla}$  and $\mathcal{L}_{MT>}$ (cf.~Section \ref{sec:analytic inductive ineq}), and a corresponding version of the algorithm ALBA \cite{CoPa:non-dist} for computing their first-order correspondents and analytic structural rules. 
\paragraph{Translation.} Correspondence theory and analytic calculi for the non-normal logics  $\mathbf{L}_\abla$ and $\mathbf{L}_>$ and their analytic extensions can be then obtained `via translation', i.e.~by recursively defining  translations  $\tau_1, \tau_2:\mathcal{L}_\abla \to \mathcal{L}_{MT\abla}$ and $(\cdot)^\tau:\mathcal{L}_>\to \mathcal{L}_{MT>}$   as follows: 
{\small
\begin{center}
\begin{tabular}{rcl c rcl c rcl}
$\tau_1(p)$              &$=$& $p$                                        && $\tau_2(p)$              &$=$& $p$      &&         $p^\tau$ &=& $p$         \\
$\tau_1(\varphi \xand \psi)$ &$=$& $\tau_1(\varphi) \xand \tau_1(\psi)$ && $\tau_2(\varphi \xand \psi)$ &$=$& $\tau_2(\varphi) \xand \tau_2(\psi)$  &&  $(\varphi \land \psi)^\tau$ &=& $\varphi^\tau \land \psi^\tau$  \\
$\tau_1(\xneg \varphi)$    &$=$& $\xneg \tau_2(\varphi)$                 && $\tau_2(\xneg \varphi)$    &$=$& $\xneg \tau_1(\varphi)$    && $(\neg\varphi)^\tau$ &=& $\neg \varphi^\tau$                \\
$\tau_1(\abla \varphi)$     &$=$& $\xdianu \dboxni \tau_1(\varphi)$  && $\tau_2(\abla \varphi)$     &$=$& $\xboxnuc \ddianni \tau_2(\varphi)$ && $(\varphi > \psi)^\tau$ &=& $(\dboxni \varphi^\tau \wedge\drhdnni \varphi^\tau)\mtra \psi^\tau$\\
\end{tabular}
\end{center}
 }
 Let  $\tau (\varphi \vdash \psi): = \varphi^\tau \vdash \psi^\tau$ if $\varphi \vdash \psi$ is an $\mathcal{L}_>$-sequent, and 
 $\tau (\varphi \vdash \psi): = \tau_1(\varphi)\vdash \tau_2(\psi)$ if $\varphi \vdash \psi$ is an $\mathcal{L}_\abla$-sequent.

\begin{prop}
\label{prop:consequence preserved and reflected}
If $\mathbb{F}$ is an n-frame (resp.~c-frame) and $\varphi\vdash \psi$ is an $\mathcal{L}_{\abla}$-sequent  (resp.~an $\mathcal{L}_{>}$-sequent), then $\mathbb{F}\Vdash \varphi\vdash \psi \quad\text{ iff }\quad \mathbb{F}^\star\Vdash \tau(\varphi\vdash \psi)$.
\end{prop}

\begin{proof} 
When $\mathbb{F}$ is an n-frame, the proposition is an immediate consequence of the following claim:
\[
(\mathbb{F},V),w\Vdash \varphi\quad\text{~iff~} \quad(\mathbb{F}^\star,V),w \Vdash \tau_1(\varphi) \quad\text{~iff~} \quad(\mathbb{F}^\star,V),w \Vdash \tau_2(\varphi),
\]
which can be proved by induction on $\varphi$. We only sketch the case in which $\varphi: =\nabla \psi$. In this case, $\tau_1(\nabla \psi) = \xdianu \dboxni \tau_1(\psi)$ and $\tau_2(\nabla \psi) =\xboxnuc \ddianni \tau_2(\psi)$.
\begin{center} 
\begin{tabular}{r c l c } 
$\mathbb{F},V,w\Vdash \nabla \psi$ &iff & 
$\exists D (D \in \nu(w) \ \& \ D \subseteq V( \psi ))$\\
&iff & $\exists D(wR_\nu D \ \& \ \forall d (DR_\ni d \Rightarrow d \in V(\psi)))$\\
&iff & 
$\exists D(wR_\nu D \ \& \ \forall d (DR_\ni d \Rightarrow d \in V(\tau_1(\psi)))$ &$\quad $Induction hypothesis\\
&iff &$\mathbb{F}^\star,V,w \Vdash \xdianu \dboxni \tau_1(\psi)$\\
\end{tabular}
\end{center}

\begin{center}
\begin{tabular}{r c l c } 
$\mathbb{F},V,w\Vdash \nabla \psi$ &iff & 
$\exists D (D \in \nu(w) \ \& \ D \subseteq V( \psi ))$\\
&iff & $\exists D(wR_\nu D \ \& \ \forall d (DR_\ni d \Rightarrow d \in V(\psi)))$\\
$(\ast)$&iff & 
$\forall D (w R_{\nu^c}D  \Rightarrow \exists d(D R_{\not \ni} d \ \& \ d \in V(\psi)))$\\
&iff & 
$\forall D (w R_{\nu^c}D  \Rightarrow \exists d(D R_{\not \ni} d \ \& \ d \in V(\tau_2(\psi))))$ & $\quad$ Induction hypothesis\\
&iff& $ \mathbb{F}^\star,V,w \Vdash \xboxnuc \ddianni \tau_2(\psi)$.\\
\end{tabular}
\end{center}
The equivalence marked by $(\ast)$ follows from Lemma \ref{lem:supported}.

When $\mathbb{F}$ is a $c$-frame, the proposition is an immediate consequence of the following claim,
which can be shown by induction on $\varphi$.
\[
(\mathbb{F},V),w \Vdash \varphi \quad \text{~iff~}\quad  (\mathbb{F}^\star ,V),w \Vdash \varphi ^\tau.
\] We only sketch the case in which $\varphi: = \varphi > \psi$. In this case, $(\varphi >\psi)^\tau = (\dboxni \varphi^\tau \wedge\drhdnni \varphi^\tau)\mtra \psi^\tau$. 
\begin{align*}
(\mathbb{F},V), w \Vdash \varphi > \psi &\text{~iff~} f(w,V(\varphi)) \subseteq V(\psi) \\
&\text{~iff~} \forall x(x \in f(w , V(\varphi)) \Rightarrow x \in V(\psi))\\
&\text{~iff~} \forall x\forall Y( x\in f(w,Y )\ \& \ Y =V(\varphi)\Rightarrow x \in V(\psi))\\
&\text{~iff~} \forall x\forall Y( x\in f(w,Y )\ \& \ Y =V(\varphi^\tau)\Rightarrow x \in V(\psi^\tau)) && \text{I.H.}\\
&\text{~iff~} \forall x\forall Y(T_f(w, Y,x) \ \& \ (\forall y(YR_\ni y \Rightarrow y \in V(\varphi^\tau )))\ \&\ 
\\&\qquad\qquad\qquad\qquad\qquad (\forall y(YR_{\not \ni} y \Rightarrow y \not\in V(\varphi^\tau))) \Rightarrow x \in V(\psi^\tau))\\
&\text{~iff~}  (\mathbb{F}^\star,V), w \Vdash (\dboxni \varphi^\tau \wedge\drhdnni \varphi^\tau)\mtra \psi^\tau.
\end{align*}
\end{proof}

With this framework in place, we are in a position to (a) retrieve correspondence results in the setting of {\em non-normal} logics, such as those collected in Theorem \ref{theor:correspondence-noAlba},
as instances of the general  Sahlqvist theory for multi-type {\em normal} logics, and (b) recognize whether the translation of a non-normal  axiom  is analytic inductive, and compute its corresponding analytic structural rules (cf.~Section \ref{sec:ALBA runs}). 

\section{Algorithmic correspondence for non-normal logics}
\label{section:correspondence}
In this section, we detail how the two-sorted environment introduced and discussed in the previous sections can be used to establish a Sahlqvist-type correspondence framework for {\em classes} of non-normal logics (the generality of this approach is further discussed in Section \ref{sec: Conclusions}) which can be specialized to the signatures of monotone modal logic and conditional logic, encompasses and extends the well-known correspondence-theoretic results for these logics collected in Theorem \ref{theor:correspondence-noAlba}, and brings them into the fold of unified correspondence theory  \cite{CoGhPa14,CoPa:non-dist}. 
The unified correspondence approach pivots on the order theoretic properties of the algebraic interpretation of logical connectives. As pointed out in \cite{bilkova2018logic}, when the relevant order theoretic properties hold in a given multi-type setting such as the one introduced in Section \ref{sec: semantic analysis}, the insights, tools and results of unified correspondence theory can be straightforwardly transferred to it. Specifically for the present cases of monotone modal logic and conditional logic, this means, firstly, that we can specialize the definition of inductive and analytic inductive inequalities/sequents to the  languages $\mathcal{L}_{MT\abla}$ and $\mathcal{L}_{MT>}$ defined in the previous section. This definition is given in Section \ref{sec:analytic inductive ineq}; 
in the following table, we list the translations of the axioms of Theorem \ref{theor:correspondence-noAlba}, and for each, the last column of the table specifies whether its translation is analytic inductive. 

{\small{
\begin{center}
\begin{tabular}{@{}r l c l cc c}
&Axiom && Translation & Inductive & Analytic \\
N\, & $\abla \top\quad$ &&   $\top\leq \xboxnuc\ddianni \top$ & $\checkmark$ & $\checkmark$\\
P\, & $\neg \abla \bot\quad$ &&  $\top\leq \neg \xdianu \dboxni \bot$ & $\checkmark$ & $\checkmark$ \\
C\, &$ \abla p \land \abla q \to \abla(p \land q)\quad$ &&  $\xdianu \dboxni  p \land\xdianu \dboxni q \leq \xboxnuc\ddianni (p \land q)$ & $\checkmark$ & $\checkmark$\\
T\, &$\abla p \to p \quad$ &&   $\xdianu \dboxni p\leq p $& $\checkmark$ & $\checkmark$\\
4\, & $\abla \abla  p \to \abla p\quad$ &&   $\xdianu \dboxni \xdianu \dboxni   p \leq\xboxnuc\ddianni  p$& $\checkmark$ & $\times$\\
4'\, & $\abla p \to \abla \abla p\quad$ &&  $\xdianu \dboxni  p \leq \xboxnuc\ddianni\xboxnuc\ddianni  p$& $\checkmark$ & $\times$\\
5\, & $\neg \abla \neg p \to \abla \neg \abla \neg p \quad$ &&  $\neg \xboxnuc\ddianni\neg  p \leq \xboxnuc\ddianni \neg\xdianu \dboxni \neg p$& $\checkmark$ & $\times$\\
B\, & $p \to \abla \neg \abla \neg p \quad$ &&   $  p \leq \xboxnuc\ddianni\neg \xdianu \dboxni \neg p$& $\checkmark$ & $\times$\\
D\, & $\abla p \to \neg \abla \neg p\quad$ &&   $\xdianu \dboxni   p \leq  \neg \xdianu \dboxni \neg p$& $\checkmark$ & $\checkmark$\\
CS\, & $(p\wedge q) \to (p > q)\quad$ &&  $(p\wedge q) \leq ((\dboxni p \wedge\drhdnni p)\mtra q)$& $\checkmark$ & $\checkmark$\\
CEM\, & $(p > q) \vee (p> \neg q)\quad$ &&  $\top\leq ((\dboxni p \wedge\drhdnni p)\mtra q) \vee ((\dboxni p \wedge\drhdnni p)\mtra \neg q)$& $\checkmark$ & $\checkmark$\\
ID\, & $p >  p \quad$ &&$\top\leq (\dboxni p \wedge\drhdnni p)\mtra p$& $\checkmark$ & $\checkmark$\\
CN\, & $(p >  q )\lor (q >p) \quad$ &&$\top\leq (\dboxni p \wedge\drhdnni p)\mtra q )\lor ( (\dboxni q \wedge\drhdnni q)\mtra p$& $\checkmark$ & $\checkmark$\\
T\, & $(\bot >\neg p)\to p  \quad$ &&$
((\dboxni \bot \wedge\drhdnni \bot)\mtra \neg p) \leq p $& $\checkmark$ & $\times$\\

\end{tabular}
\end{center}
}} 
\begin{remark}
\label{rmk: explanation for double translation}
The positional translation of $\mathcal{L}_\abla$-axioms/sequents guarantees that a greater number of translated axioms are analytic inductive. To illustrate this point, consider axiom C above; translating it using e.g.~only $\tau_1$ yields     $\xdianu \dboxni  p \land\xdianu \dboxni q \leq \xdianu \dboxni  (p \land q)$ which is inductive but not analytic, since in $- \xdianu \dboxni  (p \land q)$ some branches (in fact all) are not good. This trick is not a panacea: occurrences of nested $\abla$ connectives, as in axiom 4, 4', 5 and B, will give rise to McKinsey-type nestings of modal operators also under the positional translation, which results in some branches being not good.
\end{remark} 

Secondly, the algorithm ALBA defined in \cite{CoPa:non-dist} can be straightforwardly adapted to $\mathcal{L}_{MT\abla}$ and $\mathcal{L}_{MT>}$ and their algebraic and relational semantics; since the translations of all the axioms listed above are inductive, by the general theory, ALBA succeeds in eliminating the propositional variables occurring in them and in equivalently transforming their validity on frames into suitable conditions expressible in the predicate languages canonically associated with n-frames (resp.~c-frames). The ALBA runs on these axioms are reported in Section \ref{sec:ALBA runs}.


To further expand on how the correspondence results of Theorem \ref{theor:correspondence-noAlba} can be obtained as instances of algorithmic correspondence on two-sorted frames and their complex algebras, let $\mathbb{F}$ be an n-frame (resp.~a c-frame) and $\varphi \vdash \psi$ an  $\mathcal{L}_{\abla}$-sequent (resp.~$\mathcal{L}_{\abla}$-sequent). Let $\tau(\varphi \vdash \psi)$ denote $\tau_1(\varphi) \vdash \tau_2(\psi)$ or $\varphi^\tau \vdash \psi^\tau$ as appropriate. Let $ALBA(\tau(\varphi \vdash \psi) )$ denote the output of ALBA when run on $\tau(\varphi \vdash \psi)$, and $\mathsf{ST}(ALBA(\tau(\varphi \vdash \psi) ))$ be its standard translation in the appropriate predicate language of n-frames (resp.~c-frames). Then the following chain of equivalences holds:
\begin{center}
\begin{tabular}{c ll}
&$\mathbb{F}\Vdash \varphi \vdash \psi$\\
iff &$\mathbb{F}^\star \Vdash \tau(\varphi \vdash \psi) $& Proposition \ref{prop:consequence preserved and reflected} \\
iff &$(\mathbb{F}^\star )^+\models \tau(\varphi \vdash \psi) $  & def.~of validity on two sorted-frames \\
iff &$(\mathbb{F}^\star )^+\models ALBA(\tau(\varphi \vdash \psi) )$ & two-sorted correspondence\\
iff &$\mathbb{F}^\star \models \mathsf{ST}(ALBA(\tau(\varphi \vdash \psi) ))$ & \\
iff &$ \mathbb{F}\models \mathsf{ST}(ALBA(\tau(\varphi \vdash \psi) ))$ 
\end{tabular}
\end{center}
Let us concretely illustrate this proof pattern by applying it  to  the following axiom:
\begin{equation}
\label{eq:axiom C}\abla p \land \abla q \vdash \abla(p \land q).\end{equation}

Let $\mathbb{F}=(W,\nu)$ be a n-frame, and $\mathbb{F}^\star = (W,\mathcal{P} (W), R_\ni,$ $R_{\not\ni}, R_{\nu}, R_{\nu^c})$ be its associated two-sorted n-frame, where e.g.~$w R_\nu Z$ iff $Z\in \nu(w)$ and so on (full details are in Definition \ref{def:two-sortedFrames}). By Proposition \ref{prop:consequence preserved and reflected}, the validity of  axiom \eqref{eq:axiom C} on $\mathbb{F}$ is equivalent to  its translation 
\begin{equation}
\label{eq:axiom C transl}\xdianu \dboxni  p \land\xdianu \dboxni q \vdash \xboxnuc\ddianni (p \land q) \end{equation} being valid on $\mathbb{F}^\star$, which, by definition of satisfaction and validity in the two-sorted environment, is equivalent to the validity of axiom \eqref{eq:axiom C transl} on the complex algebra $(\mathbb{F}^\star )^+= (\mathcal{P}(W),\mathcal{P}\mathcal{P} (W), \dboxni,$ $\ddianni$ $, \xdianu, \xboxnuc)$.

According to Definition \ref{def:analytic inductive ineq}, axiom \eqref{eq:axiom C transl} is a ($\Omega,\epsilon$)-analytic inductive inequality for $p<_\Omega q$ and $\epsilon(p)= \epsilon(q)= 1$. Let us now run ALBA on axiom \eqref{eq:axiom C transl}.  In what follows we let $\mbf{i}_1$ and $ \mbf{i}_2$ be nominal variables of type $\mathsf{N}$ and $\mbf{m}$ be a co-nominal variable of type $\mathsf{N}$. This means that $\mbf{i}_1$ and $\mbf{i}_2$ are interpreted as --- and hence range in the set of --- atoms  of the second domain  $\mathcal{P}\mathcal{P} (W)$ of the perfect heterogeneous c-algebra $(\mathbb{F}^\star)^+ $ (i.e.~singleton subsets $\{Z\}$ for $Z\subseteq W$), while $\mbf{m}$ ranges over the set of  coatoms of $\mathcal{P}\mathcal{P} (W)$, and hence is interpreted as the collection of subsets $\{Z\}^c: = \{Y\subseteq W\mid Y\neq Z\}$ for an arbitrary $Z\subseteq W$.

As no preprocessing is needed, ALBA performs first approximation, which equivalently transforms  \[\forall p \forall q [\xdianu \dboxni  p \land\xdianu \dboxni q \leq \xboxnuc\ddianni (p \land q)]\]  into the following quasi-inequality: 
\[
 \forall p \forall q \forall \mbf{i}_1 \forall \mbf{i}_2 \forall \mbf{m}[(\mbf{i}_1 \le\dboxni  p  \ \& \ \mbf{i}_2  \le  \dboxni q \ \& \ \ddianni (p \land q) \le \mbf{m})\Rightarrow \xdianu \mbf{i}_1  \land\xdianu \mbf{i}_2 \leq \xboxnuc  \mbf{m}]. 
\]
Recall that 
$\xdiain$ and $\dboxni$ form a residuation pair. Hence, $\mbf{i}_1 \le\dboxni  p$ is equivalent to $\xdiain \mbf{i}_1 \le p$ and $ \mbf{i}_2  \le  \dboxni q$ is equivalent to $\xdiain \mbf{i}_2  \le  q$. Then the quasi inequality above is equivalent to the following quasi-inequality:
\[
 \forall p \forall q \forall \mbf{i}_1 \forall \mbf{i}_2 \forall \mbf{m}[(\mbf{i}_1 \le\dboxni  p  \ \& \ \xdiain \mbf{i}_2  \le  q \ \& \ \ddianni (p \land q) \le \mbf{m})\Rightarrow \xdianu \mbf{i}_1  \land\xdianu \mbf{i}_2 \leq \xboxnuc  \mbf{m}]. 
\]
The quasi inequality above is in Ackermann shape, hence the  Ackermann rule can be applied (cf.~\cite[Lemma 4.2]{CoPa:non-dist})  to eliminate all occurrences of $p$ and $q$, yielding the following (pure) quasi inequality in output
\[
 \forall \mbf{i}_1 \forall \mbf{i}_2 \forall \mbf{m}[ \ddianni ( \xdiain \mbf{i}_1 \land \xdiain \mbf{i}_2) \le \mbf{m}\Rightarrow \xdianu \mbf{i}_1  \land\xdianu \mbf{i}_2 \leq \xboxnuc  \mbf{m}],
\]
which, for the sake of convenience, applying adjunction, we equivalently rewrite as
\begin{equation}
\label{eq:alba output}
 \forall \mbf{i}_1 \forall \mbf{i}_2 \forall \mbf{m}[ \xdiain \mbf{i}_1 \land \xdiain \mbf{i}_2 \le \xboxnin \mbf{m}\Rightarrow \xdianu \mbf{i}_1  \land\xdianu \mbf{i}_2 \leq \xboxnuc  \mbf{m}]. 
\end{equation}
Let   $ ALBA(\tau(\abla p \land \abla q \vdash \abla(p \land q) ) )$ denote the quasi inequality above. The soundness of ALBA on perfect heterogeneous m-algebras and the validity of \eqref{eq:axiom C transl} on $(\mathbb{F}^\star )^+$ imply that $ ALBA(\tau(\abla p \land \abla q \vdash \abla(p \land q) ) )$  holds in $(\mathbb{F}^\star )^+$. The next step is to translate this quasi-inequality into a condition on $\mathbb{F}^\star$ expressible in its appropriate correspondence language. 

As discussed above, nominal and conominal variables  correspond to subsets of $W$. Moreover,  recall that the heterogeneous connectives $ \dboxni,$ $\ddianni$ $, \xdianu, \xboxnuc$ are interpreted  in $(\mathbb{F}^\star )^+$ as heterogeneous operations 
defined by the following assignments: for any $D \in \mathcal{P}(W)$ and $U \in \mathcal{P}\mathcal{P} (W)$ (cf.~Definition \ref{def:2sorted Kripke frame}),  
\[
\xboxnin U = (R_{\not \ni}[U^c])^c  \quad \xdiain U = R_\ni[U] \quad \xdianu U = R^{-1}_\nu [U] \quad \xboxnuc U = (R^{-1}_{\nu^c}[U^c])^c.
\]
Let $Z_1 ,Z_2,Z_3 \subseteq W$ and $\{Z_1\}, \{Z_2\}, \{Z_3\}^c$ be the interpretations of $\mbf{i}_1, \mbf{i_2}, \mbf{m}$, respectively. Then, writing $R_{\circ}[Z]$ for $R_{\circ}[\{Z\}]$ for any $\circ\in \{\ni, \not\ni, \nu, \nu^c\}$, we can  translate  \eqref{eq:alba output} as follows:
\begin{align*}
&\forall \mbf{i}_1 \forall \mbf{i}_2 \forall \mbf{m}[  \xdiain \mbf{i}_1 \land \xdiain \mbf{i}_2 \le \xboxnin \mbf{m}\Rightarrow \xdianu \mbf{i}_1  \land\xdianu \mbf{i}_2 \leq \xboxnuc  \mbf{m}]\\
= &\forall Z_1 \forall Z_2 \forall Z_3[ \xdiain \{Z_1\} \land \xdiain \{Z_2\} \le  \xboxnin\{Z_3\}^c\Rightarrow \xdianu \{Z_1\}  \land\xdianu \{Z_2\} \leq \xboxnuc \{Z_3\}^c]\\
= &\forall Z_1 \forall Z_2 \forall Z_3[ R_\ni [Z_1] \cap R_\ni[Z_2] \subseteq (R_{\not\ni}[\{Z_3\}^{cc}])^c\Rightarrow R^{-1}_\nu [Z_1]  \cap  R^{-1}_\nu[ Z_2] \subseteq( R^{-1}_{\nu^c}[ \{Z_3\}^{cc}])^c]\\
= &\forall Z_1 \forall Z_2 \forall Z_3[ R_\ni [Z_1] \cap R_\ni[Z_2] \subseteq (R_{\not\ni}[Z_3])^c\Rightarrow R^{-1}_\nu [Z_1]  \cap  R^{-1}_\nu[ Z_2] \subseteq( R^{-1}_{\nu^c}[ Z_3])^c].
%
\end{align*}
Thus, we have obtained
\[\mathbb{F}^\star \models \forall Z_1 \forall Z_2 \forall Z_3[ R_\ni [Z_1] \cap R_\ni[Z_2] \subseteq (R_{\not\ni}[Z_3])^c\Rightarrow R^{-1}_\nu [Z_1]  \cap  R^{-1}_\nu[ Z_2] \subseteq( R^{-1}_{\nu^c}[ Z_3])^c].
\]
The final step is to translate this condition into a condition on $\mathbb{F}$. Recalling the definitions of $R_{\not \ni},R_\ni,R_\nu, R_{\nu^c}$ in Definition \ref{def:two-sortedFrames}, it is easy to see that for any $Z\subseteq W$,
\[
R_\ni[Z]= Z  =  (R_{\not \ni}[Z])^c      \quad \text{ and }\quad  R^{-1}_\nu [Z] =\{w\in W\mid Z\in \nu(w)\}  =   (R^{-1}_{\nu^c}[Z])^c.
\]
Hence, we get:
\[\mathbb{F} \models \forall Z_1 \forall Z_2 \forall Z_3[ Z_1 \cap Z_2\subseteq Z_3\Rightarrow \forall x [(Z_1\in \nu (x)  \ \&\   Z_2\in \nu(x)) \Rightarrow  Z_3\in \nu(x)]],
\]
which, by uncurrying and then currying again, and suitably distributing quantifiers, is equivalent to 
\[\mathbb{F} \models \forall Z_1 \forall Z_2\forall x [(Z_1\in \nu (x)  \ \&\   Z_2\in \nu(x)) \Rightarrow \forall Z_3[Z_1 \cap Z_2\subseteq Z_3 \Rightarrow  Z_3\in \nu(x)]],
\]
which is equivalent to
\[\mathbb{F} \models \forall Z_1 \forall Z_2\forall x [(Z_1\in \nu (x)  \ \&\   Z_2\in \nu(x)) \Rightarrow Z_1 \cap Z_2 \in \nu(x)]]:
\]
Indeed, for the top-to-bottom direction, take $Z_3= Z_1 \cap Z_2$. Conversely, assume that $Z_1 \cap Z_2 \subseteq Z_3$, and that $Z_1 \in \nu(x)$ and $Z_2 \in \nu(x)$. Then, the assumption implies that  $Z_1 \cap Z_2 \in \nu(x)$. Since $\nu(x)$ is upwards-closed, $Z_1 \cap Z_2 \subseteq Z_3$ implies that  $Z_3 \in \nu (x)$. This completes the algorithmic proof of item C  of Theorem \ref{theor:correspondence-noAlba}. The remaining items can be obtained by similar arguments. In Appendix \ref{sec:ALBA runs} we collect the relevant ALBA runs and translations of their output.

Finally, the tools of unified correspondence can be used also for computing analytic rules corresponding to analytic inductive axioms in the  given two-sorted languages, so to obtain analytic calculi for some axiomatic extensions of the basic monotone modal logic and basic conditional logic as an application of the theory developed in \cite{greco2016unified}. This treatment yields the analytic calculi defined in the next section.

\section{Proper display calculi for non-normal logics}
\label{sec:calculi}
In this section we introduce proper multi-type display calculi for $\mathbf{L}_\abla$ and $\mathbf{L}_>$ and their axiomatic extensions generated by the analytic axioms in the table above. 

\noindent \emph{Languages. } \ \ The language $\mcl{L}_{DMT\abla}$ of the  calculus D.MT$\abla$ for $\mathbf{L}_\abla$ is defined as follows:
{\small{
\begin{center}
\begin{tabular}{l}
$\mathsf{S}\left\{\begin{array}{l}
A::= p \mid \xtop \mid \xbot \mid \neg A \mid A \xand A \mid \xdianu \alpha \mid \xboxnuc \alpha \\
X ::= A\mid \XTOP \mid \XBOT \mid \XNEG X \mid X \XAND X \mid X \XOR X \mid \XDIANU \Gamma \mid \XBOXNUC \Gamma \mid \XDIAIN \Gamma \mid \XBOXNIN \Gamma \\
\end{array} \right.$
 \\
 \\
$\mathsf{N}\left\{\begin{array}{l}
\alpha ::= \dboxni A \mid \ddianni A \\
\Gamma ::= \alpha \mid \DTOP \mid \DBOT \mid \DNEG \Gamma \mid \Gamma \DAND \Gamma \mid \Gamma \DOR \Gamma \mid \DBOXNI X \mid \DDIANNI X \mid \DBOXUN X \mid \DDIAUNC X \\
\end{array} \right.$
 \\
\end{tabular}
\end{center}
}}
The language $\mcl{L}_{DMT>}$ of the calculus D.MT$>$ for $\mathbf{L}_>$ is defined as follows:

{\small{
\begin{center}
\begin{tabular}{l}
$\mathsf{S}\left\{\begin{array}{l}
A::= p \mid \xtop \mid \xbot \mid \neg A \mid A \xand A \mid \alpha \mtra A \\
X ::= A\mid \XTOP \mid \XBOT \mid \XNEG X \mid X \XAND X \mid X \XOR X \mid \XDIAIN \Gamma \mid \Gamma \MTRA X \mid \Gamma \MTAND X \mid \XRHDNIN \Gamma\\
\end{array} \right.$
 \\
 \\
$\mathsf{N}\left\{\begin{array}{l}
\alpha ::= \dboxni A \mid \drhdnni A \mid \alpha \dand \alpha \\
\Gamma ::= \alpha \mid \DTOP \mid \DBOT \mid \DNEG \Gamma \mid \Gamma \DAND \Gamma \mid \Gamma \DOR \Gamma \mid \DBOXNI X \mid \DRHDNNI X \mid X \MTBRA X\\
\end{array} \right.$
 \\
\end{tabular}
\end{center}
}}

\noindent\emph{Multi-type display calculi.}\ \ In what follows, we use $X, Y, W, Z$ as structural $\mathsf{S}$-variables, and $\Gamma, \Delta, \Sigma, \Pi$ as structural $\mathsf{N}$-variables. 

\noindent {\bf Propositional base.}\ \ The calculi D.MT$\abla$ and D.MT$>$ share the rules listed below.
\begin{itemize}
\item Identity and Cut:
\end{itemize}
{\fns
\begin{center}
\begin{tabular}{ccc}
\AXC{\rule[0mm]{0mm}{0.21cm}}
\LL{\scs $Id_\mathsf{S}$}
\UI$p \fCenter p$
\DP
 & 
\AX $X \fCenter A$
\AX $A \fCenter Y\rule[0mm]{0mm}{0.25cm}$
\RL{\scs $Cut_\mathsf{S}$}
\BI $X \fCenter Y$
\DP
 & 
\AX$\Gamma \fCenter \alpha$
\AX$\alpha \fCenter \Delta\rule[0mm]{0mm}{0.25cm}$
\RL{\scs $Cut_\mathsf{N}$}
\BI$\Gamma \fCenter \Delta$
\DP
\end{tabular}
\end{center}
}
\begin{itemize}
\item Pure $\mathsf{S}$-type display rules:
\end{itemize}
{\fns
\begin{center}
\begin{tabular}{rlrl}
\mc{4}{c}{
\AXC{\ }
\LL{\scs $\bot$}
\UI$\xbot \fCenter \XBOT$
\DP
 \ \ 
\AXC{\ }
\RL{\scs $\top$}
\UI$\XTOP \fCenter \xtop$
\DP}
\\

 & & & \\

\AX $\XNEG X \fCenter Y$
\LeftLabel{\scriptsize $gal_\mathsf{S}$}
\doubleLine
\UI$\XNEG Y \fCenter X$
\DisplayProof
 & 
\AX$X \fCenter \XNEG Y$
\RightLabel{\scriptsize $gal_\mathsf{S}$}
\doubleLine
\UI$Y \fCenter \XNEG X$
\DisplayProof
 & 
\AX$X \XAND Y \fCenter Z$
\LeftLabel{\scriptsize $res_\mathsf{S}$}
\doubleLine
\UI $Y \fCenter \XNEG X \XOR Z$
\DP
 & 
\AX$X \fCenter Y \XOR Z $
\RightLabel{\scriptsize $res_\mathsf{S}$}
\doubleLine
\UI$\XNEG Y \XAND X \fCenter Z$
\DP
 \\
\end{tabular}
\end{center}
}
\begin{itemize}
\item Pure $\mathsf{N}$-type display rules:
\end{itemize}
{\fns
\begin{center}
\begin{tabular}{rlrl}
\AX$\DNEG \Gamma \fCenter \Delta$
\LeftLabel{\scriptsize $gal_\mathsf{N}$}
\doubleLine
\UI$\DNEG \Delta \fCenter \Gamma$
\DP
 & 
\AX$\Gamma \fCenter \DNEG \Delta$
\RL{\scriptsize $gal_\mathsf{N}$}
\doubleLine
\UI$\Delta \fCenter \DNEG \Gamma$
\DP
\ \ & \ \ 
\AX$\Gamma \DAND \Delta \fCenter \Sigma$
\LeftLabel{\scriptsize $res_\mathsf{N}$}
\doubleLine
\UI$\Delta \fCenter \DNEG \Gamma \DOR \Sigma$
\DP
 & 
\AX$\Gamma \fCenter \Delta \DOR \Sigma$
\RightLabel{\scriptsize $res_\mathsf{N}$}
\doubleLine
\UI$\DNEG \Delta \DAND \Gamma \fCenter \Sigma$
\DP
\end{tabular}
\end{center}
}
\begin{itemize}
\item Pure $\mathsf{S}$-type structural rules: 
\end{itemize}
{\fns
\begin{center}
\begin{tabular}{crl}
\AX$X \fCenter Y $
\LL{\scs $cont_\mathsf{S}$}
\doubleLine
\UI$\XNEG Y \fCenter \XNEG X$
\DP
 & 
\AX$X \fCenter Y$
\LL{\scs $\XTOP$}
\doubleLine
\UI$X \XAND \XTOP \fCenter Y$
\DP
 & 
\AX$X \fCenter Y$
\RL{\scs $\XBOT$}
\doubleLine
\UI$X \fCenter Y \XOR \XBOT$
\DP
 \\
\end{tabular}
\end{center}

\begin{center}
\begin{tabular}{rlrl}
\AX$X \fCenter Y$
\LL{\scs $W_\mathsf{S}$}
\UI$X \XAND Z \fCenter Y$
\DP
 & 
\AX$X \fCenter Y$
\RL{\scs $W_\mathsf{S}$}
\UI$X \fCenter Y \XOR Z$
\DP
 & 
\AX$X \XAND X \fCenter Y$
\LL{\scs $C_\mathsf{S}$}
\UI$X \fCenter Y$
\DP
 & 
\AX$X \fCenter Y \XOR Y$
\RL{\scs $C_\mathsf{S}$}
\UI$X \fCenter Y$
\DP
 \\

 & & & \\

\AX$X \XAND Y \fCenter Z$
\LL{\scs $E_\mathsf{S}$}
\UI$Y \XAND X \fCenter Y$
\DP
 & 
\AX$X \fCenter Y \XOR Z$
\RL{\scs $E_\mathsf{S}$}
\UI$X \fCenter Z \XOR Y$
\DP
 & 
\AX$X \XAND (Y \XAND Z) \fCenter W$
\LL{\scs $A_\mathsf{S}$}
\UI$(X \XAND Y) \XAND Z \fCenter W$
\DP
 & 
\AX$W \fCenter X \XAND (Y \XAND Z)$
\RL{\scs $A_\mathsf{S}$}
\UI$W \fCenter (X \XAND Y) \XAND Z$
\DP
 \\
\end{tabular}
\end{center}
 }

\begin{itemize}
\item Pure $\mathsf{N}$-type structural rules: 
\end{itemize}
{\fns
\begin{center}
\begin{tabular}{crl}
\AX$\Gamma \fCenter \Delta$
\LL{\scs $cont_\mathsf{N}$}
\doubleLine
\UI$\DNEG \Delta \fCenter \DNEG \Gamma$
\DP
 & 
\AX $\Gamma \fCenter \Delta$
\LL{\scs $\DTOP$}
\doubleLine
\UI $\Gamma \DAND \DTOP \fCenter \Delta$
\DisplayProof
 & 
\AX$\Gamma \fCenter \Delta $
\RL{\scs $\DBOT$}
\doubleLine
\UI$\Gamma \fCenter \Delta \DOR \DBOT$
\DP
 \\
\end{tabular}
\end{center}

\begin{center}
\begin{tabular}{rlrl}
\AX$\Gamma \fCenter \Delta$
\LL{\scs $W_\mathsf{N}$}
\UI$\Gamma \DAND \Pi \fCenter \Delta$
\DP
 & 
\AX$\Gamma \fCenter \Delta$
\RL{\scs $W_\mathsf{N}$}
\UI$\Gamma \fCenter \Delta \DOR \Pi$
\DP
 & 
\AX$\Gamma \DAND \Gamma \fCenter \Delta$
\LL{\scs $C_\mathsf{N}$}
\UI$\Gamma \fCenter \Delta$
\DP
 & 
\AX$\Gamma \fCenter \Delta \DOR \Delta$
\RL{\scs $C_\mathsf{N}$}
\UI$\Gamma \fCenter \Delta$
\DP
 \\

 & & & \\

\AX$\Gamma \DAND \Delta \fCenter \Pi$
\LL{\scs $E_\mathsf{N}$}
\UI$Y \DAND \Gamma \fCenter \Delta$
\DP
 & 
\AX$\Gamma \fCenter \Delta \DOR \Pi$
\RL{\scs $E_\mathsf{N}$}
\UI$\Gamma \fCenter \Pi \DOR \Delta$
\DP
 & 
\AX$\Gamma \DAND (\Delta \DAND \Pi) \fCenter \Sigma$
\LL{\scs $A_\mathsf{N}$}
\UI$(\Gamma \DAND \Delta) \DAND \Pi \fCenter \Sigma$
\DP
 & 
\AX$\Sigma \fCenter \Gamma \DAND (\Delta \DAND \Pi)$
\RL{\scs $A_\mathsf{N}$}
\UI$\Sigma \fCenter (\Gamma \DAND \Delta) \DAND \Pi$
\DP
 \\
\end{tabular}
\end{center}
 }

%

\begin{itemize}
\item Pure $\mathsf{S}$-type logical rules:
\end{itemize}
{\fns
\begin{center}
\begin{tabular}{rlrl}
\AX$\XNEG A \fCenter X$
\LL{\scs $\xneg$}
\UI$\xneg A \fCenter X$
\DP
 & 
\AX$X \fCenter \XNEG A$
\RL{\scs $\xneg$}
\UI$X \fCenter \xneg A$
\DP
 & 
\AX$A \XAND B \fCenter X$
\LL{\scs $\xand$}
\UI$A \xand B \fCenter X$
\DP
 & 
\AX$X \fCenter A$
\AX$Y \fCenter B$
\RL{\scs $\xand$}
\BI$X \XAND Y \fCenter A \xand B$
\DP
 \\
\end{tabular}
\end{center}
}

\noindent {\bf Monotonic modal logic.} \ \ D.MT$\abla$ also includes the rules listed below.
\begin{itemize}
\item Multi-type display rules:
\end{itemize}
{\fns
\begin{center}
\begin{tabular}{ccc}
\AX$\XDIANU \Gamma \fCenter X$
\doubleLine
\LL{\scs $\XDIANU\DBOXUN$}
\UI$\Gamma \fCenter \DBOXUN X$
\DP
 & 
\AX$\DDIAUNC X \fCenter \Gamma$
\doubleLine
\LL{\scs $\DDIAUNC\XBOXNUC$}
\UI$X \fCenter \XBOXNUC \Gamma$
\DP
 & 
\AX$\XDIAIN \Gamma \fCenter X$
\doubleLine
\LL{\scs $\XDIAIN\DBOXNI$}
\UI$\Gamma \fCenter \DBOXNI X$
\DP
 \\
 & & \\
\mc{3}{c}{
\AX$\XDIAIN \Gamma \fCenter X$
\doubleLine
\LL{\scs $\XDIAIN\DBOXNI$}
\UI$\Gamma \fCenter \DBOXNI X$
\DP
\ \ \ \ \ 
\AX$\DDIANNI X \fCenter \Gamma$
\doubleLine
\LL{\scs $\DDIANNI\XBOXNIN$}
\UI$X \fCenter \XBOXNIN \Gamma$
\DP}
 \\
\end{tabular}
\end{center}
}
\begin{itemize}
\item Logical rules for multi-type connectives:
\end{itemize}
{\fns
\begin{center}
\begin{tabular}{rlrlrl}
\AX$\XDIANU \alpha \fCenter X$
\LL{\scs $\xdianu$}
\UI$\xdianu \alpha \fCenter X$
\DP
 & 
\AX$\Gamma \fCenter \alpha$
\RL{\scs $\xdianu$}
\UI$\XDIANU \Gamma \fCenter \xdianu \alpha$
\DP
 & 
\AX$\alpha \fCenter \Gamma$
\LL{\scs $\xboxnuc$}
\UI$\xboxnuc \alpha \fCenter \XBOXNUC \Gamma$
\DP
 & 
\AX$X \fCenter \XBOXNUC \alpha$
\RL{\scs $\xboxnuc$}
\UI$X \fCenter \xboxnuc \alpha$
\DP
 \\

 & & & \\

\AX$\DDIANNI A \fCenter \Gamma$
\LL{\scs $\ddianni$}
\UI$\ddianni A \fCenter \Gamma$
\DP
 & 
\AX$X \fCenter A$
\RL{\scs $\ddianni$}
\UI$\DDIANNI X \fCenter \ddianni A$
\DP
 & 
\AX$A \fCenter X$
\LL{\scs $\dboxni$}
\UI$\dboxni A \fCenter \DBOXNI X$
\DP
 & 
\AX$\Gamma \fCenter \DBOXNI A$
\RL{\scs $\dboxni$}
\UI$\Gamma \fCenter \dboxni A$
\DP
 \\
\end{tabular}
\end{center}
}

\noindent {\bf Conditional logic.} \ \ D.MT$>$ includes left and right logical rules for $\dboxni$, the display postulates $\XDIAIN\DBOXNI$ and the rules listed below.
\begin{itemize}
\item Multi-type display rules:
\end{itemize}
{\fns
\begin{center}
\begin{tabular}{ccc}
\AX$X \fCenter \Gamma \MTRA Y$
\doubleLine
\LL{\scs $\MTAND\MTRA$}
\UI$\Gamma \MTAND X \fCenter Y$
\doubleLine
\DP
\ \ & \ \ 
\AX$\Gamma \fCenter X \MTBRA Y$
\doubleLine
\RL{\scs $\MTBRA\MTRA$}
\UI$X \fCenter \Gamma \MTRA Y$
\DP
\ \ & \ \ 
\AX$X \fCenter \XRHDNIN \Gamma$
\doubleLine
\RL{\scs $\XRHDNIN\DRHDNNI$}
\UI$\Gamma \fCenter \DRHDNNI X$
\DP
 \\
\end{tabular}
\end{center}
}
\begin{itemize}
\item Logical rules for multi-type connectives and pure $\mathsf{G}$-type logical rules:
\end{itemize}
{\fns
\begin{center}
\begin{tabular}{rlrl}
\AX$\Gamma \fCenter \alpha$
\AX$A \fCenter X$
\LL{\scs $\mtra$}
\BI$\alpha \mtra A \fCenter \Gamma \MTRA X$
\DP
 & 
\AX$X \fCenter \alpha \MTRA A$
\RL{\scs $\mtra$}
\UI$X \fCenter \alpha \mtra A$
\DP
 & 
\AX$X \fCenter A$
\LL{\scs $\drhdnni$}
\UI$\drhdnni A \fCenter \DRHDNNI X$
\DP
 & 
\AX$\Gamma \fCenter \DRHDNNI A$
\RL{\scs $\drhdnni$}
\UI$\Gamma \fCenter \drhdnni A$
\DP
 \\
 & & & \\
\mc{4}{c}{\AX$\alpha \DAND \beta \fCenter \Gamma$
\LL{\scs $\dand$}
\UI$\alpha \dand \beta \fCenter \Gamma$
\DP
 \ \ 
\AX$\Gamma \fCenter \alpha$
\AX$\Delta \fCenter \beta$
\RL{\scs $\dand$}
\BI$\Gamma \DAND \Delta \fCenter \alpha \dand \beta$
\DP}
\\
\end{tabular}
\end{center}
}

\noindent {\bf Axiomatic extensions.} \ \ Each rule is labelled with the name of its corresponding axiom. 
{\fns
\begin{center}
\begin{tabular}{c}
\AX$ \DDIANNI \XTOP \fCenter \Gamma$
\LL{\scs N}
\UI$ \XTOP \fCenter \XBOXNUC \Gamma$
\DP
 \ \ \ \ 
\AX$\Delta \fCenter \DRHDNNI \XDIAIN \Gamma$
\AX$\XDIAIN \Gamma \fCenter X$
\LL{\scs ID}
\BI$\XTOP \fCenter (\Gamma\DAND \Delta)\MTRA X$
\DP
 \ \ \ \ 
\AX$ \DDIANNI (\XDIAIN\Gamma\XAND\XDIAIN\Delta) \fCenter \Theta$
\LL{\scs C}
\UI$ \XDIANU\Gamma \XAND\XDIANU\Delta \fCenter \XBOXNUC \Theta $
\DP
 \\

 \\

\AX$\Gamma \fCenter \DBOXNI \XNEG\XDIAIN \Delta$
\LL{\scs D}
\UI$ \XDIANU \Delta \fCenter \XNEG\XDIANU \Gamma$
\DP
 \ \ \ \ 
\AX$\Gamma \fCenter \DBOXNI\XBOT$
\LL{\scs P}
\UI$\XTOP \fCenter \XNEG\XDIANU \Gamma$
\DP
 \ \ \ \ 
\AXC{$\Gamma \vdash \DBOXNI \XRHDNIN \Delta \quad\quad X \vdash \XRHDNIN \Delta \quad\quad Y \vdash Z$}
\LL{\scs CS}
\UIC{$X\XAND Y \vdash (\Gamma\DAND\Delta)\MTRA Z$}
\DP
 \\

 \\

\AXC{$ \Pi\vdash \DRHDNNI \XDIAIN \Gamma  \quad  \Pi\vdash \DRHDNNI \XDIAIN \Theta  \quad  \Delta\vdash \DRHDNNI \XDIAIN \Gamma  \quad  \Delta\vdash \DRHDNNI \XDIAIN \Theta \quad Y\vdash X$}
\LL{\scs CEM}
\UIC{$ \XTOP\vdash ((\Gamma\DAND \Delta)\MTRA X)\XOR((\Theta\DAND \Pi)\MTRA \XNEG Y)$}
\DP
\ \ \ \ \ \ \ \ \ 
\AX$\Gamma \fCenter \DBOXNI X$
\LL{\scs T}
\UI$\XDIANU \Gamma \fCenter X$
\DP
 \\
 
 \\

\AXC{$ \Gamma \vdash  \DBOXNI \XRHDNIN \Delta\quad \Gamma \vdash   \DBOXNI Y \quad \Theta  \vdash  \DBOXNI \XRHDNIN \Pi \quad  \Theta \vdash   \DBOXNI X$}
\LL{\scs CN}
\UIC{$ \XTOP\vdash ((\Gamma\DAND \Delta)\MTRA X)\XOR((\Theta\DAND \Pi)\MTRA Y)$}
\DP
\end{tabular}
\end{center}
}


\section{Properties}
\label{sec:properties}

 The calculi introduced above are proper (cf.~\cite{wansing2013displaying,greco2016unified}), and hence the  general theory of proper multi-type display calculi guarantees that they enjoy {\em cut elimination} and {\em subformula property} \cite{TrendsXIII}. 

Let $H_m$ (resp. $H_c$) be the class of all perfect heterogeneous m-algebras (resp. perfect heterogeneous c-algebras). Given a set of analytic sequents $R$, the extension of D.MT$\nabla$ (resp. D.MT$>$) with inference rules obtained by running ALBA on $R$ is denoted by D.MT$\nabla R$ (resp. D.MT$>R$). The subclass of $H_m$ (resp. $H_c$) defined by $R$ is denoted by $H_m(R)$ (resp. $H_c(R)$).

\subsection{Soundness}
\label{ssec:soundness}
To show the soundness of the rules of D.MT$\nabla R$ (resp. $D.MT>R$) w.r.t. $H_m(R)$ (resp. $H_c(R))$, it suffices to show that the interpretation of each rule in D.MT$\nabla R$  (resp. $D.MT>R$) is valid in $H_m(R)$ (resp. $H_c(R)$). The soundness of the rules in D.MT$\nabla$ and D.MT$>$ follows from the definitions of $H_m$ and $H_c$, respectively. And the soundness of the rules from $R$ follows from the soundness of ALBA rules on members of $H_m$ (resp. $H_c$), and the ALBA runs reported in the appendix. Specifically, in what follows, for any perfect m-algebra (resp. c-algebra) $\mathbb{H}: = (\mathbb{A}, \mathbb{B}, ...)$,  let $x$  range over $\mathbb{A}$ and $\gamma$, $\delta$, $\theta$ range over $\mathbb{B}$. Then  the rules on the left-hand side of the squiggly arrows below are interpreted as the quasi-inequalities on the
right-hand side:

  \begin{center}
\begin{tabular}{rcl}
\AxiomC{$\XDIAIN \Gamma \fCenter X$}
\doubleLine
\UnaryInfC{$\Gamma \fCenter \DBOXNI X$}
\DisplayProof
&$\quad\rightsquigarrow\quad$&
$\forall \gamma \forall x [\xdiain \gamma \le x \Leftrightarrow \gamma \le \dboxni x]$
\end{tabular}
\end{center}

  \begin{center}
\begin{tabular}{rcl}
\AX$ \DDIANNI (\XDIAIN\Gamma\XAND\XDIAIN\Delta) \fCenter \Theta$
\UI$ \XDIANU\Gamma \XAND\XDIANU\Delta \fCenter \XBOXNUC \Theta $
\DP
&$\quad\rightsquigarrow\quad$&
$\forall \gamma \forall \delta  \forall \theta [ \ddianni (\xdiain \gamma \xand \xdiain \delta) \le  \theta \Rightarrow \xdianu \gamma \xand \xdianu \delta  \le  \xboxnuc  \theta]$
\end{tabular}
\end{center}

  \begin{center}
\begin{tabular}{rcl}
\AX$\Gamma \fCenter \DBOXNI\XBOT$
\UI$\XTOP \fCenter \XNEG\XDIANU \Gamma$
\DP
&$\quad\rightsquigarrow\quad$&
$\forall \gamma [ \gamma \le \dboxni \bot \Rightarrow \top \le \xneg \xdianu \gamma ]$
\end{tabular}
\end{center}

The validity of $\forall \gamma \forall x [\xdiain \gamma \le x \Leftrightarrow \gamma \le \dboxni x]$ follows from the fact that $\xdiain $ and $ \dboxni$ form a residuation pair in $\mathbb{H}$. The validity of the quasi-inequalities corresponding to axioms C and P in $H_m(\{C\})$ and $H_m(\{P\})$ respectively follows from the validity-preserving ALBA runs reported in the appendix. We report below on the validity-preserving ALBA run for C. 
{\small{
\begin{flushleft}
\begin{tabular}{clr}
\mc{3}{l}{C.\ \, $\mathbb{H}\models \nabla p \land \nabla q \to \nabla ( p \land q)  \ \rightsquigarrow\ \langle \nu \rangle [\ni] p \land \langle \nu \rangle [\ni] q \subseteq [\nu^c] \langle \not \ni \rangle (p \land q)$} \\
\hline
    & \mc{2}{l}{$\mathbb{H} \models \langle \nu \rangle [\ni] p \land \langle \nu \rangle [\ni] q \subseteq [\nu^c] \langle \not \ni \rangle (p \land q)$} \\
iff & \mc{2}{l}{$\mathbb{H} \models\forall  \gamma \forall \delta \forall \theta \forall p q[ \gamma\subseteq [\ni] p \ \& \   \delta\subseteq [\ni] q \ \&\  \langle \not \ni \rangle(p \land q)\subseteq  \theta\Rightarrow \langle \nu \rangle  \gamma \land \langle \nu \rangle  \delta\subseteq [\nu^c]\theta]$}{first approx.} \\
iff & \mc{2}{l}{$\mathbb{H} \models\forall  \gamma \forall \delta \forall \theta \forall p q[ \langle \in \rangle \gamma\subseteq p \ \& \ \langle \in \rangle   \delta\subseteq  q \ \&\  \langle \not \ni \rangle(p \land q)\subseteq  \theta \Rightarrow \langle \nu \rangle  \gamma \land \langle \nu \rangle  \delta\subseteq [\nu^c] \theta]$}{Residuation} \\
iff & $\mathbb{H} \models\forall  \gamma \forall \delta \forall \theta[   \langle \not \ni \rangle(\langle \in \rangle \gamma \land \langle \in \rangle   \delta)\subseteq  \theta\Rightarrow \langle \nu \rangle  \gamma \land \langle \nu \rangle  \delta\subseteq [\nu^c] \theta]$ & $(\star)$ Ackermann \\
\end{tabular}
\end{flushleft}
}}
  
\subsection{Completeness}
\label{ssec:completeness}
As discussed  above, the algorithmic correspondence perspective on the theory of analytic calculi (here in their incarnation as  ``proper display calculi'') allows for a uniform justification of the soundness of analytic rules in terms of the soundness of the algorithm ALBA used to generate them. These benefits extend also to the uniform justification of the  {\em completeness} of proper display calculi w.r.t.~the logics they are intended to capture. Specifically, in \cite{chen2020}, 
an effective procedure is introduced for generating cut free derivations of the translations of each rule and analytic inductive axiom (of any normal lattice expansion signature) in the corresponding proper display calculus. Below, we illustrate this effective procedure by applying it to the analytic axioms of the present setting.

\begin{itemize}
\item[N.] \ $ \abla \top \ \rightsquigarrow \ \xboxnuc\ddianni \top$ \ \ \ \ \ \ \ \ \ \ \ \ \ \ \ \ P. \ $\neg \abla \bot \ \rightsquigarrow\ \neg \xdianu \dboxni \bot$ \ \ \ \ T. \ $\abla A \to A \ \rightsquigarrow\ \xdianu \dboxni A \vdash A$
\end{itemize}
{\fns
\begin{center}
\begin{tabular}{ccc}
\AX$\XTOP \fCenter \xtop$
\UI$\DDIANI \XTOP \fCenter \ddiani \xtop$
\LL{\scs N}
\UI$\XTOP \fCenter \XBOXNUC \ddiani \xtop$
\DP
\ \ & \ \ 
\AX$\xbot \fCenter \XBOT$
\UI$\dboxni \xbot \fCenter \DBOXNI \XBOT$
\LL{\scs P}
\UI$\XTOP \fCenter \XNEG \dboxni \xbot$
\DP
\ \ & \ \ 
\AX$A \fCenter A$
\UI$\dboxni A \fCenter \DBOXNI A$
\LL{\scs T}
\UI$\XDIANU \dboxni A \fCenter A$
\DP
 \\
\end{tabular}
\end{center}
}
\begin{itemize}
\item[ID.] $A > A \ \rightsquigarrow\ (\dboxni A \wedge\drhdnni A) \mtra A$
\end{itemize}
{\fns
\begin{center}
\begin{tabular}{c}
\AX$A \fCenter A$
\UI$\drhdnni A \fCenter \DRHDNNI A$
\UI$A \fCenter \XRHDNIN \drhdnni A$
\UI$\dboxni A \fCenter \DBOXNI \XRHDNIN \drhdnni A$
\UI$\XDIAIN \dboxni A \fCenter \XRHDNIN \drhdnni A$
\UI$\drhdnni A \fCenter \DRHDNNI \XDIAIN \dboxni A$
\AX$A \fCenter A$
\UI$\dboxni A \fCenter \DBOXNI A$
\UI$\XDIAIN \dboxni A \fCenter A$
\LL{\scs ID}
\BI$\XTOP \fCenter (\DBOXNI A \DAND \DRHDNNI A) \MTRA A$
\DP
\end{tabular}
\end{center}
}

\begin{itemize}
\item[CS.] $(A \wedge B) \to (A > B) \ \rightsquigarrow\ (A \wedge B) \vdash (\dboxni A \dand \drhdnni A) \mtra B$
\end{itemize}

{\fns
\begin{center}
\begin{tabular}{c}
\AX$A \fCenter A$
\UI$\drhdnni A \fCenter \DRHDNNI A$
\UI$A \fCenter \XRHDNIN \drhdnni A$
\UI$\dboxni A \fCenter \DBOXNI \XRHDNIN \drhdnni A$
\AX$A \fCenter A$
\UI$\drhdnni A \fCenter \DRHDNNI A$
\UI$A \fCenter \XRHDNIN \drhdnni A$
\AX$B \fCenter B$
\LL{\scs CS}
\TI$A \XAND B \fCenter (\DBOXNI A \DAND \DRHDNNI A) \MTRA B$
\DP 
 \\
\end{tabular}
\end{center}
}

\begin{itemize}
\item[CEM.] $(A > B) \vee (A > \neg B) \ \rightsquigarrow\ (\dboxni A \dand \drhdnni A)\mtra B \vee (\dboxni A \dand \drhdnni A) \mtra \neg B$
\end{itemize}
{\fns
\begin{center}
\begin{center}
$\mkern-40mu$
\AxiomC{$A\fCenter A$}
\UnaryInfC{$\dboxni A \fCenter  \DBOXNI A$}
\UnaryInfC{$\XDIAIN \dboxni A \fCenter  A$}
\UnaryInfC{$\drhdnni A \fCenter \DRHDNNI \XDIAIN \dboxni A $}
\AxiomC{$A\fCenter A$}
\UnaryInfC{$\dboxni A \fCenter  \DBOXNI A$}
\UnaryInfC{$ \XDIAIN \dboxni A \fCenter  A$}
\UnaryInfC{$\drhdnni A \fCenter \DRHDNNI \XDIAIN \dboxni A$}
\AxiomC{$A\fCenter A$}
\UnaryInfC{$\dboxni A \fCenter  \DBOXNI A$}
\UnaryInfC{$ \XDIAIN \dboxni A \fCenter  A$}
\UnaryInfC{$\drhdnni A \fCenter \DRHDNNI \XDIAIN \dboxni A$}
\AxiomC{$A\fCenter A$}
\UnaryInfC{$\dboxni A \fCenter  \DBOXNI A$}
\UnaryInfC{$ \XDIAIN \dboxni A \fCenter  A$}
\UnaryInfC{$\drhdnni A \fCenter \DRHDNNI \XDIAIN \dboxni A$}
\AxiomC{$B\fCenter B$}
\LL{\scs CEM}
\QuinaryInfC{$\XTOP \fCenter (\dboxni A \DAND \drhdnni A) \MTRA B \XOR (\dboxni A \DAND \drhdnni A) \MTRA \XNEG B$}
\DisplayProof
\end{center}
\end{center}
}

\begin{itemize}
\item[C.] $\abla A \land \abla B \to \abla(A \land B) \rightsquigarrow \xdianu \dboxni A \land\xdianu \dboxni B \vdash \xboxnuc\ddianni (A \land B)$ 
\item[D.] $\abla A \to \neg \abla \neg A \rightsquigarrow \xdianu \dboxni A \vdash \neg \xdianu \dboxni \neg A$
\end{itemize}
{\fns
\begin{center}
\begin{tabular}{cc}

\AX$A \fCenter A$
\UI$\dboxni A \fCenter \dboxni A$
\UI$\XDIAIN \dboxni A \fCenter A$
\AX$B \fCenter B$
\UI$\dboxni B \fCenter \dboxni B$
\UI$\XDIAIN \dboxni B \fCenter B$
\BI$\XDIAIN \dboxni A \XAND \XDIAIN \dboxni B \fCenter A \xand B$
\UI$\DDIANNI (\XDIAIN \dboxni A \XAND \XDIAIN \dboxni B) \fCenter \ddianni (A \xand B)$
\LL{\scs C}
\UI$\XDIANU \dboxni A \XAND \XDIANU \dboxni B \fCenter \XBOXNUC \ddianni (A \xand B)$
\DP
 & 
\AX$A \fCenter A$
\UI$\dboxni A \fCenter \DBOXNI A$
\UI$\XDIAIN \dboxni A \fCenter A$
\UI$\xneg A \fCenter \XNEG \XDIAIN \dboxni A$
\UI$\dboxni \xneg A \fCenter \DBOXNI \XNEG \XDIAIN \dboxni A$
\LL{\scs D}
\UI$\XDIANU \dboxni A \fCenter \XNEG \XDIANU \dboxni \xneg A$
\DP

 \\
\end{tabular}
\end{center}
}

\begin{itemize}
\item[CN.] $(A > B) \vee (B > A) \ \rightsquigarrow\ (\dboxni A \dand \drhdnni A)\mtra B \vee (\dboxni B \dand \drhdnni B) \mtra A$
\end{itemize}
{\fns
\begin{center}
\begin{center}
$\mkern-40mu$
\AxiomC{$A\fCenter A$}
\UnaryInfC{$ \drhdnni A  \vdash   \DRHDNNI A $}
\UnaryInfC{$ A  \vdash   \XRHDNIN \drhdnni A $}
\UnaryInfC{$\dboxni A  \vdash  \DBOXNI \XRHDNIN \drhdnni A $}
\AxiomC{$A\fCenter A$}
\UnaryInfC{$\dboxni A  \vdash   \DBOXNI A $}
\AxiomC{$B\fCenter B$}
\UnaryInfC{$ \drhdnni B  \vdash   \DRHDNNI B $}
\UnaryInfC{$ B  \vdash   \XRHDNIN \drhdnni B $}
\UnaryInfC{$ \dboxni B  \vdash  \DBOXNI \XRHDNIN  \drhdnni B$}
\AxiomC{$B\fCenter B$}
\UnaryInfC{$\dboxni B \vdash   \DBOXNI B$}
\LL{\scs CN}
\QuaternaryInfC{$\XTOP\vdash ((\dboxni A \DAND \drhdnni A)\MTRA B)\XOR((\dboxni B\DAND  \drhdnni B)\MTRA A)$}
\DisplayProof
\end{center}
\end{center}
}

The (translations of the) rules M, RCEA and RCK$_n$ are derivable as follows.

\begin{itemize}
\item[M.] 
\AX$A \fCenter B$
\UI$\nabla A \fCenter \nabla B$
\DP
$\rightsquigarrow$
\AX$A \fCenter B$
\UI$\xdianu \dboxni A \fCenter \xdianu \dboxni B$
\DP
\end{itemize}

{\fns
\begin{center}
\begin{tabular}{c}
\AX$A \fCenter B$
\UI$  \dboxni A \fCenter  \DBOXNI B$
\UI$  \dboxni A \fCenter  \dboxni B$
\UI$ \XDIANU \dboxni A \fCenter  \xdianu \dboxni B$
\UI$ \xdianu \dboxni A \fCenter  \xdianu \dboxni B$
\DP 
 \\
\end{tabular}
\end{center}
}

\begin{itemize}
\item[RCEA.] 
\AXC{$A \leftrightarrow B$}
\UIC{$(A > C)\leftrightarrow (B > C)$}
\DP
$\rightsquigarrow$
\AX$A \fCenter B$
\AX$B \fCenter A$
\BI$(\dboxni A \dand \drhdnni A)\mtra C \fCenter (\dboxni B \dand \drhdnni B)\mtra C$
\DP
\end{itemize}

{\fns
\begin{center}
\begin{tabular}{c}
\AX$B \fCenter A$
\UI$\dboxni B \fCenter \DBOXNI A$
\UI$\dboxni B \fCenter \dboxni A$
\AX$A \fCenter B$
\UI$\drhdnni B \fCenter \DRHDNNI A$
\UI$\drhdnni B \fCenter \drhdnni A$
\BI$\dboxni B \DAND \drhdnni B \fCenter \dboxni \varphi \dand \drhdnni A$
\UI$\dboxni B \dand \drhdnni B \fCenter \dboxni \varphi \dand \drhdnni A$
\AX$C \fCenter C$
\BI$ (\dboxni A \dand \drhdnni A) \mtra C \fCenter (\dboxni B \dand \drhdnni B) \MTRA C$
\UI$ (\dboxni A \dand \drhdnni A) \mtra C \fCenter (\dboxni B \dand \drhdnni B) \mtra C$
\DP 
 \\
\end{tabular}
\end{center}
}

\begin{itemize}
\item [$RCK_n.$] 
\AXC{$A_1 \land \ldots \land A_n \to B$}
\UIC{$(C > A_1) \land \ldots \land (C > A_n) \to (C > B)$}
\DP   \\
$\phantom{AAAAAAAAA} \rightsquigarrow$   
\AX$A_1 \xand \ldots \xand A_n \fCenter B$
\UI$(\dboxni C \dand \drhdnni C) \mtra A_1 \xand \ldots \xand (\dboxni \chi \dand \drhdnni C) \mtra A_n \fCenter (\dboxni C \dand \drhdnni C) \mtra B$
\DP
\end{itemize}

\noindent To show that the translation of $RCK_n$ is derivable, let us preliminarily show that $(\dboxni C \dand \drhdnni C) \MTAND (\dboxni C \dand \drhdnni C) \mtra A_1 \xand (\dboxni C \dand \drhdnni C) \mtra A_2 \fCenter A_1 \xand A_2$ is derivable.

{\scriptsize
\begin{center}
\begin{tabular}{c}
$\mkern-140mu$
\AX$C \fCenter C$
\UI$\dboxni C \fCenter \DBOXNI C$
\UI$\dboxni C \fCenter \dboxni C$
\AX$C \fCenter C$
\UI$\drhdnni C \fCenter \DRHDNNI C$
\UI$\drhdnni C \fCenter \drhdnni C$
\BI$ \dboxni C \DAND \drhdnni C \fCenter \dboxni C \dand \drhdnni C$
\UI$ \dboxni C \dand \drhdnni C \fCenter \dboxni C \dand \drhdnni C$
\AX$A_1 \fCenter A_1$
\BI$(\dboxni C \dand \drhdnni C) \mtra A_1 \fCenter (\dboxni C \dand \drhdnni C) \MTRA A_1$
\LL{\scs $W_S$}
\UIC{$(\dboxni C \dand \drhdnni C) \mtra A_1 \XAND (\dboxni C \dand \drhdnni C) \mtra A_2 \fCenter (\dboxni C \dand \drhdnni C) \MTRA A_1$}
\UIC{$(\dboxni C \dand \drhdnni C) \mtra A_1 \xand (\dboxni C \dand \drhdnni C) \mtra A_2 \fCenter (\dboxni C \dand \drhdnni C) \MTRA A_1$}
\UIC{$(\dboxni C \dand \drhdnni C) \MTAND (\dboxni C \dand \drhdnni C) \mtra A_1 \xand (\dboxni C \dand \drhdnni C) \mtra A_2 \fCenter A_1$}
\AX$C \fCenter C$
\UI$\dboxni C \fCenter \DBOXNI C$
\UI$\dboxni C \fCenter \dboxni C$
\AX$C \fCenter C$
\UI$\drhdnni C \fCenter \DRHDNNI C$
\UI$\drhdnni C \fCenter \drhdnni C$
\BI$ \dboxni C \DAND \drhdnni C \fCenter \dboxni C \dand \drhdnni C$
\UI$ \dboxni C \dand \drhdnni C \fCenter \dboxni C \dand \drhdnni C$
\AX$A_2 \fCenter A_2$
\BI$(\dboxni C \dand \drhdnni C) \mtra A_1 \fCenter (\dboxni C \dand \drhdnni C) \MTRA A_2$
\LL{\scs $W_S$}
\UIC{$(\dboxni C \dand \drhdnni C) \mtra A_1 \XAND (\dboxni C \dand \drhdnni C) \mtra A_2 \fCenter (\dboxni C \dand \drhdnni C) \MTRA A_2$}
\UIC{$(\dboxni C \dand \drhdnni C) \mtra A_1 \xand (\dboxni C \dand \drhdnni C) \mtra A_2 \fCenter (\dboxni C \dand \drhdnni C) \MTRA A_2$}
\UIC{$(\dboxni C \dand \drhdnni C) \MTAND (\dboxni C \dand \drhdnni C) \mtra A_1 \xand (\dboxni C \dand \drhdnni C) \mtra A_2 \fCenter A_2$}
\BIC{$((\dboxni C \dand \drhdnni C) \MTAND (\dboxni C \dand \drhdnni C) \mtra A_1 \xand (\dboxni C \dand \drhdnni C) \mtra A_2) \XAND ((\dboxni C \dand \drhdnni C) \MTAND (\dboxni C \dand \drhdnni C) \mtra A_1 \xand (\dboxni C \dand \drhdnni C) \mtra A_2) \fCenter A_1 \xand A_2$}
\LL{\scs $C_S$}
\UIC{$(\dboxni C \dand \drhdnni C) \MTAND (\dboxni C \dand \drhdnni C) \mtra A_1 \xand (\dboxni C \dand \drhdnni C) \mtra A_2 \fCenter A_1 \xand A_2$}
\DP
 \\
\end{tabular}
\end{center}
}

\noindent Iterating the previous derivation $n-1$ times (where the specific instantiation of $W_S$ is suitably chosen so as to derive the specific instantiation of the end sequent), we obtain the left premise of the following derivation, which provides the required derivation of the conclusion of  $RCK_n$ from its premise. 

{\fns
\begin{center}
\begin{tabular}{c}
$\mkern-40mu$
\AXC{$\vdots$}
\UI$(\dboxni C \dand \drhdnni C) \MTAND ((\dboxni C \dand \drhdnni C) \mtra A_1 \xand \ldots \xand (\dboxni C \dand \drhdnni C) \mtra A_n) \fCenter A_1 \xand \ldots \xand A_n$
\AX$A_1 \xand \ldots \xand A_n \fCenter B$
\RL{\scs $Cut_\mathsf{S}$}
\BI$(\dboxni C \dand \drhdnni C) \MTAND ((\dboxni C \dand \drhdnni C) \mtra A_1 \xand \ldots \xand (\dboxni C \dand \drhdnni C) \mtra A_n) \fCenter B$
\UI$(\dboxni C \dand \drhdnni C) \mtra A_1 \xand \ldots \xand (\dboxni C \dand \drhdnni C) \mtra A_n \fCenter (\dboxni C \dand \drhdnni C) \MTRA B$
\UI$(\dboxni C \dand \drhdnni C) \mtra A_1 \xand \ldots \xand (\dboxni C \dand \drhdnni C) \mtra A_n \fCenter (\dboxni C \dand \drhdnni C) \mtra B$
\DP 
 \\
\end{tabular}
\end{center}
}

\subsection{Conservativity}\label{ssec: conservativity}
To argue that the calculi introduced in Section \ref{sec:calculi}  conservatively extend their corresponding Hilbert systems, we follow the standard proof strategy discussed in \cite{greco2016unified,linearlogPdisplayed}. Let  $\vdash_{\mbf{L}}$ denote the syntactic consequence relation arising from  Hilbert systems, and $\models_{\textrm{H}}$ denote the semantic consequence relation arising from heterogeneous Kripke frames and their complex (heterogeneous) algebras. We need to show that, for all formulas $A$ and $B$ of the original language of the Hilbert system, if $\tau(A \vdash B)$ is derivable in a display calculus,  then  $A \vdash_{\mbf{L}} B$. This claim can be proved using  the following facts: (a) the rules of display calculi are sound w.r.t.~heterogeneous Kripke frames and their complex (heterogeneous) algebras (cf.~Section \ref{ssec:soundness});  (b) Hilbert systems are complete w.r.t.~their respective class of algebras; and (c)  homogenous algebras are equivalently presented as heterogeneous  algebras (cf.~Section \ref{Heterogeneous presentation}), so that the semantic consequence relations arising from each type of structures preserve and reflect the translation (cf.~Proposition \ref{prop:consequence preserved and reflected}). Then, let $A\vdash B$ be an entailment between formulas of  the language of the original Hilbert systems. If  $\tau(A\vdash B)$ is derivable in a display calculus, then, by (a),  $\models_{\textrm{H}} \tau(A\vdash B)$. By (c), this implies that $A \models_{V} B$, where $\models_{\textrm{V}}$ denotes the semantic consequence relation arising from m-algebras or c-algebras. By (b), this implies that $A\vdash_{\mbf{L}} B$, as required.

\section{Conclusions and further directions}
\label{sec: Conclusions}
\paragraph{Present contributions.} In the present paper, we have proposed a semantic analysis of two well-known non-normal logics (monotone modal logic and conditional logic), and used it  to introduce both a uniform correspondence-theoretic framework encompassing and significantly extending various well-known Sahlqvist-type results for these logics, and a proof-theoretic framework modularly capturing not only the basic logics but also an infinite class of axiomatic extensions of the basic monotone modal logic and conditional logic which includes well-known logics such as coalitional logic \cite{pauly2001logic} and preferential logic \cite{xu2006some}. 
The correspondence-theoretic and the proof-theoretic frameworks are closely connected with each other, both because they stem from the same semantic analysis, and because, more fundamentally,  they instantiate results, tools and insights developed at the interface of correspondence theory and structural proof theory \cite{greco2016unified}. This line of research can be naturally extended in various ways, and in what follows we list some natural further directions.

\paragraph{A modular framework for classical modal logic.} In the present paper, we have considered monotone modal logic and conditional logic because this choice made it possible to address a significant diversity of order-theoretic behaviour of the non-normal connectives with a minimal set of examples: namely a unary monotone operator and a binary operator which is normal (finitely meet-preserving) in its second coordinate and arbitrary in the first coordinate. A natural further direction concerns the systematic application of these techniques to wider classes of non-normal logics. Even restricting attention to the signature of $\mathcal{L}_\abla$, a natural direction concerns developing a modular account of classical modal logic  \cite{chellas1980modal} and its (monotone, regular) extensions up to normal modal logic. Of course the translations employed in the present paper for monotone modal logic do not account for classical modal logic, because monotonicity is in-built in these translations. The question is then whether one can express monotonicity as an (analytic) inductive condition under a translation similar to the one used in the non-normal coordinate of the conditional logic operator $>$.

\paragraph{From Boolean to distributive lattice-based non-normal logics.} The semantic analysis of the present paper hinges on the embedding of well-known state-based semantics (monotone neighbourhood frames, selection functions) into two-sorted classical Kripke frames and their discrete dualities with perfect (heterogeneous) Boolean algebras. Pivoting on more general discrete dualities, such as Birkhoff's discrete duality between perfect distributive lattices and posets, one can develop the systematic theory of e.g.~the {\em non-normal} counterparts of positive modal logic \cite{dunn1995positive,celani1997new} or intuitionistic modal logics \cite{FisSer1977,FisServi84,ono1977some}. In particular, it would be interesting to investigate the applicability of  the present approach  for capturing the lattice of non-normal intuitionistic modal logics introduced in \cite{DalGreOli2020}. 

\paragraph{Neighbourhood and selection functions as formal tools for context-relativization and category-formation.} We plan to investigate alternative (intuitive) interpretations of neighbourhood and conditional frames in order to expand the realm of possible applications. 

\noindent A natural option would be to consider a neighbourhood as a context relativising the interpretation of a term. An obvious application would be in lexical semantics (see e.g.~\cite{FregeInSpace14}) where the meaning of a word is often context-dependent. 

\noindent A second option would be to consider neighbourhoods as categories. Again, an obvious application would be in computational linguistics (see e.g.~\cite{lam61}) where each word is assigned to a syntactical category depending on the role it plays in the formation of grammatically correct sentences or phrases. 

\noindent Notice that a word can occur in different contexts or it can be assigned to different categories. Therefore, one may consider generalizations of the framework with multiple (weighed) neighbourhood functions or relations as a way to represent (probabilistic) distributions in a data set. 

\noindent In many machine learning approaches, a system needs both positive and negative evidence. For example, a classification system needs examples for each class that it is capable of predicting; if the classification is binary (e.g.~the system tries to decide whether an email is spam or not), it needs to have positive and negative examples. This generalises to multiple classes (e.g.~given a music song, predict the genre of that song). Therefore, one may consider (generalisations of) bi-neighbourhood frames (see e.g.~\cite{DalOliNeg18}), in which sets of pairs of neighbourhoods provide independent positive and negative evidence. 

\noindent Finally, each neighbourhood can be endowed with additional structure in order to capture specific behaviour. This refinement would build a bridge between the literature in non-normal modal logics and the literature on so-called modal logics for structural control in linguistics and logic (see e.g.~\cite{KurMoo97,Moo96,Gir87,linearlogPdisplayed}).

\appendix
\section{Analytic inductive inequalities}
\label{sec:analytic inductive ineq}
In the present section, we  specialize the definition of {\em analytic inductive inequalities} (cf.\ \cite{greco2016unified}) to the  multi-type languages $\mathcal{L}_{MT\abla}$ and $\mathcal{L}_{MT>}$ reported below.
{\small
\begin{center}
$\begin{array}{lll}
\mathsf{S} \ni A::= p  \mid \top \mid \bot \mid \neg A \mid A \land A \mid \xdianu \alpha\mid \xboxnuc\alpha &\quad\quad&\mathsf{S} \ni A::= p  \mid \top \mid \bot \mid \neg A \mid A \land A \mid  \alpha\mtra A
\\
\mathsf{N} \ni \alpha ::=  \dtop\mid \dbot \mid {\sim} \alpha \mid \alpha \dand \alpha \mid \dboxni A\mid \ddianni A &\quad\quad&\mathsf{N} \ni \alpha ::=  \dtop\mid \dbot \mid {\sim} \alpha \mid \alpha \dand \alpha \mid \dboxni A\mid \drhdnni A.
\end{array}$
\end{center}
}
An {\em order-type} over $n\in \mathbb{N}$  is an $n$-tuple $\epsilon\in \{1, \partial\}^n$. If $\epsilon$ is an order type, $\epsilon^\partial$ is its {\em opposite} order type; i.e.~$\epsilon^\partial(i) = 1$ iff $\epsilon(i)=\partial$ for every $1 \leq i \leq n$.
The connectives of the language above are grouped together  into the  families $\mathcal{F}: = \mathcal{F}_{\mathsf{S}}\cup \mathcal{F}_{\mathsf{N}}\cup \mathcal{F}_{\textrm{MT}}$ and $\mathcal{G}: = \mathcal{G}_{\mathsf{S}}\cup \mathcal{G}_{\mathsf{N}} \cup  \mathcal{G}_{\textrm{MT}}$, defined as follows:
\begin{center}
\begin{tabular}{lcl}
$\mathcal{F}_{\mathsf{S}}: = \{\xneg\}$&&$ \mathcal{G}_{\mathsf{S}} = \{\xneg\}$\\
$\mathcal{F}_{\mathsf{N}}: = \{\dneg\}$ && $\mathcal{G}_{\mathsf{N}}: = \{\dneg\}$\\
$\mathcal{F}_{\textrm{MT}}: = \{\xdianu, \ddianni \}$ && $\mathcal{G}_{\textrm{MT}}: = \{\dboxni, \xboxnuc, \mtra, \drhdnni\}$\\

\end{tabular}
\end{center}
For any $f\in \mathcal{F}$  (resp.\ $g\in \mathcal{G}$), we let $n_f\in \mathbb{N}$ (resp.~$n_g\in \mathbb{N}$) denote the arity of $f$ (resp.~$g$), and the order-type $\epsilon_f$ (resp.~$\epsilon_g$) on $n_f$ (resp.~$n_g$)  indicate whether the $i$th coordinate of $f$ (resp.\ $g$) is positive ($\epsilon_f(i) = 1$,  $\epsilon_g(i) = 1$) or  negative ($\epsilon_f(i) = \partial$,  $\epsilon_g(i) = \partial$). 
%
				%
				\begin{defn}[\textbf{Signed Generation Tree}]
					\label{def: signed gen tree}
					The \emph{positive} (resp.\ \emph{negative}) {\em generation tree} of any $\mathcal{L}_\textrm{MT}$-term $s$ is defined by labelling the root node of the generation tree of $s$ with the sign $+$ (resp.\ $-$), and then propagating the labelling on each remaining node as follows:
					For any node labelled with $\ell\in \mathcal{F}\cup \mathcal{G}$ of arity $n_\ell$, and for any $1\leq i\leq n_\ell$, assign the same (resp.\ the opposite) sign to its $i$th child node if $\epsilon_\ell(i) = 1$ (resp.\ if $\epsilon_\ell(i) = \partial$). Nodes in signed generation trees are \emph{positive} (resp.\ \emph{negative}) if are signed $+$ (resp.\ $-$).
				\end{defn}
				For any term $s(p_1,\ldots p_n)$, any order type $\epsilon$ over $n$, and any $1 \leq i \leq n$, an \emph{$\epsilon$-critical node} in a signed generation tree of $s$ is a leaf node $+p_i$ with $\epsilon(i) = 1$ or $-p_i$ with $\epsilon(i) = \partial$. An $\epsilon$-{\em critical branch} in the tree is a branch ending in an $\epsilon$-critical node. For any term $s(p_1,\ldots p_n)$ and any order type $\epsilon$ over $n$, we say that $+s$ (resp.\ $-s$) {\em agrees with} $\epsilon$, and write $\epsilon(+s)$ (resp.\ $\epsilon(-s)$), if every leaf in the signed generation tree of $+s$ (resp.\ $-s$) is $\epsilon$-critical.
				 We will also write $+s'\prec \ast s$ (resp.\ $-s'\prec \ast s$) to indicate that the subterm $s'$ inherits the positive (resp.\ negative) sign from the signed generation tree $\ast s$. Finally, we will write $\epsilon(s') \prec \ast s$ (resp.\ $\epsilon^\partial(s') \prec \ast s$) to indicate that the signed subtree $s'$, with the sign inherited from $\ast s$, agrees with $\epsilon$ (resp.\ with $\epsilon^\partial$).	
		\begin{defn}[\textbf{Good branch}]
					\label{def:good:branch}
					Nodes in signed generation trees will be called \emph{$\Delta$-adjoints}, \emph{syntactically left residual (SLR)}, \emph{syntactically right residual (SRR)}, and \emph{syntactically right adjoint (SRA)}, according to the specification given in Table \ref{Join:and:Meet:Friendly:Table}.
					A branch in a signed generation tree $\ast s$, with $\ast \in \{+, - \}$, is called a \emph{good branch} if it is the concatenation of two paths $P_1$ and $P_2$, one of which may possibly be of length $0$, such that $P_1$ is a path from the leaf consisting (apart from variable nodes) only of PIA-nodes 
					and $P_2$ consists (apart from variable nodes) only of Skeleton-nodes.


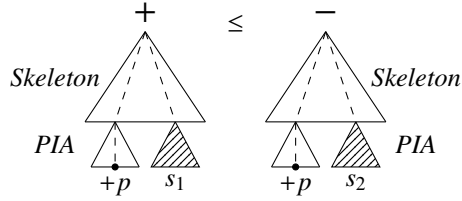
\begin{table}
\begin{center}
														
							\begin{tabular}{| c | c |}
								\hline
								Skeleton &PIA\\
								\hline
								$\Delta$-adjoints & SRA \\
								\begin{tabular}{ c c c c c c c}
									$+\ $ &\ $\xor\ $ &\ $\dor\ $ \\ 
									$-\ $ & $\xand$ &$\dand$ \\ 
								\end{tabular}
								&
								\begin{tabular}{c c c c c c c c c  }
									$+\ $ & \ $\xand$\  &\ $\dand$\ &\ $\dboxni$ \ & \ $\xboxnuc$ \ & \ $\mtra$ \ & \ $\drhdnni$\ &\ $\rdneg$\ &\ $\dneg$\ \\
									$-\ $ & \  $\xor $\  &\ $\dor$\  &\ $\xdianu$ \ &\ $ \ddianni$\ & \ $\rdneg$ \ &\ $\dneg$\    \\
								\end{tabular}
								\\
								\hline
								SLR &SRR\\
								\begin{tabular}{c c c c c c c c c}
									$+\ $ & \ $\xand$\  &\ $\dand$\ &\ $\xdianu$ \ &\ $ \ddianni$\ & \ $\rdneg$ \ &\ $\dneg$\   \\
									$-\ $ & \ $\xor$\  &\ $\dor$\  &\ $\dboxni$ \ & \ $\xboxnuc$ \ & \ $\mtra$ \ & \ $\drhdnni$\ &\ $\rdneg$\ &\ $\dneg$\  \\
								\end{tabular}
								&\begin{tabular}{c c c}
									 $+\ $ &\ $\xor$\  &\ $\dor$\ \\
									$-\ $ &\ $\xand$ \ &\ $\dand$ \ \\
								\end{tabular}
								\\
								\hline

							\end{tabular}
\vspace{0.5em}	
\caption{Skeleton and PIA nodes.}\label{Join:and:Meet:Friendly:Table}
\end{center}
\end{table}

\begin{center}
\begin{tikzpicture}[scale=0.4]
		\draw (-5,-1.5) -- (-3,1.5) node[above]{\Large$+$} ;
		\draw (-5,-1.5) -- (-1,-1.5) ;
		\draw (-3,1.5) -- (-1,-1.5);
		\draw (-6,0) node{Skeleton} ;
		\draw[dashed] (-3,1.5) -- (-4,-1.5);
		\draw[dashed] (-3,1.5) -- (-2,-1.5);
		\draw (-4,-1.5) --(-4.8,-3);
		\draw (-4.8,-3) --(-3.2,-3);
		\draw (-3.2,-3) --(-4,-1.5);
		\draw[dashed] (-4,-1.5) -- (-4,-3);
		\draw[fill] (-4,-3) circle[radius=.1] node[below]{$+p$};
		\draw
		(-2,-1.5) -- (-2.8,-3) -- (-1.2,-3) -- (-2,-1.5);
		\fill[pattern=north east lines]
		(-2,-1.5) -- (-2.8,-3) -- (-1.2,-3);
		\draw (-2,-3.5)node{$s_1$};
		\draw (-6,-2.25) node{PIA} ;
		\draw (0,1.8) node{$\leq$};
		\draw (5,-1.5) -- (3,1.5) node[above]{\Large$-$} ;
		\draw (5,-1.5) -- (1,-1.5) ;
		\draw (3,1.5) -- (1,-1.5);
		\draw (6,0) node{Skeleton} ;
		\draw[dashed] (3,1.5) -- (4,-1.5);
		\draw[dashed] (3,1.5) -- (2,-1.5);
		\draw (2,-1.5) --(2.8,-3);
		\draw (2.8,-3) --(1.2,-3);
		\draw (1.2,-3) --(2,-1.5);
		\draw[dashed] (2,-1.5) -- (2,-3);
		\draw[fill] (2,-3) circle[radius=.1] node[below]{$+p$};
		\draw
		(4,-1.5) -- (4.8,-3) -- (3.2,-3) -- (4, -1.5);
		\fill[pattern=north east lines]
		(4,-1.5) -- (4.8,-3) -- (3.2,-3) -- (4, -1.5);
		\draw (4,-3.5)node{$s_2$};
		\draw (6,-2.25) node{PIA} ;
		\end{tikzpicture}
\end{center}
\end{defn}

				\begin{defn}[\textbf{Analytic inductive inequalities}]
	\label{def:analytic inductive ineq}
					For any order type $\epsilon$ and any irreflexive and transitive relation $<_\Omega$ on $p_1,\ldots p_n$, the signed generation tree $*s$ $(* \in \{-, + \})$ of an $\mathcal{L}_{MT}$ term $s(p_1,\ldots p_n)$ is \emph{analytic $(\Omega, \epsilon)$-inductive} if
					\begin{enumerate}
						\item  every branch of $*s$ is good (cf.\ Definition \ref{def:good:branch});
						\item for all $1 \leq i \leq n$, every SRR-node occurring in  any $\epsilon$-critical branch with leaf $p_i$ is of the form $ \circledast(s, \beta)$ or $ \circledast(\beta, s)$, where the critical branch goes through $\beta$ and 
						\begin{enumerate}
							\item $\epsilon^\partial(s) \prec \ast s$ (cf.\ discussion before Definition \ref{def:good:branch}), and
							%
							\item $p_k <_{\Omega} p_i$ for every $p_k$ occurring in $s$ and for every $1\leq k\leq n$.
						\end{enumerate}

					\end{enumerate}
					
					 An inequality $s \leq t$ is \emph{analytic $(\Omega, \epsilon)$-inductive} if the signed generation trees $+s$ and $-t$ are analytic $(\Omega, \epsilon)$-inductive. An inequality $s \leq t$ is \emph{analytic inductive} if is analytic $(\Omega, \epsilon)$-inductive for some $\Omega$ and $\epsilon$.
				\end{defn}
								


\section{Algorithmic proof of Theorem \ref{theor:correspondence-noAlba}}
\label{sec:ALBA runs}

In what follows, we show that the correspondence results collected in Theorem \ref{theor:correspondence-noAlba} can be retrieved as instances of a suitable multi-type version of algorithmic correspondence  for normal logics (cf.~\cite{CoGhPa14,CoPa:non-dist}), hinging on the usual order-theoretic properties of the algebraic interpretations of the logical connectives, while admitting nominal variables of two sorts. For the sake of enabling a swift translation into the language of m-frames and c-frames, we write nominals directly as singletons, and, abusing notation, we quantify over the elements defining these singletons. These computations also serve to prove that each analytic structural rule is sound on the heterogeneous perfect algebras validating its correspondent axiom. In the computations relative to each analytic axiom, the line marked with $(\star)$ marks the quasi-inequality that interprets the corresponding analytic rule. This computation does {\em not} prove the equivalence between the axiom and the rule, since the variables occurring in each starred quasi-inequality are restricted rather than arbitrary. However, the proof of soundness is completed by observing that all ALBA rules in the steps above the marked inequalities are (inverse) Ackermann and adjunction rules, and hence are sound also when arbitrary variables replace (co-)nominal variables. 


{\small{
\begin{flushleft}
\begin{tabular}{@{}clr c clr}
\mc{3}{l}{N.\ \, $\mathbb{H}\models \nabla \top \ \rightsquigarrow\ \top\subseteq [\nu^c] \langle \not \ni \rangle \top$} & \ \ \ \ \ \ \ \ \ & \mc{3}{l}{\ P. \ $\mathbb{H}\models \neg \nabla \bot\ \rightsquigarrow\ \top\subseteq \neg \langle \nu \rangle [\ni ] \bot$} \\
\cline{1-3}
\cline{5-7}
     & $\top \subseteq [\nu^c] \langle \not \ni \rangle \top $ & & & & $\top \subseteq\neg \langle \nu \rangle [\ni ] \bot$ & \\
iff \ & \mc{2}{l}{$\forall X \forall w [\langle \not \ni \rangle \top \subseteq \{X\}^c  \Rightarrow \{w\} \subseteq [\nu^c] \{X\}^c]$} & & iff \ & \mc{2}{l}{$ \forall X [ X \subseteq [\ni]\bot \Rightarrow T \subseteq \neg \langle \nu \rangle  X]$} \\
\mc{3}{r}{$(\star)$ first. app.} & & \mc{3}{r}{$(\star)$ first. app.} \\
iff \ & \mc{2}{l}{$\forall X \forall w[X = W  \Rightarrow \{w\} \subseteq [\nu^c] \{X\}^c)$} & & iff \ & $W \subseteq\neg \langle \nu \rangle [\ni ] \emptyset$  & \\
\mc{3}{r}{($\langle \ni \rangle \top =  \{W\}^c$)} & & & & \\
iff \ & $\forall w[ \{w\} \subseteq [\nu^c] \{W\}^c]$  & & & iff \ & $W \subseteq\neg \langle \nu \rangle \{ \emptyset \}$  & $[\ni ] \emptyset = \{Z\subseteq W\mid Z\subseteq \emptyset\}$ \\
iff \ & $\forall w[\{w\} \subseteq (R_{\nu^c}^{-1}[W])^c]$  & & & iff \ & $W \subseteq \{w \in W \mid  w R_\nu \emptyset\}^c$  & \\
iff \ & $\forall w[ \{w\} \subseteq R_{\nu}^{-1}[W]]$  & & & iff \ & $\forall w[ \emptyset \not \in \nu(w)]$.  & \\
iff \ & $\forall w[ W\in \nu (w) ]$ & & & & & \\
\end{tabular}
\end{flushleft}
 }

{\small{
\begin{flushleft}
\begin{tabular}{clr}
\mc{3}{l}{C.\ \, $\mathbb{H}\models \nabla p \land \nabla q \to \nabla ( p \land q)  \ \rightsquigarrow\ \langle \nu \rangle [\ni] p \land \langle \nu \rangle [\ni] q \subseteq [\nu^c] \langle \not \ni \rangle (p \land q)$} \\
\hline
    & \mc{2}{l}{$\langle \nu \rangle [\ni] p \land \langle \nu \rangle [\ni] q \subseteq [\nu^c] \langle \not \ni \rangle (p \land q)$} \\
iff & \mc{2}{l}{$\forall  Z_1 Z_2 Z_3 \forall p q[ \{Z_1\}\subseteq [\ni] p \ \& \   \{Z_2\} \subseteq [\ni] q \ \&\  \langle \not \ni \rangle(p \land q)\subseteq  \{Z_3\}^c\Rightarrow \langle \nu \rangle  \{Z_1\} \land \langle \nu \rangle  \{Z_2\} \subseteq [\nu^c]\{Z_3\}^c]$} \\
\mc{3}{r}{first approx.} \\
iff & \mc{2}{l}{$\forall  Z_1 Z_2 Z_3 \forall p q[ \langle \in \rangle \{Z_1\}\subseteq p \ \& \ \langle \in \rangle   \{Z_2\} \subseteq  q \ \&\  \langle \not \ni \rangle(p \land q)\subseteq  \{Z_3\}^c \Rightarrow \langle \nu \rangle  \{Z_1\} \land \langle \nu \rangle  \{Z_2\} \subseteq [\nu^c] \{Z_3\}^c]$} \\
\mc{3}{r}{Residuation} \\
iff & $\forall  Z_1 \forall  Z_2 \forall Z_3[   \langle \not \ni \rangle(\langle \in \rangle \{Z_1\} \land \langle \in \rangle   \{Z_2\} )\subseteq  \{Z_3\}^c\Rightarrow \langle \nu \rangle  \{Z_1\} \land \langle \nu \rangle  \{Z_2\} \subseteq [\nu^c] \{Z_3\}^c]$ & $(\star)$ Ackermann \\
iff & $\forall  Z_1 \forall  Z_2 \forall Z_3[   (\langle \in \rangle \{Z_1\} \land \langle \in \rangle   \{Z_2\} )\subseteq [\not \in] \{Z_3\}^c\Rightarrow \langle \nu \rangle  \{Z_1\} \land \langle \nu \rangle  \{Z_2\} \subseteq [\nu^c] \{Z_3\}^c]$ & Residuation \\
iff & \mc{2}{l}{$\forall  Z_1 \forall  Z_2 \forall Z_3[\forall x(xR_\in Z_1\ \& \ xR_\in Z_2\Rightarrow \lnot  xR_{\notin}Z_3)\Rightarrow\forall x(xR_\nu Z_1\ \&\ xR_\nu Z_2\Rightarrow \lnot xR_{\nu^c}Z_3)]$} \\
\mc{3}{r}{Standard translation} \\
iff & \mc{2}{l}{$\forall  Z_1 \forall  Z_2 \forall Z_3[\forall x(x\in Z_1\ \& \ x\in Z_2\Rightarrow x\in Z_3)\Rightarrow\forall x(Z_1\in \nu(x)\ \& \ Z_2\in\nu(x)\Rightarrow Z_3\in\nu(x))]$} \\
\mc{3}{r}{Relations interpretation} \\
iff &$\forall  Z_1 \forall  Z_2 \forall Z_3[Z_1\cap Z_2\subseteq Z_3\Rightarrow\forall x(Z_1\in \nu(x)\ \& \ Z_2\in\nu(x)\Rightarrow Z_3\in\nu(x))]$ & \\
iff &$\forall  Z_1 \forall  Z_2 \forall x[Z_1\in \nu(x)\ \& \ Z_2\in\nu(x)\Rightarrow Z_1\cap Z_2\in\nu(x))]$. & Monotonicity \\
\end{tabular}
\end{flushleft}
}}

{\small{
\begin{center}
\begin{tabular}{clr}
\mc{3}{l}{4'.\ \, $\mathbb{H}\models \nabla p\to \nabla \nabla p\ \rightsquigarrow\ \langle \nu\rangle [\ni] p\subseteq [\nu^c] \langle \not \ni \rangle [\nu^c] \langle \not \ni \rangle  p$} \\
\hline
    \ & $\langle \nu\rangle [\ni] p\subseteq [\nu^c] \langle \not \ni \rangle [\nu^c] \langle \not \ni \rangle p$ & \\
iff \ & $\forall Z_1\forall x' \forall p [ \{Z_1\}  \subseteq [\ni] p \ \&\ [\nu^c] \langle \not \ni \rangle [\nu^c] \langle \not \ni \rangle p\subseteq \{x'\}^c)\Rightarrow \langle \nu \rangle  \{Z_1\}  \subseteq  \{x'\}^c]$ & first approx. \\
iff \ & $\forall  Z_1 \forall x'\forall p [\langle \in \rangle  \{Z_1\} \subseteq  p \ \&\ [\nu^c] \langle \not \ni \rangle [\nu^c] \langle \not \ni \rangle p\subseteq  \{x'\}^c)\Rightarrow \langle \nu \rangle  \{Z_1\}  \subseteq  \{x'\}^c]$ & Residuation \\
iff \ & $\forall  Z_1 \forall x'[[\nu^c] \langle \not \ni \rangle [\nu^c] \langle \not \ni \rangle\langle \in \rangle \{Z_1\}  \subseteq  \{x'\}^c \Rightarrow \langle \nu \rangle  \{Z_1\}  \subseteq  \{x'\}^c]$ & Ackermann \\
iff \ & $\forall  Z_1[\langle \nu \rangle \{Z_1\} \subseteq[\nu^c] \langle \not \ni \rangle [\nu^c] \langle \not \ni \rangle\langle \in \rangle \{Z_1\} ]$ \\
iff \ & \mc{2}{l}{$\forall  Z_1 \forall x[ xR_\nu Z_1 \Rightarrow \forall Z_2 (x R_{\nu^c} Z_2 \Rightarrow \exists y (Z_2 R_{\not \ni} y \ \& \ \forall Z_3 (yR_{\nu^c} Z_3 \Rightarrow \exists w ( Z_3 R_{\not \ni} w \ \& \ wR_\in Z_1 ) )))]$} \\
\mc{3}{r}{Standard translation} \\
iff \ & \mc{2}{l}{$\forall  Z_1 \forall x[ x \in \nu(Z) \Rightarrow \forall Z_2 (Z_2 \not \in \nu(x) \Rightarrow \exists y (y \not \in Z_2 \ \& \ \forall Z_3 (Z_2 \not \in \nu(y) \Rightarrow \exists w( w \not \in Z_3 \ \& \ w \in Z_1 ))))]$} \\
\mc{3}{r}{Relations translation} \\
iff \ & \mc{2}{l}{$\forall  Z_1 \forall x[ x \in \nu(Z) \Rightarrow \forall Z_2 (Z_2 \not \in \nu(x) \Rightarrow \exists y (y \not \in Z_2 \ \& \ \forall Z_3 (Z_2 \not \in \nu(y) \Rightarrow Z_1\nsubseteq Z_3)))]$} \\
\mc{3}{r}{Relations translation} \\
iff \ & \mc{2}{l}{$\forall  Z_1 \forall x[ x \in \nu(Z) \Rightarrow(\forall Z_2(\forall y(\forall Z_3(Z_1\subseteq Z_3\Rightarrow Z_3\in\nu(y))\Rightarrow y\in Z_2)\Rightarrow Z_2\in\nu(x)))]$} \\
\mc{3}{r}{Contraposition} \\
iff \ &$\forall  Z_1 \forall x[ x \in \nu(Z) \Rightarrow(\forall Z_2(\forall y(Z_1\in\nu(y))\Rightarrow y\in Z_2)\Rightarrow Z_2\in\nu(x)))]$ & Monotonicity \\
iff \ &$\forall  Z_1 \forall x[ x \in \nu(Z) \Rightarrow\{ y\mid Z_1\in\nu(y)\}\in\nu(x)]$. & Monotonicity \\
\end{tabular}
\end{center}
}}

{\small{ 
\begin{flushleft}
\begin{tabular}{clr}
\mc{3}{l}{4.\ \, $\mathbb{H}\models \nabla \nabla  p\to \nabla p\ \rightsquigarrow\ \langle \nu\rangle [\ni] \langle \nu\rangle [\ni] p\subseteq[ \nu^c ] \langle \not \ni \rangle p$} \\
\hline
    \ &$ \langle \nu\rangle [\ni]\langle \nu\rangle [\ni] p\subseteq[ \nu^c ] \langle \not \ni \rangle p$ & \\ 
iff \ & $\forall  x\forall Z_1 \forall p[ \{x\} \subseteq \xdianu\dboxni\xdianu\dboxni p\ \&\ \ddianni p\subseteq\{Z_1\}^c \Rightarrow  \{x\}\subseteq \xboxnuc\{Z_1\}^c]$ & first approx. \\
iff \ & $\forall  x\forall Z_1 \forall p[ \{x\} \subseteq \xdianu\dboxni\xdianu\dboxni p\ \&\  p\subseteq [\notin]\{Z_1\}^c \Rightarrow  \{x\}\subseteq \xboxnuc\{Z_1\}^c]$ & Adjunction \\
iff \ & $\forall  x\forall Z_1[ \{x\} \subseteq \xdianu\dboxni\xdianu\dboxni  [\notin]\{Z_1\}^c \Rightarrow  \{x\}\subseteq \xboxnuc\{Z_1\}^c]$ & Ackermann \\
iff \ & \mc{2}{l}{$\forall  x\forall Z_1[(\exists Z_2(xR_\nu Z_2\ \&\ \forall y(Z_2 R_\ni y\Rightarrow\exists Z_3(y R_\nu Z_3\ \&\ \forall w(Z_3 R_\ni w\Rightarrow \lnot wR_{\not\in} Z_1)))))\Rightarrow \lnot x R_{\nu^c} Z_1]$} \\
\mc{3}{r}{Standard translation} \\
iff \ & \mc{2}{l}{$\forall  x\forall Z_1[((\exists Z_2\in \nu(x))(\forall y\in Z_2)(\exists Z_3\in\nu(y))(\forall w\in Z_3)(w\in Z_1))\Rightarrow Z_1\in\nu(x)]$} \\
\mc{3}{r}{Relation translation} \\
iff \ &$\forall  x\forall Z_1[((\exists Z_2\in \nu(x))(\forall y\in Z_2)(\exists Z_3\in\nu(y))(Z_3\subseteq Z_1))\Rightarrow Z_1\in\nu(x)]$ & \\
iff \ &$\forall  x\forall Z_1\forall Z_2[(Z_2\in \nu(x)\ \&\ (\forall y\in Z_2)(\exists Z_3\in\nu(y))(Z_3\subseteq Z_1))\Rightarrow  Z_1\in\nu(x)]$ & \\
iff \ &$\forall  x\forall Z_1\forall Z_2[(Z_2\in \nu(x)\ \&\ (\forall y\in Z_2)( Z_1\in\nu(y)))\Rightarrow  Z_1\in\nu(x)]$ & Monotonicity \\
\end{tabular}
\end{flushleft}
}}

{\small{
\begin{flushleft}
\begin{tabular}{clr}
\mc{3}{l}{5.\ \, $\mathbb{H}\models \neg \nabla \neg p\to \nabla \neg  \nabla \neg p \ \rightsquigarrow\ \neg[\nu^c] \langle \not \ni \rangle \neg p \subseteq  [\nu^c] \langle \not \ni \rangle \neg \langle \nu \rangle  [ \ni ] \neg p$} \\
\hline
    \ & $\neg[\nu^c] \langle \not \ni \rangle \neg p \subseteq [\nu^c] \langle \not \ni \rangle \neg \langle \nu \rangle  [ \ni ] \neg p$ & \\
iff \ & $\forall x\forall Z_1[[\nu^c] \langle \not \ni \rangle \neg \langle \nu \rangle  [ \ni ] \neg p\subseteq \{x\}^c\ \&\ \langle\not\ni\rangle \lnot p\subseteq\{Z_1\}^c \Rightarrow \neg[\nu^c] \{Z\}^c \subseteq\{x\}^c ]$ & first approx. \\
iff \ & $\forall x\forall Z_1[[\nu^c] \langle \not \ni \rangle \neg \langle \nu \rangle  [ \ni ] \neg p\subseteq \{x\}^c\ \&\ \lnot[\not\in]  \{Z_1\}^c\subseteq p \Rightarrow \neg[\nu^c] \{Z\}^c \subseteq\{x\}^c ]$ & Residuation \\
iff \ & $\forall x\forall Z_1[[\nu^c] \langle \not \ni \rangle \neg \langle \nu \rangle  [ \ni ] \neg \lnot[\not\in]  \{Z_1\}^c\subseteq \{x\}^c\Rightarrow \neg[\nu^c] \{Z\}^c \subseteq\{x\}^c ]$ & Ackermann \\
iff \ & $\forall Z_1[\neg[\nu^c] \{Z_1\}^c \subseteq[\nu^c] \langle \not \ni \rangle \neg \langle \nu \rangle  [ \ni ] \neg \lnot[\not\in]  \{Z_1\}^c ]$ & \\ 
iff \ & \mc{2}{l}{$\forall Z_1\forall x[xR_{\nu^c}Z_1\Rightarrow\forall Z_2(xR_{\nu^c}Z_2\Rightarrow\exists y(Z_2 R_{\not\ni} y\ \&\ \forall Z_3(yR_{\nu}Z_3\Rightarrow\exists w(Z_3 R_{\ni}w\ \&\ wR_{\notin}Z_1))))]$} \\
\mc{3}{r}{Standard translation} \\
iff \ & $\forall Z_1\forall x[Z_1\notin\nu(x)\Rightarrow(\forall Z_2\notin\nu(x))(\exists y\notin Z_2)(\forall Z_3\in\nu(y))(\exists w\in Z_3)(w\notin Z_1)]$ & \\
\mc{3}{r}{Relation translation} \\
iff \ & $\forall Z_1\forall x[Z_1\notin\nu(x)\Rightarrow(\forall Z_2\notin\nu(x))(\exists y\notin Z_2)(\forall Z_3\in\nu(y))(Z_3\nsubseteq Z_1)]$ & \\
iff \ & $\forall Z_1\forall x[Z_1\notin\nu(x)\Rightarrow\forall Z_2(((\forall y\notin Z_2)(\exists Z_3\in\nu(y))(Z_3\subseteq Z_1))\Rightarrow Z_2\in\nu(x))]$ & \\
\mc{3}{r}{Contraposition} \\
iff \ & $\forall Z_1\forall x[Z_1\notin\nu(x)\Rightarrow\forall Z_2((\forall y\notin Z_2) (Z_1\in\nu(y))\Rightarrow Z_2\in\nu(x))]$ & Monotonicity \\
iff \ & $\forall Z_1\forall x[Z_1\notin\nu(x)\Rightarrow\{y\mid Z_1\in\nu(y)\}^c\in\nu(x))]$ & Monotonicity \\
\end{tabular}
\end{flushleft}
}}

{\small{ 
\begin{flushleft}
\begin{tabular}{clr}
\mc{3}{l}{B.\ \, $\mathbb{H}\models p \to \nabla \neg  \nabla \neg p\ \rightsquigarrow\  p \subseteq [\nu^c] \langle \not \ni \rangle \neg \langle \nu \rangle  [ \ni ] \neg p$} \\
\hline
    \ &$p \subseteq [\nu^c] \langle \not \ni \rangle \neg \langle \nu \rangle  [ \ni ] \neg p$ & \ \\
iff \ & $\forall  x \forall p [  \{x\} \subseteq p \Rightarrow  \{x\} \subseteq [\nu^c] \langle \not \ni \rangle \neg \langle \nu \rangle  [ \ni ] \neg p]$ & \ first approx. \\
iff \ & $\forall  x [  \{x\} \subseteq [\nu^c] \langle \not \ni \rangle \neg \langle \nu \rangle  [ \ni ] \neg  \{x\}]$ & \ Ackermann \\
iff \ & $\forall  x [  \{x\} \subseteq [\nu^c] \langle \not \ni \rangle  [ \nu ]  \langle \ni \rangle  \{x\}]$ & \ \\
iff \ & $\forall  x [ \forall Z_1 (x R_{\nu^c} Y \Rightarrow \exists y (Y R_{\not \ni} x \ \& \ \forall Z_2 (y R_\nu Z_2 \Rightarrow Z_2 R_\ni x )))]$ & \ Standard translation \\
iff \ & $\forall  x [ \forall Z_1 (Z_1 \not \in \nu(x) \Rightarrow \exists y (x \not \in Z_1 \ \& \ \forall Z_2 (Z_2 \in \nu(y) \Rightarrow x \in Z_2 )))]$ & \ Relations translation \\
iff \ & $\forall  x [ \forall Z_1(\forall y(\forall Z_2(x\notin Z_2\Rightarrow Z_2\notin\nu(y))\Rightarrow y \in Z_1) \Rightarrow Z_1 \in\nu(x) )]$ & \ Contrapositive \\
iff \ & $\forall  x [ \forall Z_1(\forall y(\{x\}^c\notin\nu(y_1))\Rightarrow y \in Z_1) \Rightarrow Z_1 \in\nu(x) )]$ & \ Monotonicity \\
iff \ & $\forall  x [ \{ y \mid \{x\}^c\notin\nu(y)\}\in\nu(x) )]$ & \ Monotonicity \\
iff \ & $\forall x \forall X[ x \in X \Rightarrow  \{ y\mid X^c \notin \nu(y) \} \in \nu(x)]$ & \ Monotonicity \\
\end{tabular}
\end{flushleft}
}}

{\small{
\begin{flushleft}
\begin{tabular}{clr}
\mc{3}{l}{D.\ \, $\mathbb{H}\models\nabla p \to \neg  \nabla \neg p\ \rightsquigarrow\  \langle \nu \rangle  [ \ni ] p \subseteq \neg\langle \nu \rangle  [ \ni ] \neg p$} \\
\hline
   \ &$\langle \nu \rangle   [ \ni ] p \subseteq\neg\langle \nu \rangle  [ \ni ] \neg p$ \\ 
iff \ & $ \forall Z \forall Z' [ \{Z \} \subseteq [\ni ] p \ \&\  Z' \subseteq \ [\ni] \neg p  \Rightarrow \langle \nu \rangle \{Z \}\subseteq\neg \langle \nu \rangle Z']$ & first approx. \\
iff \ & $ \forall Z \forall Z' [ \langle \in \rangle \{Z \} \subseteq  p \ \&\  \{Z'\} \subseteq \ [\ni] \neg p  \Rightarrow \langle \nu \rangle \{Z \}\subseteq\neg \langle \nu \rangle \{Z'\}]$ & Residuation \\
iff \ & $ \forall Z \forall Z' [   \{Z'\} \subseteq \ [\ni] \neg \langle \in \rangle \{Z \}  \Rightarrow \langle \nu \rangle \{Z \}\subseteq\neg \langle \nu \rangle \{Z'\}]$ & $(\star)$ Ackermann \\
iff \ & $\forall Z [ \langle \nu \rangle \{Z\} \subseteq \neg\langle \nu \rangle  [ \ni ] \neg \langle \in \rangle \{Z\}]$ & \\
iff \ & $\forall Z [ \langle \nu \rangle \{Z\} \subseteq [ \nu]  \langle \ni \rangle \langle \in \rangle \{Z\}]$ & \\
iff \ & $\forall Z\forall x [ xR_\nu Z \Rightarrow  \forall Y(xR_\nu Y\Rightarrow \exists w(Y R_\ni w\ \&\ w R_\in Z))]$ & Standard Translation \\
iff \ & $\forall Z\forall x [ Z\in \nu(x) \Rightarrow  \forall Y(Y\in\nu(x)\Rightarrow \exists w( w\in Y \ \&\ w \in Z))]$ & Relation translation \\
iff \ & $\forall Z\forall x [ Z\in \nu(x) \Rightarrow  \forall Y(Y\in\nu(x)\Rightarrow Y\nsubseteq Z^c)]$ & \\
iff \ & $\forall Z\forall x [ Z\in \nu(x) \Rightarrow  \forall Y(Y\subseteq Z^c\Rightarrow Y\notin\nu(x))]$ & Contrapositive \\
iff \ & $\forall Z\forall x \forall Y[ Z\in \nu(x) \Rightarrow  Z^c\notin\nu(x)]$ & Monotonicity \\
\end{tabular}
\end{flushleft}
}}

{\small{
\begin{flushleft}
\begin{tabular}{cl@{}r}
\mc{3}{l}{CS.\ \, $\mathbb{H}\models(p \land q) \to (p\succ q)\ \rightsquigarrow\ (p\land q) \subseteq ([\ni] p\cap[\not\ni\rangle p){\rhd}  q$} \\
\hline
      & $(p\land q) \subseteq ([\ni] p\dand[\not\ni\rangle p){\rhd}  q $ \\
iff \ & \mc{2}{l}{$\forall  x \forall  Z\forall x' \forall p q[  \{x\} \subseteq p \land q \ \& \  \{Z\}  \subseteq [ \ni ] p \dand [\not \ni \rangle p \ \& \ q  \subseteq \{x'\}^c   \Rightarrow   \{x\} \subseteq  \{Z\}    {\rhd}    \{x'\}^c]$} \\
\mc{3}{r}{first. approx.} \\
iff \ & \mc{2}{l}{$\forall  x\forall  Z  \forall x\forall p \forall q[  \{x\} \subseteq p \ \& \  \{x\} \subseteq q \ \& \  \{Z\}  \subseteq [ \ni ] p \ \& \  \{Z\}  \subseteq [\not \ni \rangle p \ \& \ q  \subseteq  \{x'\}^c   \Rightarrow   \{x\} \subseteq  \{Z\}    {\rhd}   \{x'\}^c]$} \\
\mc{3}{r}{Splitting rule} \\
iff \ & \mc{2}{l}{$\forall  x \forall  Z\forall x'\forall p \forall q[  \{x\} \subseteq p \ \& \  \{x\} \subseteq q \ \& \  \{Z\}  \subseteq [ \ni ] p \ \& \ p \subseteq [ \not  \in \rangle  \{Z\}   \ \& \ q  \subseteq  \{x'\}^c   \Rightarrow   \{x\} \subseteq  \{Z\}    {\rhd}  \{x'\}^c]$} \\
\mc{3}{r}{Residuation} \\
iff \ & \mc{2}{l}{$\forall  x \forall  Z  \forall x'  \forall q[  \{x\} \subseteq [ \not  \in \rangle  \{Z\}  \ \& \  \{x\} \subseteq q \ \& \  \{Z\}  \subseteq [ \ni ] [ \not  \in \rangle  \{Z\}    \ \& \ q  \subseteq  \{x'\}^c   \Rightarrow   \{x\} \subseteq  \{Z\}    {\rhd}   \{x'\}^c]$} \\
\mc{3}{r}{Ackermann} \\
iff \ & \mc{2}{l}{$\forall  x\forall  Z \forall x' [  \{x\} \subseteq [ \not  \in \rangle  \{Z\}   \ \& \  \{Z\}  \subseteq [ \ni ] [ \not  \in \rangle  \{Z\}    \ \& \  \{x\}  \subseteq  \{x'\}^c   \Rightarrow   \{x\} \subseteq  \{Z\}    {\rhd}  \{x'\}^c]$} \\
\mc{3}{r}{$(\star)$ Ackermann} \\
iff \ & \mc{2}{l}{$\forall  x \forall  Z  [  \{x\} \subseteq [ \not  \in \rangle  \{Z\}   \ \& \  \{Z\}  \subseteq [ \ni ] [ \not  \in \rangle  \{Z\}    \Rightarrow   \{x\} \subseteq  \{Z\}    {\rhd}    \{x\}]$} \\
iff \ & $\forall  x \forall  Z  [  \lnot x R_{\not\in} Z   \ \& \  \forall y(Z R_\ni y\Rightarrow \lnot y R_{\not \in} Z)   \Rightarrow   \forall y( T_f(x,Z,y) \Rightarrow y =x)]$ & Standard translation \\
iff \ & $\forall  x \forall  Z  [   x\in Z   \ \& \  \forall y(y\in Z\Rightarrow  Z\in y )   \Rightarrow   \forall y( y\in f(x,Z) \Rightarrow y=x)]$ & Relation interpretation \\
iff \ & $\forall  x \forall  Z  [   x\in Z \   \Rightarrow \   \forall y( y\in f(x,Z) \Rightarrow y=x)]$ & \\
iff \ & $\forall x\forall Z[x \in Z \Rightarrow f(x,Z) \subseteq \{x\}]$ & \\
\end{tabular}
\end{flushleft}
}}

{\small{
\begin{flushleft}
	\begin{tabular}{clr}
\mc{3}{l}{ID.\ \, $\mathbb{H}\models p\succ p \ \rightsquigarrow\ ([\ni] p\dand[\not\ni \rangle p){\rhd} p$} \\
\hline
		&$\top\subseteq ([\ni] p\dand[\not\ni\rangle p){\rhd} p$ &\\  
		iff & $\forall  Z Z'\forall x' p [( \{Z\} \subseteq [\ni] p \ \&\ \{Z'\}\subseteq [\not\ni\rangle p \ \&\ p\subseteq  \{x'\}^c)\Rightarrow  \top \subseteq( \{Z\} \dand\{Z'\}){\rhd}   \{x'\}^c]$ & first approx. \\ 
		iff & $\forall  Z Z'\forall x' p [(\langle\in\rangle \{Z\} \subseteq p \ \&\ \{Z'\}\subseteq [\not\ni \rangle p \ \&\ p\subseteq \{x'\}^c)\Rightarrow  \top\subseteq( \{Z\} \dand\{Z'\}){\rhd}   \{x'\}^c]$ & Adjunction \\
		iff & $\forall  Z \forall Z'\forall  x' [(\{Z'\}\subseteq [\not\ni \rangle \langle\in\rangle \{Z\}  \ \&\ \langle\in\rangle \{Z\} \subseteq  \{x'\}^c)\Rightarrow  \top\subseteq( \{Z\} \dand\{Z'\}){\rhd}  \{x'\}^c$ & Ackermann \\
		iff & $\forall  Z \forall Z'[\{Z'\}\subseteq [\not\ni \rangle \langle\in\rangle \{Z\}  \ \Rightarrow \forall  x' [ \langle\in\rangle \{Z\} \subseteq  \{x'\}^c \Rightarrow  \top \subseteq( \{Z\} \dand\{Z'\}){\rhd}  \{x'\}^c]]$ & Currying \\
		iff & $\forall  Z\forall Z'[\{Z'\}\subseteq [\not\ni \rangle \langle\in\rangle \{Z\}  \ \Rightarrow  \top \subseteq( \{Z\} \dand\{Z'\}){\rhd}  \langle\in\rangle \{Z\} ]$ & $(\star)$ Ackermann \\
		iff & $\forall x\forall  Z\forall Z'[ \forall w(Z'R_{\not\ni}w\Rightarrow \lnot wR_\in Z) \Rightarrow \forall y (T_f (x, Z, y ) \ \& \ Z=Z'   \Rightarrow y \in Z) ]$ & \\
\mc{3}{r}{Standard Translation} \\
		iff & $\forall x\forall  Z\forall Z'\forall y[ \forall w(Z'R_{\not\ni}w\Rightarrow \lnot wR_\in Z) \ \&\   (T_f (x, Z, y ) \ \& \ Z=Z'   \Rightarrow y \in Z) ]$ & \\
		iff & $\forall x\forall  Z\forall Z'\forall y[ \forall w(w\notin Z'\Rightarrow  w\notin Z) \ \&\   (y\in f (x, Z) \ \& \ Z=Z'   \Rightarrow y \in Z) ]$ & \\
\mc{3}{r}{Relation interpretation} \\
		iff & $\forall x\forall  Z\forall Z'\forall y[ Z\subseteq Z' \ \&\   (y\in f (x, Z) \ \& \ Z=Z'   \Rightarrow y \in Z) ]$ & \\
		iff & $\forall x\forall  Z\forall y[(y\in f (x, Z)\Rightarrow y \in Z) ]$ & \\
		iff  & $\forall x\forall Z[f(x,Z)\subseteq Z]$ & \\		
	\end{tabular}
\end{flushleft}
}}

{\small{
\begin{flushleft}
\begin{tabular}{clr}
\mc{3}{l}{T.\ \, $\mathbb{H}\models\nabla p\to p\ \rightsquigarrow\ \langle \nu\rangle [\ni] p\subseteq p$} \\
\hline
    \ & $\langle \nu\rangle [\ni] p\subseteq p$ & \ \\
iff \ & $\forall  x\forall Z\forall p [ p \subseteq \{x\}^c\ \&\ \{Z\}\subseteq\dboxni p \Rightarrow \langle \nu\rangle\{Z\} \subseteq  \{x\}^c]$ & \ first approx. \\
iff \ & $\forall  x\forall Z\forall p [ p \subseteq \{x\}^c\ \&\ \langle\in\rangle\{Z\}\subseteq p \Rightarrow \langle \nu\rangle\{Z\} \subseteq  \{x\}^c]$ & \ Adjunction \\
iff \ & $\forall  x\forall Z[\langle\in\rangle\{Z\}\subseteq \{x\}^c \Rightarrow \langle \nu\rangle\{Z\} \subseteq  \{x\}^c]$ & \ $(\star)$ Ackermann \\
iff \ & $\forall  Z  [\langle \nu\rangle \{Z\} \subseteq \langle \ni \rangle \{Z\} ]$ & \ inverse approx. \\
iff \ & $\forall x\forall Z[x  R_\nu Z\Rightarrow xR_\ni Z]$ & \ Standard translation\\
iff \ & $\forall x\forall Z[Z\in \nu(x)\Rightarrow x \in Z]$. & \ Relation translation\\
\end{tabular}
\end{flushleft}
}}

{\small{
\begin{flushleft}
\begin{tabular}{clr}
\mc{3}{l}{CEM.\ \, $\mathbb{H}\models (p \succ q) \lor (p\succ \neg q) \ \rightsquigarrow\ (([\ni] p\dand[\not\ni\rangle p){\rhd}  q)\lor (([\ni] p\dand[\not\ni\rangle p){\rhd}  \neg q)$} \\
\hline
    \ &$\top \subseteq (([\ni] p\dand[\not\ni\rangle p){\rhd}  q)\lor (([\ni] p\dand[\not\ni\rangle p){\rhd}  \neg q)$ &\\
iff \ & $\forall p \forall q\forall X \forall Y \forall x \forall y   (\{X\} \subseteq [\ni] p\dand[\not\ni\rangle p \ \& \ \{Y\} \subseteq [\ni] p\dand[\not\ni\rangle p\ \& \ q \subseteq \{x\}^c \ \& \ \{y\} \subseteq q$ & \\
      & \mc{2}{r}{$\Rightarrow \top \subseteq (\{X\}{\rhd}\{x\}^c)\lor (\{Y\}{\rhd}\neg \{y\}) $ \ \ \ \ first approx.} \\
iff \ & $\forall p \forall q\forall X \forall Y \forall x \forall y   (\{X\} \subseteq [\ni] p \ \& \ \{X\} \subseteq [\not\ni\rangle p \ \& \ \{Y\} \subseteq [\ni] p\ \& \ \{Y\} \subseteq [\not\ni\rangle p\ \& \ q \subseteq \{x\}^c \ \& \ \{y\} \subseteq q $ & \\
      & \mc{2}{r}{$\Rightarrow \top \subseteq (\{X\}{\rhd}\{x\}^c)\lor (\{Y\}{\rhd}\neg \{y\}) $  $(\star)$ \ \ \ \ Splitting} \\ 
iff \ & $\forall p \forall q\forall X \forall Y \forall x \forall y   (\{X\} \subseteq [\ni] p \ \& \ p \subseteq [\not\in\rangle \{X\} \ \& \ \{Y\} \subseteq [\ni] p\ \& \ p \subseteq [\not\in\rangle \{Y\} \ \& \ q \subseteq \{x\}^c \ \& \ \{y\} \subseteq q $ & \\
      & \mc{2}{r}{$\Rightarrow \top \subseteq (\{X\}{\rhd}\{x\}^c)\lor (\{Y\}{\rhd}\neg \{y\}) $ \ \ \ \ Residuation} \\ 
iff \ & $\forall X \forall Y \forall x \forall y   (\{X\} \lor \{Y\} \subseteq [\ni] ([\not\in\rangle \{X\} \land [\not\in\rangle \{Y\})   \ \& \ \{y\} \subseteq \{x\}^c $ & \\
     & \mc{2}{r}{$\Rightarrow \top \subseteq (\{X\}{\rhd}\{x\}^c)\lor (\{Y\}{\rhd}\neg \{y\}) $ \ \ \ \ Ackermann} \\
iff \ & \mc{2}{l}{$\forall X \forall Y \forall x    (\{X\} \lor \{Y\} \subseteq [\ni] ([\not\in\rangle \{X\} \land [\not\in\rangle \{Y\})   \Rightarrow \forall y( \{y\} \subseteq \{x\}^c \Rightarrow \top \subseteq (\{X\}{\rhd}\{x\}^c)\lor (\{Y\}{\rhd}\neg \{y\})) $} \\
\mc{3}{r}{Currying} \\ 
iff \ & $\forall X \forall Y \forall x    (\{X\} \lor \{Y\} \subseteq [\ni] ([\not\in\rangle \{X\} \land [\not\in\rangle \{Y\})   \Rightarrow \top \subseteq (\{X\}{\rhd}\{x\}^c)\lor (\{Y\}{\rhd}\neg \{x\}^c)) $  &  \\
iff \ & \mc{2}{l}{$\forall X \forall Y \forall x[(\forall y(XR_\ni y\ \text{ or }\ YR_\ni y)\Rightarrow\lnot yR_{\not\in} X\ \&\ \lnot yR_{\not\in} Y)\ $} \\
      & \mc{2}{r}{$\Rightarrow \forall y(\lnot T_f(y,X,x)\ \text{ or }\ (\forall z( T_f(y,Y,z)\Rightarrow z=x)))]$ \ \ \ \ Standard translation} \\
iff \ & \mc{2}{l}{$\forall X \forall Y \forall x[(\forall y(y\in X \text{ or }\ y\in Y)\Rightarrow y\in X\ \&\ y\in Y)\ $} \\
      & \mc{2}{r}{$\Rightarrow \forall y(x\notin f(y,X)\ \text{ or }\ (\forall z( z\in f(y,Y)\Rightarrow z=x)))]$ \ \ \ \ Relation interpretation} \\
iff \ & $\forall X \forall Y \forall x[(X\cup Y\subseteq  X\cap Y)\ \Rightarrow \forall y(x\notin f(y,X)\ \text{ or }\ (\forall z( z\in f(y,Y)\Rightarrow z=x)))]$  & \\
iff \ & $\forall X \forall Y \forall x[X=Y \Rightarrow \forall y(x\notin f(y,X)\ \text{ or }\ (\forall z( z\in f(y,Y)\Rightarrow z=x)))]$  & \\	
iff \ & $\forall X \forall x\forall y[(x\notin f(y,X)\ \text{ or }\ (\forall z( z\in f(y,X)\Rightarrow z=x)))]$  & \\
iff \ & $\forall X \forall x\forall y[(x\in f(y,X)\ \Rightarrow\ f(y,X)=\{x\})]$  & \\
iff \ & $\forall X \forall y[|f(y,X)|\leq 1]$ & \\		
\end{tabular}
\end{flushleft}
}}

{\small{
\begin{flushleft}
\begin{tabular}{clr}
\mc{3}{l}{CN.\ \, $\mathbb{H}\models(p > q) \lor (q > p ) \ \rightsquigarrow\ (\dboxni p \wedge\drhdnni p)\mtra q )\lor  ( (\dboxni q \wedge\drhdnni q)\mtra p$} \\
\hline
    \ & $\top \subseteq ( (\dboxni p \wedge\drhdnni p)\mtra q )\lor  ( (\dboxni q \wedge\drhdnni q)\mtra p)$&\\
iff \ & $\forall p \forall q \forall x \forall y \forall X \forall Y  (\{X\} \subseteq [\ni] p\dand[\not\ni\rangle p \ \& \ \{Y\} \subseteq [\ni] q\dand[\not\ni\rangle q \ \& \ q \subseteq \{x\}^c \ \& \   p \subseteq \{y\}^c$ & \\
      & \mc{2}{r}{$\Rightarrow \top \subseteq (\{X\}{\rhd}\{x\}^c)\lor (\{Y\}{\rhd}\{y\}^c) $ \ \ \ \ first approx.} \\
      iff \ & $\forall p \forall q \forall x \forall y \forall X \forall Y   (\{X\} \subseteq [\ni] p \ \& \ \{X\} \subseteq [\not\ni\rangle p \ \& \ \{Y\} \subseteq [\ni] q  \ \& \ \{Y\} \subseteq [\not\ni\rangle q \ \& \ q \subseteq \{x\}^c \ \& \   p \subseteq \{y\}^c$ & \\
      & \mc{2}{r}{$\Rightarrow \top \subseteq (\{X\}{\rhd}\{x\}^c)\lor (\{Y\}{\rhd}\{y\}^c) $ \ \ \ \  Splitting} \\
            iff \ & $\forall p \forall q \forall x \forall y \forall X \forall Y  (\{X\} \subseteq [\ni] p \ \& \ p \subseteq [\not\in\rangle\{X\} \ \& \ \{Y\} \subseteq [\ni] q  \ \& \ q \subseteq [\not\in\rangle\{Y\} \ \& \ q \subseteq \{x\}^c \ \& \   p \subseteq \{y\}^c$ & \\
      & \mc{2}{r}{$\Rightarrow \top \subseteq (\{X\}{\rhd}\{x\}^c)\lor (\{Y\}{\rhd}\{y\}^c) $ \ \ \ \  Residuation} \\
         iff \ & $ \forall x \forall y\forall X \forall Y   (\{X\} \subseteq [\ni] ( [\not\in\rangle\{X\}  \cap \{y\}^c)  \& \ \{Y\} \subseteq [\ni]( [\not\in\rangle\{Y\} \cap \{x\}^c )\ \& \  $ & \\
      & \mc{2}{r}{$\Rightarrow \top \subseteq (\{X\}{\rhd}\{x\}^c)\lor (\{Y\}{\rhd}\{y\}^c) $ \ \ \ \  Ackermann} \\
              iff \ & $ \forall x \forall y \forall X \forall Y (\{X\} \subseteq (R_\ni^{-1}[R_{\not\ni}[\{X\}]  \cap \{y\}^c  ])^c     \& \ \{Y\} \subseteq (R_\ni^{-1}[R_{\not\ni}[\{Y\}]   \cap \{x\}^c ])^c   \ \& \  $ & \\
      & \mc{2}{r}{$\Rightarrow \top \subseteq ((T_f^{(0)}[\{X\}, \{x\}])^c)\lor ((T_f^{(0)}[\{Y\}, \{y\}])^c) $ \ \ \ \  Standard translation} \\
           iff \ & $\forall x \forall y \forall X \forall Y  (\{X\} \subseteq (R_\ni^{-1}[X^c \cap \{y\}^c  ])^c     \& \ \{Y\} \subseteq (R_\ni^{-1}[Y^c\cap \{x\}^c ])^c   \ \& \  $ & \\
      & \mc{2}{r}{$\Rightarrow \top \subseteq ((T_f^{(0)}[\{X\}, \{x\}])^c)\lor ((T_f^{(0)}[\{Y\}, \{y\}])^c) $ \ \ \ \ } \\ 
      iff \ & $\forall x \forall y \forall X \forall Y  ( X\subseteq X\cup\{y\}   \& \ Y \subseteq Y\cup \{x\}\Rightarrow \top \subseteq \{z\mid x \in f(z,X)\}^c \cup \{z\mid y \in f(z,Y)\}^c )$  & \\
            iff \ & $\forall z  \forall x \forall y \forall X \forall Y  [  x \not \in f(z,X) \text{~or~} y \not \in f(z,Y)] $  & \\
\end{tabular}
\end{flushleft}
}}

{\small{
\begin{flushleft}
\begin{tabular}{clr}
\mc{3}{l}{T.\ \, $\mathbb{H}\models(\bot >\neg p)\to p \ \rightsquigarrow\ ((\dboxni \bot \wedge\drhdnni \bot)\mtra \neg p) \subseteq p $} \\
\hline
    \ & $((\dboxni \bot \wedge\drhdnni \bot)\mtra \neg p) \subseteq p $ & \ \\
iff \ & $((\dboxni \bot \wedge\drhdnni \bot)\mtra \neg \bot) \subseteq \bot $  &  Variable elimination\\
iff \ & $(\{\emptyset\})\mtra \neg \bot) \subseteq \bot $  &  \\
iff \ & $\{z \mid \forall X \forall x(T_f(z,X,x)\to X \in \{\emptyset\} \ \& \ x \in W)\} \subseteq \bot $  &  \\
iff \ & $\forall z \exists x[ x \in f(z,\emptyset) ]$  &  \\

\end{tabular}
\end{flushleft}
}}

\end{document}